\documentclass[reqno]{amsart}

\usepackage{amsmath,amsfonts,amssymb,amsthm}
\usepackage{xcolor,mhsetup,bm}
\usepackage[extra,safe]{tipa}
\usepackage{mathtools}
\usepackage{hyperref}
\usepackage{enumerate}
\usepackage{float}
\usepackage{tikz}
\usepackage{textcomp}
\usetikzlibrary{shapes}
\usepackage{enumitem}   
\usetikzlibrary{positioning}
\usetikzlibrary{shapes.geometric, arrows}
\usetikzlibrary{calc,decorations.pathmorphing,shapes,arrows,decorations.pathreplacing}
\usepackage{textcomp}
\usepackage{soul}

\usepackage{marginnote}%this avoids floating marginal notes
\usepackage[applemac]{inputenc}
\usepackage{mathrsfs}
\usepackage{mathabx}

\newtheorem{theorem}{Theorem}
\newtheorem{definition}[theorem]{Definition}
\newtheorem{proposition}[theorem]{Proposition}

\newtheorem{lemma}[theorem]{Lemma}

\newtheorem{corollary}[theorem]{Corollary}
\newtheorem{remark}[theorem]{Remark}

\def\my_c{c_\infty}

\newcommand{\mynewtheorem}[2]{
        \newaliascnt{#1}{dummy}
        \newtheorem{#1}[#1]{#2}
        \aliascntresetthe{#1}
        \expandafter\def\csname #1autorefname\endcsname{#2}
}

\newcommand{\be}{\begin{equation}}
        \newcommand{\ee}{\end{equation}}
\newcommand{\bde}{\begin{displaymath}}
        \newcommand{\ede}{\end{displaymath}}
\newcommand{\beq}{\begin{eqnarray*}}
        \newcommand{\eeq}{\end{eqnarray*}}
\newcommand{\beqa}{\begin{eqnarray}}
        \newcommand{\eeqa}{\end{eqnarray}}
\newcommand{\bel }{\left\{\begin{array}{ll}}
        \newcommand{\eel}{\cr \end{array} \right.}

\newcommand{\dcb}{\begin{array}{lll}}
        \newcommand{\dce}{\end{array}}
\newcommand{\ebe}{\begin{enumerate}\setlength{\baselineskip}{13pt}\setlength{\parskip}{0pt}}
        \newcommand{\dbe}{\end{enumerate}}

\newcommand{\leftrightharpoonup}{%
        \mathrel{\mathpalette\lrhup\relax}%
}
\newcommand{\lrhup}[2]{%
        \ooalign{$#1\leftharpoonup$\cr$#1\rightharpoonup$\cr}%
}
\newcommand\rharp[1]{\mathstrut\mkern2.5mu#1\mkern-11mu\raise1.2ex%
        \hbox{$\scriptscriptstyle\rightharpoonup$}}
\newcommand\lharp[1]{\mathstrut\mkern2.5mu#1\mkern-11mu\raise1.2ex%
        \hbox{$\scriptscriptstyle\leftharpoonup$}}
\newcommand\lrharp[1]{\mathstrut\mkern2.5mu#1\mkern-11mu\raise1.2ex%
        \hbox{$\scriptscriptstyle\leftrightharpoonup$}}

\def\0{{\mathbf{0}}}

\begin{document}
        
        \title[The derivative of a one dimensional
        killed diffusion semigroup]{A probabilistic representation of the derivative of a one dimensional
                killed diffusion semigroup}
        \author{Dan Crisan and Arturo Kohatsu-Higa }
        \address{\noindent Dan Crisan:
                Department of Mathematics, Imperial College London, 180 Queen's Gate, London
                SW7 2AZ, UK\newline Arturo Kohatsu-Higa:
                Department of Mathematical Sciences, Ritsumeikan University 1-1-1
                Nojihigashi, Kusatsu, Shiga, 525-8577, Japan.}
        \thanks{The research of the first author was partially supported by a Royal Society grant.  The research of the 
        	author was supported by KAKENHI grant 20K03666.
        }
        \begin{abstract}
                We provide a probabilistic representation for the derivative of the
                semigroup corresponding to a diffusion process killed at the boundary of a
                half interval. In particular, we show that the derivative of the semi-group
                can be expressed as the expected value of a functional of a reflected
                diffusion process. Furthermore, as an application, we obtain a Bismut-Elworthy-Li formula which is also valid at the boundary.\\[2mm]
        \end{abstract}
        \maketitle

        	\section{Introduction and motivation}
        	
        	\label{sec:intro} Let $(\Omega ,\mathcal{F},\mathbb{P})$ be a probability
        	space on which we have defined a $1-$dimensional Brownian motion $W$. We
        	denote by $X$ the solution of a $1$-dimensional stochastic differential
        	equation 
        	\begin{equation*}
        		dX^x_{t}=b\left( X^x_{t}\right) dt+\sigma (X^x_{t})dW_{t},  \label{x}
        	\end{equation*}%
        	with initial condition $X_{0}=x>L$ and   {bounded} coefficients  $b$ and $\sigma
        	\geq c>0$ for some constant $c$. We assume that the coefficients are twice differentiable with bounded derivatives. A basic result about the derivatives of stochastic flows (see e.g. \cite{Kunita}) states that 
        	\begin{align*}
        		\partial_xX_t^x=\exp\left(\int_0^t\sigma'(X_s^x)dW_s+\int_0^t(b'-\frac{(\sigma')^2}{2})(X_s^x)ds\right).
        	\end{align*}
        	In particular, this result implies that the derivative of the diffusion semigroup can be rewritten using $ \partial_xX_t^x $ as follows:
        	\begin{align}
        		\partial_x\mathbb{E}\left[f(X_t^x)\right]=\mathbb{E}\left[f'(X_t^x)\partial_xX_t^x\right],
        		\label{eq:clas}
        	\end{align}
        	where $ f\in C^1_b(\mathbb{R}) $.  In this article, we are interested in obtaining a probabilistic representation for the derivative of the killed diffusion semigroup. More precisely, let $\tau $ be the first exit time from the
        	interval $D:=(L,\infty )$, i.e., 
        	\begin{equation*}
        		\tau :=\inf \left\{ t\geq 0,~X^x_{t}=L\right\} .
        	\end{equation*}%
        	Let also $f$ be a smooth function\footnote{%
        		The precise degree of smoothness of the function $f$ is explained in the
        		statement of the Theorem \ref{th:main}. 
        		In the one-dimensional case, an alternative to $\left( L,\infty \right) $ is 
        		$(-\infty,R) $. This case is treated with the same techniques as those
        		developed for the case $\left( L,\infty \right) $. } $f:[L,\infty )\rightarrow \mathbb{R}
        	$ which satisfies that $f(L)=0$. In this article, we introduce a
        	probabilistic representation of the derivative $\partial _{x}P_{t}f(x)$, where $%
        	P_{t}$ is the semigroup associated to the process $X$ killed when it reaches
        	the boundary $\partial D,$ i.e.,%
        	\begin{equation}
        		P_{t}f\left( x\right) :=\mathbb{E}\left[ f\left( X^x_{t}\right) 1_{(
        			t<\tau )}\right] .  \label{ptf}
        	\end{equation}
        	To explain the results in this paper, we consider first the case when $b=0$.
        	We will show below that 
        	\begin{equation}
        		\partial _{x}P_{T}f(x)=\mathbb{E}\left[ f^{\prime }(Y_{T})\xi _{T}\right] .
        		\label{eq:-1}
        	\end{equation}
        	In \eqref{eq:-1}, $Y$ is the process reflected at the boundary $L$ of the
        	domain $[L,\infty )$ and $\xi$ is a solution of a stochastic linear
        	equation. More precisely, 
        	\begin{align}
        		Y_{t}=& x+\int_{0}^{t}\sigma (Y_{s})dW_{s}+B_{t}\geq L,  \label{y1} \\
        		B_t=&\int_0^t1_{(Y_s=L)}d|B|_s \notag \\\xi _{t}=& \exp\left(\int_{0}^{t}\sigma ^{\prime }(Y_{s})dW_{s}-\frac 12\int_{0}^{t}\sigma ^{\prime }(Y_{s})^2d{s}\right).  \label{xi1}
        	\end{align}%
        	%
        	%
        	%
        	%
        	%
        	%
        	%
        	%
        	%
        	%Here $B $ is the local time term which is the minimal positive increasing
        	%process so that it keeps the process $Y$ above the threshold $ L $. 
        	The pair $(Y,B)$ is uniquely identified by the equations \eqref{y1}+%
        	\eqref{xi1}, once we impose the constraint that $Y_{t}\geq L$ for any $t\geq
        	0$ and that $ B $ is a non-decreasing process with $ B_0=0 $ satisfying  $\int_{0}^{t}1_{(L,\infty )}\left( Y_{s}\right) dB_{s}=0$ for
        	any $t\geq 0$, see Theorem 1.2.1 in \cite{ap}. Note that in comparison with the classical case \eqref{eq:clas}, we have instead of $ (X_t^x,\partial_xX_t^x) $, the pair $ (Y_t,\xi_t) $.
        	
        	The above representation may be somewhat expected if one recalls the classical 
        	result connecting the densities of
        	killed and reflected Brownian motion. That is, if we denote the density of
        	killed Brownian motion started at $x $ by $q^-_t(x,y) $ and similarly by $%
        	q^+_t(x,y) $ the density of reflected Brownian motion, one has (see also Section \ref{sec:sc})
        	\begin{align}
        		\label{eq:dens}
        		\partial_xq^-_t(x,y)=\partial_yq^+_t(x,y).
        	\end{align}
        	
        	The above equality uses the symmetry properties of the Gaussian transition
        	kernel which arise due to the reflection principle. These are lost once one introduces a drift or a non-constant coefficients. As a result, in the case with
        	non-zero drift, the corresponding probabilistic representation is more involved 
        	due to the interaction between the drift and the boundary of the domain. In this case, we have for $ a(x):=\sigma^2(x) $
        	\begin{align}
        		\partial_xP_Tf(x)=&\mathbb{E}[f^{\prime }(Y_{T})\Psi_{T}]
        		\label{eq:main} \\
        		Y_{t}=& x+\int_{0}^{t}b(Y_{s})ds+\int_{0}^{t}\sigma (Y_{s})dW_{s}+B _{t}\geq L,
        		\label{eq:rep} \\
        		B_t=&\int_0^t1_{(Y_s=L)}d|B|_s   \\
        		\Psi_{t}=&\exp\left(
        		\int_{0}^{t}\left( \left(b^{\prime
        		}-\frac {(\sigma ^{\prime })^2}2\right)(Y_{s})ds+\sigma ^{\prime }(Y_{s})dW_{s}\right)+{\frac{2b(L)}{a(L)}}B _{t}\right ) .  \label{def:psi}
        	\end{align}
        	The interpretation of the last term  in \eqref{def:psi} is important. In fact, if $ b(L)>0 $ increases in value then $ \Psi_t $ will have a larger value  which implies that the killing will happen less frequently.
        	
        	We will prove the general representation formula (\ref{eq:main}) in two steps. 
        	\begin{itemize}
        		\item  First, we show that (see Theorem \ref{th:main})  for $ \mathcal{E}_t:=e^{-\frac{b}{a}(L)B_t}\Psi _t$, we have for \\ $\tau =\inf \left\{ t\geq 0,~Y_{t}=L\right\}$
        		and $\rho_T:=\sup\{s<T: Y_s=L\} $:
        		
        		\begin{align}  \label{eq:0}
        			\partial_xP_Tf(x)=&\mathbb{E}[f^{\prime }(Y_{T}){\mathcal{E}}_{T}]+{\frac{b}{
        					a}}\left( L\right) \mathbb{E}\left[ f(Y_{T}){\mathcal{E}}_{\rho
        				_{T}}1_{(\tau \leq T)}\right] .
        			%\\
        			%_{t}=& 1+\int_{0}^{t}\mathcal{E}_{s}\left( b^{\prime
        			%}(Y_{s})ds+\sigma ^{\prime }(Y_{s})dW_{s}+{\frac{b}{\sigma ^{2}}}\left(
        			%L\right) dB _{s}\right) .  \label{e}
        		\end{align}
        		The additional term that appears in the representation formula \eqref{eq:0}
        		when compared with \eqref{eq:-1} is the effect that the drift has when the
        		diffusion hits the boundary. Heuristically, the effect of the drift only
        		lasts until the last time the process touches the boundary $\rho_T $ \emph{%
        			before} time $T$ (one can assume without loss of generality that $Y_T>L $).

        		\item    Second, we show that equation (\ref{eq:0}) is equivalent to (\ref%
        		{eq:main}) (see Theorem \ref{th:36}). To do so we prove that %
        		\eqref{eq:0} leads to a linear equation in $\partial_xP_tf(x) $ which can be
        		solved explicitly in order to obtain \eqref{eq:main}. The proof of the non-zero drift case is
        		based on a reduction to the case $b=0$ via an application of Girsanov's
        		theorem. The extra difficulty arises because the Girsanov's change of measure depends on the
        		whole path of the process $Y$ and therefore additional arguments are
        		required.
        		
        	\end{itemize}

        	The challenge in obtaining a probabilistic representation of the killed
        	diffusion semigroup is perhaps not surprising. It is known that the
        	stochastic flow corresponding to a killed diffusion can not be
        	differentiated with respect to its initial condition. The presence of the
        	stopping time $\tau$ in the definition \eqref{ptf} of the semigroup $P_t$ is
        	an important hurdle.         %   This is due to the fact that the diffusion sample paths can not
        	%be  $ 1/2 $-H\"older continuous as a function of time   and
        	%the effect of the stopping time. 
        	For example, \cite{MaAREN} proves that diffusion stopping times are not
        	differentiable in the Malliavin sense. On the other hand, it is well known
        	in the theory of partial differential equations, that $P_tf$ is
        	differentiable. This is because the function $u$ defined as $u(t,x):=P_tf(x)$
        	is the solution of an initial value partial differential equation with
        	Dirichlet boundary condition and therefore differentiable on the domain. 
        	
        	%
        	%appear due to the differentiation of the Girsanov measure change according
        	%to different measures $\bar{m} $ (reflecting) and $m $ (stopping) at the
        	%initial and end point of any partition interval in the approximation
        	%process. The reason why $\rho_T $ appears is because the push-forward
        	%derivative property ends there. That is, the path after $\rho_T $ is not
        	%differentiated at all.
        	
        	The proof of the main result hinges on an iterative differentiation of
        	expected values of functionals for the Euler approximation under a discrete time change of probability measures which correspond to an approximation of the killed diffusion, see Proposition \ref{lem:pf}. The resulting probabilistic
        	representation of the relevant functionals for the Euler approximation of
        	the killed diffusion is expressed in terms of a set of processes based on a reflected process which are
        	shown to converge in distribution. To do this we make use of weak
        	convergence results for stochastic equations (see, e.g. \cite{KurtzProtter}%
        	). Here we note that there will be terms converging to the regulator process $ B $. This is known to be a delicate problem in general (see \cite{Jacod}). Once this is done, we identify the probabilistic representation for the
        	derivative of the semigroup corresponding to the killed diffusion as the
        	limit of probabilistic representation of the derivative of semigroup-like
        	expressions corresponding to the killed Euler approximations. 
        	
        	In this article, the first of a series of four papers,
        	%\footnote{In a second paper, we will cover the corresponding multi-dimensional case and, in a third and fourth paper, we will cover applications of the results covered in the first two papers to the analysis of the sensitivity of diffusion semigroups with respect to varying boundary conditions (of various degrees of smoothness) and to the analysis of a sedimentation process in fluids.}, 
        	we provide the basic theoretical argument that can be used in a variety of situations including the case for killed diffusion processes in a multidimensional smooth domain. We will introduce the derivative process $\xi $ with the intention to prove in our subsequent research that the meaning of the derivative process $ \xi $ is close to the usual flow derivative for smooth diffusions. In fact, a potential application of our results can be seen in Theorem 2.2 in \cite{CGK} which will require the formula for $ \xi $ when the initial point is at the boundary.
        	
%        	In fact, one may consider these derivatives using other means such as It\^o-Wiener chaos decompositions. In such a case, one will have to prove the differentiability of the kernel functions of each multiple integral with respect to the initial point. We believe that the convergence of the infinite sum does not hold which probably means that the derivative obtained through this argument is not a random variable. In a heuristic sense, in this article we obtain an alternative definition which is a weak form of derivative which has a meaning and which satisfies a certain type of stochastic equation with jumps. This equation reduces to classical situations in particular cases  (e.g. when the path of the reflected process does not touch the boundary). 

        	A secondary result of the paper is a new version of the Bismut-Elworthy-Li (BEL) formula for killed diffusions in one dimension which is valid up to and including the boundary (see Theorem \ref{th:19}). There are other classical variations of these formulas which are valid only when the initial point is at the interior of the domain.
        	For example, this is the case of the formula appearing in Theorem 2.18 in \cite{Mall}.  Without going into details, Theorem 2.18 in \cite{Mall} gives the following representation in our setting: 
        	\begin{align}
        		\partial_xP_Tf(x)=\mathbb{E}\left[\! f(X^x_T)1_{(\tau>T)}\!\!\int_0^T \!\!\!\theta_t\exp\left(
        		\int_{0}^{t}\!\!\left( \left(b^{\prime
        		}-\frac {(\sigma ^{\prime })^2}2\right)(X^x_{s})ds+\sigma ^{\prime }(X^x_{s})dW_{s}\right)\right )dW_t\right ].
        	\label{eq:MT}
        	\end{align}
        	Here, $\theta $ is a predictable process such that $ \theta(s)=0 $ for $ s>\tau  $ and $ \int_0^T\theta(s)ds=1 $ as detailed in Definition 2.17 of  \cite{Mall}. Under the conditions of Theorem 2.18 of \cite{Mall},  one can not take limits as $ x $ converges to a point in the boundary. On the other hand, as the formula in Theorem \ref{th:19} uses the last excursion of the path of $ X^x $ away from the boundary. Therefore the above  formula is not useful from a simulation point of view when the support of $ f $ is close to the boundary. This is the case in some applications (e.g. see \cite{Cose}). In fact, in the course of this research, we were able to obtain other integration by parts formulas which may be useful in situations when the initial and arrival point are close to the boundary but they seem to be cumbersome to be detailed here.  Furthermore, the method of proof for \eqref{eq:MT} can not be applied here because $ P_{T-t}(Y_t) $ is not a martingale in the present case.

        In the first sequel to the present article, we extend the one dimensional result introduced here to a
        	multi-dimensional case. Following this, we analyze the behaviour of the distribution of a diffusion as a function of the corresponding boundary  of the domain. In particular we study the behaviour of the diffusion as the boundary moves and, separately as the boundary becomes less smooth. with a limiting case when the boundary is H\"older continuous (a rough path). As an application of these theoretical results, the authors will analyse in a fourth paper the process of sedimentation at boundaries of multi-dimensional domains when the diffusion represents movement of sediment/debris carried by the fluid and deposited on the boundary. This is one of the reasons why it is essential that our probabilistic representation covers the case where the initial point can also be chosen on the boundary. 
        	
        	There exists a different version of the BEL formula valid also at the boundary which appeared in \cite{FKL}. Such a formula is of a different nature: it uses interpolation techniques so as to obtain a formula that can be simulated. Its geometrical interpretation it is not clear and it does not offer a natural proxy for the derivative process. 
        	The formula presented in this work contains an integral over the last excursion of the reflected process. For details, see Theorem \ref{th:19}. We exploit this result in the subsequent work to provide a natural physical interpretation of the boundary behavior of the derivative of the semigroup that characterizes the  particle sedimentation rate at the boundary.    
        	
        	Other applications of the representation  \eqref{eq:-1}  are as follows:
        	\begin{itemize}
        		\item Feynman-Kac formulas and derivatives of parabolic equations. (e.g. \cite{Taira}). In this area we will explore the relations of the present approach in the multi-dimensional case with the so-called Steklov problem in partial differential equations.
        		\item Duality between optimal stopping problems and singular control (e.g. \cite{CGK} and \cite{SC}). 
        		\item Greeks in finance (e.g.  \cite{Four}). In this application one may extension of formulas such as \eqref{eq:-1} for the cases when the test function is differentiable and then as a second application deal with some particular cases when $ f $ is not differentiable through Theorem \ref{th:19} and compare with other available methods. 
        		\item Accumulation of sediments at the boundary of a domain (work in progress by the authors).
        	\end{itemize} These developments require additional work in order to incorporate the particular set-up in each application and will be discussed in future research. For further information on applications, heuristic explanations on the arguments used in this research we refer the reader to our YouTube video presentation \cite{YT}.
        	%Also note that, by the Markov property, a similar result can be obtained for
        	%        the derivative of the expectation $\mathbb{E}\left[f(X_{t_1},...,X_{t_n}) %
        	%        \right] $ with $0<t_1<...<t_n\leq T $, again, with respect to the initial
        	%        value of the process $X$.

        	In the multi-dimensional case further geometric considerations and the covariance structure at the boundary force the
        	introduction of jumps in the equation for $\Psi $ which has some dual connections with previous results about the differentiation of reflected processes (see \cite{Andres}, \cite{Burdzy}, \cite{DZ} and \cite{LK}). Therefore, for
        	clarity of presentation, we will treat in this paper only the one dimensional
        	case with an argument which is similar to the one to be used in the multi-dimensional case. The case where $D=(L,R) $
        	is a special case of the multi-dimensional case. 
        	Finally, one may foresee that the uniqueness of the process $ \Psi $ can be somewhat gauged from the fact that it is adapted to the filtration generated by  $Y $ and that a similar formula as \eqref{eq:main} can be obtained in the case that one considers $ \partial_x\mathbb{E}\left[f(X_{t_1},...,X_{t_n})1_{(\tau>T)}\right] $ for $ 0<t_1<....<t_n=T $ for an appropriate class of test functions $ f $. This is a future research direction for the authors.
        	
        	\section{Simple cases and some heuristic arguments}
        	\label{sec:sc}
        	In this section, we will provide some simple arguments for the Brownian case ($ b=0 $, $ \sigma $ constant). This clarifies the reason as to $Y$ is the reflected process but also gives the main line of our arguments in the general case of non-constant coefficients. We remark that the goal in this case is to prove that for $ f:[L,\infty)\to\mathbb{R} $ satisfying $ f(L)=0 $ and  for $ Y $ the reflected Wiener process started at $ x\geq L $, the following probabilistic representation is satisfied  $ x\geq L $
        	\begin{align*}
        	       		\partial_x\mathbb{E}\left[f(X^x_t)1_{(\tau>t)}\right] =&\mathbb{E}\left[f'({ Y_t  })\xi_t\right],\\
        	       			X_t\equiv&  X_t^x=x+W_t,\  \xi_t=1.
        	\end{align*}
        This equality can be rewritten using densities as $$  \partial_x\int_L^\infty f(y)q_t^-(x,y)dy=\int_L^\infty f'(y)q_t^+(x,y)dy.$$
        	Here, $ q_t^- $ and $ q_t^+ $ denote the explicit densities of the killed $ q_t^- $ and reflected Wiener processes, which can be obtained using the reflection principle written below.
        	\begin{align*}
        		q_t^{\mp}(x,y):=&g_t(y-x) \mp g_t(y+x-2L),\quad
        		 g_t(y):=\frac{e^{-\frac{y^2}{2t}}}
        		{\sqrt{2\pi t}}.
        	\end{align*}
        	We will now dwelve into a variation of the above proof as this idea can not be applied directly to the general non-constant coefficient case because densities in the general case do not have the symmetry property stated in \eqref{eq:dens}. 
        	
        	Instead, we will give an alternative argument which may have a deeper probabilistic flavor which will be the one used in the proofs. 
        	In fact, from the above  expression one can see that given the initial and final value of the Wiener process in the interval $ [0,t] $, the probability of crossing the boundary $ L $ is given as
        	\begin{align*}
        		p_t(x,y)=\mathbb{P}(\tau<t|X^x_t=y)=e^{-2\frac{(y-L)(x-L)}{t}}.
        	\end{align*}
        	We define the non-negative random variables: $ M_t=1-p_t(x,X^x_t) $ and $ \bar{M}_t=1+p_t(x,X^x_t) $.
        	With these definitions, the main line of the argument to be used in the general case uses the above equality is as follows:
        	\begin{align*}
        		\partial_x\mathbb{E}\left[f(X^x_t)1_{(\tau>t)}\right] =&\partial_x\mathbb{E}\left[f(X^x_t)\mathbb{E}\left[1_{(\tau>t)}|X^x_t=y\right]|_{y=X^x_t}\right]=\partial_x\mathbb{E}\left[f(X^x_t)(1-p_t(x,X^x_t))\right]=\partial_x\mathbb{E}\left[f(X^x_t)M_t\right]	.
        	\end{align*}
        	The above calculation shows how to handle the indicator $ 1_{(\tau>t)} $ through the use of the conditional probability $  1-p_t(x,X^x_t)$.
        	This conditional probability can now be differentiated. 
        	In contrast, this property can not be obtained in the case that coefficients of the diffusion $ X $ are not constant because one does not have an explicit expression for the crossing probabilities.
        	
        	 Next, we can compute the expressions for the derivative using an integration by parts formula:
        	\begin{align}
        		\partial_x\mathbb{E}\left[f(X^x_t)M_t\right]=&\mathbb{E}\left[f'(X^x_t)(1-p_t(x,X^x_t))\right]-\mathbb{E}\left[f(X^x_t)\partial_xp_t(x,X^x_t)\right]\notag\\
        		=&\mathbb{E}\left[f'(X^x_t)(1+p_t(x,X^x_t))\right]
        		=\mathbb{E}\left[f'(X^x_t)\bar{M}_t\right]\notag\\
        		=&\mathbb{E}\left[f'({ Y_t  })\mathcal{E}_t\right].
        		\label{eq:refl}
        	\end{align}
        	We have used the integration by parts formula in the second equality which is the reason why we require the ellipticity condition on $ \sigma $:
        	\begin{align*}
        		\mathbb{E}\left[f(X^x_t)\partial_xp_t(x,X^x_t)\right]=-2\mathbb{E}\left[f'(X^x_t)p_t(x,X^x_t)\right].
        	\end{align*}
        	The non-negative random variables $ M_t $ and $ \bar{M}_t $ can be interpreted heuristically as a change of measure from Bownian motion to killed and reflected Brownian motion respectively. These measures are not related in the same way as in the above expression for the general non-constant coefficient case (for more on this, see the case of the Wiener process with drift). In fact, a change of measure which will change a diffusion into a reflected diffusion does not exist. This is the main reason why we have to resort to an approximation argument.
        	%\begin{align*}
        	%	q_t^{\mp}(x,y):=&g_t(y-x) \mp g_t(y+x-2L)=\left(1\mp e^{-2\frac{(y-L)(x-L)}{t}}\right)
        	%	g_t(y-x),\\
        	%	=&(1\mp p_t(x,y))g_t(y-x);\quad g_t(y)=\frac{e^{-\frac{y^2}{2\sigma^2t}}}
        	%	{\sqrt{2\pi\sigma }}
        	%\end{align*}
        	%
        	%\begin{gather*}
        	%	%	\partial_x\mathbb{E}\left[f(X_t)1_{(\tau>t)}\right] =	\partial_x\!\!\!\int_L^\infty \!\!\!\!\! f(y)q^-_t(x,y)dy=\!\!\!\int_L^\infty\!\!\!\!\! f'(y)q^{\textcolor{red}{ +  }}(x,y)dy=\mathbb{E}\left[f'(\textcolor{red}{ Y_t  })\mathcal{E}_t\right].
        	%	\hspace{-0.9cm}\partial_x\mathbb{E}\left[f(X_t)(1-p_t(x,X_t))\right]=\mathbb{E}\left[f'(X_t)(1-p_t(x,X_t))\right]-\mathbb{E}\left[f(X_t)\partial_xp_t(x,X_t)\right]\\
        	%	\hspace{1.9cm}=\mathbb{E}\left[f'(X_t)(1+p_t(x,X_t))\right]
        	%	=\mathbb{E}\left[f'(\textcolor{red}{ Y_t  })\mathcal{E}_t\right]
        	%\end{gather*}
        	%We have used the integration by parts formula
        	%\begin{align*}
        	%	\mathbb{E}\left[f(X_t)\partial_xp_t(x,X_t)\right]=-2\mathbb{E}\left[f'(X_t)p_t(x,X_t)\right]
        	%\end{align*}
        	%%	When applying $ \partial_x $ to the above expression there is a change in sign!

        	Now, consider the case of Brownian motion with constant drift. That is, $ X^x_t=x+bt+\sigma W_t $. In this case the above argument can not be used because the reflection principle is not satisfied for the Wiener process with drift. Although the densities for the killed and reflected Wiener processes with drift are explicitly known, the above probabilistic type calculation does not work in a straightforward fashion. In fact, we use Girsanov's change of measure to obtain a representation which we can differentiate:
        	\begin{align*}
        		\mathbb{E}\left[ f\left( X^x_{t}\right) 1_{(
        			t<\tau ) }\right] =\mathbb{E}^{\mathbb{Q}}\left[ f(X^x_t) \frac{d\mathbb{P}}{d\mathbb{Q}} 1_{(
        			t<\tau ) }\right] =\mathbb{E}^{\mathbb{Q}}\left[ f(X^x_t)\exp \left( \frac{b}{a}({X}^x_{t}-x)-\frac{
        			b^{2}}{2a}t\right)1_{(
        			t<\tau )}\right].
        	\end{align*}
        	Therefore applying the formula, we have obtained for the Wiener process with no drift, we have
        	\begin{align*}
        		\partial_x\mathbb{E}\left[ f\left( X^x_{t}\right) 1_{\left\{
        			t<\tau \right\} }\right]=&\mathbb{E}^{\mathbb{Q}}\left[ f'\left( Y_{t}\right)\exp \left( \frac{b}{a}({Y}_{t}-x)-\frac{
        			b^{2}}{2a}t\right)\right]\\
        		&+\frac{b}{a}\mathbb{E}^{\mathbb{Q}}\left[ f\left( Y_{t}\right)\exp \left( \frac{b}{a}({Y}_{t}-x)-\frac{
        			b^{2}}{2a}t\right)\right]-\frac{b}{a}\mathbb{E}^{\mathbb{Q}}\left[ f\left( Y_{t}\right)\exp \left( \frac{b}{a}({Y}_{t}-x)-\frac{
        			b^{2}}{2a}t\right)1_{(
        			t<\tau  )}\right]\\
        		=&\mathbb{E}\left[\left(f'\left( Y_{t}\right)+f\left( Y_{t}\right)\frac ba1_{(\tau\leq t)}\right) \exp \left( \frac{b}{a}B_t\right)\right].
        	\end{align*}
        	Here $ B_t $ is the regulator process associated to the reflected Wiener process with drift $$ Y_t=x+bt+\sigma W_t+B_t\geq L .$$ The above formula is obtained in Theorem \ref{th:main} (see also \eqref{eq:0}) for the general (non-constant) coefficient case.
        	It is not clear to us if there is an easier way to deduce the formula \eqref{eq:-1}. 
        	In fact, our method consists of  showing that the above equation gives a special type  of linear equation for the derivative $\partial_x\mathbb{E}\left[ f\left( X^x_{t}\right) 1_{(
        		t<\tau )}\right]  $. Once that equation is solved we obtain the claimed formula for the case of the Wiener process with drift. For more details in the general case, see Theorem \ref{th:36} and its proof.
        	
%        	As the argument we will uses approximations a larger part of our proof consists in proving that the first derivative of the approximating quantities are equicontinuous. This is done in the Appendix.
        	
        	Finally, we want to remark that, one may seek other types of arguments in the one dimensional case but as we plan to show in subsequent work, the argument presented here is the one that allows the extension of the present results to the multi-dimensional case.
        Similarly, one may also seek arguments using partial differential equations. These arguments are not available as they require a careful analysis of the corresponding functions at the boundary. In fact, not even in the multidimensional reflected situation (Neumann case), these analytical arguments are available, because the corresponding partial derivative equation for the gradient should contain the jump behavior of the process as in e.g. \cite{Andres}. The corresponding results for the associated Feynman-Kac formulas for the derivative of the killed diffusion will also be discussed in future research.
        	
        	In the next section, we introduce the set-up for discrete approximations.

        	\section{Notation, Assumptions and the push-forward formula}
        	\label{sec:2} %\subsection{General notation}
        	We will work on the time interval $[0,T]$ with the following uniform partition $%
        	t_{i}\equiv t_{i}^{n}:=i\Delta$, $i=0,...,n$ with $\Delta=\frac{T}{n}$. We assume without loss of generality, that $%
        	\Delta\leq 1 $ and $ T $ is fixed. 
        	
        	All constants $c,C$, $M $ or 
        	$\mathsf{M} $ appearing in the article are independent of $n $. They may
        	depend on the general bounds for the coefficients $b $ and $\sigma $ and
        	their derivatives as well as on the uniform elliptic condition constant $%
        	\sigma(x)\geq c>0 $ for all $x\in\mathbb{R} $. We also denote $ a(x)=\sigma^2(x) $. 
        	
        	The supremum norm for a
        	function $f :\mathbb{R}\to\mathbb{R}$ is denoted as $\|f\|_\infty
        	:=\sup\{|f(x)|;x\in\mathbb{R}\} $. Some of the functions may only be defined
        	in $[L,\infty) $. In this case we assume that they have extensions defined
        	on $\mathbb{R} $ that satisfy the required continuity/smoothness properties. 
        	%We use $\delta _{L}(x)$ to denote
        	%the Dirac delta distribution at the point $L$ and $ g_a(x) $ to denote the %Gaussian density with mean zero and variance $ a $.
        	
        	%\subsection{Set-up}
        	In the following, we will use the Euler approximation process of $X$ with
        	initial point $X_0 $, defined as follows: For $i=0,...,n-1$, 
        	\begin{equation*}  %\label{xn}
        		{X}_{t_{i+1}}^{n}=X_{t_i}^{n}+b(X_{t_i}^{n})(t_{i+1}-t_{i})+\sigma
        		(X_{t_i}^{n})(W_{t_{i+1}}-W_{t_{i}}).
        	\end{equation*}%
        	We will use the following notation to simplify equations: $
        	X_i^{n,x}=X_{t_i}^{n,x} $  where $X_{0}^{n,x}=x$ and in general $ X_i=X^n_{t_i} $. For coefficients, we use  $b_i=b_{i}^{n}=b(X_{i}^{n,x}),$ $\sigma_i=\sigma
        	_{i}^{n}=\sigma (X_{i}^{n,x})$ and\footnote{%
        		In general, we use this difference operator notation for other sequences of random
        		variables too.} $\Delta_{i+1}W=W_{t_{i+1}}-W_{t_{i}}$. Therefore, we have 
        	\begin{equation}
        		X_{i+1}^{n,x}\equiv X_i^{n,x}+b(X_i^{n,x})(t_{i+1}-t_{i})+\sigma
        		(X_i^{n,x})\Delta_{i+1}W\equiv X_i^{n,x}+b_{i}^{n}\Delta +\sigma
        		_{i}^{n}\Delta_{i+1}W. \label{eq:defbX}
        	\end{equation}%

        	Subsequently, whenever possible, we will omit the dependence of $%
        	X_{i+1}^{n,x}$ on $n$ and/or $x$ in their notation as well as that of $%
        	b_{i}^{n}$ and $\sigma _{i}^{n}$ on $n$. In particular, one has 
        	\begin{equation*}
        		\partial _{x}X_{p}=\partial _{x}X_{p-1}(1+b_{p-1}^{\prime
        		}(t_{p}-t_{p-1})+\sigma _{p-1}^{\prime
        		}(W_{t_{p}}-W_{t_{p-1}}))=\prod_{i=1}^{p}(1+b_{i-1}^{\prime }\Delta +\sigma
        		_{i-1}^{\prime }\Delta_iW). 
        	\end{equation*}
        	We will use the standard imbedding  of the discrete time process $X_i$ into the piece-wise constant continuous time process $X=X^n$ defined as $X_t=X^n_{t_{i}}$ for 
        	$t\in        [t_i,t_{i+1})$. 
        	
        	\textbf{\noindent Notation.} \textit{In what follows, in order to simplify
        		the notation, we will use $\partial_i\equiv \partial_{X_i} $ (and therefore $%
        		\partial_0=\partial_x$). To be more precise, $\partial_{X_i} $means that if
        		we have a random variable of the form $f (X_i)$, where $f $ is a real valued
        		function, then $\partial_{i}f(X_i) = f^{\prime }(X_i)$, where $f^{\prime }$%
        		is the standard derivative of $f$. In particular, we also have for $f(x,y) $%
        		, $x,y\in\mathbb{R} $, $\partial_if(X_i,\Delta_{i+1}W)=\partial_xf
        		(X_i,\Delta_{i+1}W)$.\newline
        	}

        	As explained in the introduction, the probabilistic representation of the derivative of the one dimensional
        	killed diffusion semigroup
        	uses the semigroup of the corresponding reflected diffusion. To obtain the result, we also make use of the piecewise interpolated version $X^{c,n}$ of the Euler approximation defined as        
        	\begin{equation}  \label{xcn}
        		{X}_{t}^{c,n}=X_{t_i}^{c,n}+b(X_{t_i}^{c,n})(t-t_{i})+\sigma
        		(X_{t_i}^{c,n})(W_{t}-W_{t_{i}}).
        	\end{equation}%
        	for $i=0,...,n-1$ and $t\in \left[ t_{i-1},t_{i}\right] $.
        	Note that $X_{t_i}^{c,n}=X_{t_i}^{n}=X_i$ for $i=0,...,n$. %The use of the %interpolated version will play a central role in identifying the limits %of the approximating sequences, see Theorem \ref{bigT}.   

        	Let $U_{i}$, $ i\in\mathbb{N} $, be i.i.d. uniformly distributed random variables on $[0,1]$
        	independent of $W$. It is known that we can identify the event that the path
        	of the Euler approximation $X^{c,n}$, as defined in %
        	\eqref{xcn}, did not touch the boundary in the interval $\left[ t_{i-1},t_{i}%
        	\right] $ with the event $\left\{ U_{i}>p_{i}\right\} $ where 
        	\begin{equation*}
        		p_{i}\equiv p_{i}(X_{i-1}^{c,n},X_{i}^{c,n})=p_{i}(X_{i-1},X_{i}):=e^{-2\frac{\left( X_{i-1}-L\right) \left(
        				X_{i}-L\right) }{a_{i-1}\Delta }},
        	\end{equation*}%
        	where $a_{i}=\left( \sigma _{i}\right) ^{2}$, 
        	%For more details, \textcolor{red}{ see Section \ref{sec:rem}  }. 
        	see \cite{gobetesaim} for details.  Define the ``change of measure" process
        	\footnote{%
        		In other words, we define the measure $\mathbb{P}^{M}$ such that $\left. \frac{d\mathbb{P}^{M}%
        		}{d\mathbb{P}}\right\vert _{\mathcal{F}_{{n}}}=M^n_{n}$. Note that $M^n_{n}$ is not a
        		martingale, but a \emph{supermartingale}, see \eqref{lem:3} for details.
        		Hence the measure $\mathbb{P}^{M}\,\ $is, in fact, a sub-probability measure as it
        		does not integrate to 1. Under it, the continuously interpolated Euler approximation $X^{c,n}_{n}$ defined by (\ref{xcn}) has the same law as the process $X^{c,n}_{n}$ killed when it exits
        		the domain, see \cite{gobetesaim} for details.
        		
        	}
        	
        	\begin{equation*}
        		M^n_{n}:=\prod_{i=1}^{n}m_{i}=\prod_{i=1}^{n}1_{\left( {X}_{i}>L\right)
        		}1_{\left( U_{i}>p_{i}\right) }.
        	\end{equation*}%
        	The conditional expectation with respect to the $\sigma$-algebra $\mathcal{F}%
        	_{i}=\sigma (U_{j},X_{j};j=1,...,i)$ will be denoted by $\mathbb{E}_{i}$ and
        	we also use the notation $\mathbb{E}_{i,x}[\cdot]\equiv \mathbb{E}_{i}[\cdot|X_i=x] $. In particular, we may use sometimes $\mathbb{E}_{0,x}$ if we want to stress that the expectation is a function of $ x $. In this sense $\mathbb{E} $ can be interpreted as $\mathbb{E}_0$ or $\mathbb{E}_{0,x} $ without any confusion.
        	
        	In order to reduce the length of various equations we let $%
        	f_{i}:=f_{i}(X_{i})$, $i=0,...,n$ denote a random variable obtained from the
        	evaluation of the function $f_{i}$ on the random variable $X_{i}$ where $%
        	f_{i}(x):=\mathbb{E}_{i,x}\left[ f\left( {X}_{n}\right) M^n_{i:n}%
        	\right] $ where $M^n_{i:n}:=\prod_{j=i+1}^{n}m_{j}  $. Therefore, using this notation and the Markov property we have 
        	for $ \tau^n:=\inf\{s\geq 0; X^{c,n}_s\in\partial D\} $.
        	\begin{equation*}
        		\mathbb{E}_0\left[f(X_n)1_{(\tau^n>T)}\right]=
        		\mathbb{E}_0\left[f(X_n)\mathbb{E}\left[1_{(\tau^n>T)}\Big/\mathcal{F}_n\right]\right]=
        		\mathbb{E}\left[ f\left( {X}_{n}\right) M^n_{n}\right] =\mathbb{E}\left[
        		f_{1}M^n_{1}\right] .
        	\end{equation*}%
        	In general, 
        	\begin{align*}
        		f_{i}(X_{i}) :=&\mathbb{E}\left[ f\left( {X}_{n}\right) M^n_{n}(M^n_{i})^{-1}|%
        		\mathcal{F}_{i}\right] =\mathbb{E}\left[ f\left( {X}_{n}\right)
        		M^n_{n}(M^n_{i})^{-1}|X_{i}\right]  \\
        		=&\mathbb{E}_{i}\left[ f\left( {X}_{n}\right)
        		M^n_{n}(M^n_{i})^{-1}\right] =\mathbb{E}_{i}\left[ f_{i+1}\left( {X}_{i+1}\right) m_{i+1}\right] .
        	\end{align*}
        	
        	Note that the function $f_{i}$ satisfies the boundary condition $f_{i}(L)=0 $%
        	, $i=0,...,n$. In fact, if $i<n $ and $X_i=L $ then $p_{i+1}=1 $ and
        	therefore $m_{i+1}=0 $. Note that the case $i=n$ requires that $f(L)=0$.
        	
        	As noted before, under the measure $M^n_n $ the continuously interpolated Euler approximation $X^{c,n}_{n}$ defined by \eqref{xcn} has the same law as the process $X^{c,n}_{n}$ killed when it reaches the boundary $ L $. 
        	In fact, one
        	has, for $\tau^i:=\inf\{s>t_i;X^{c,n}_s=L\} $ 
        	\begin{align*}
        		\mathbb{E}_i\left[ f\left( {X}_{n}\right) M^n_{i:n}\right]= \mathbb{E}_i\left[ f\left( {X}_{n}^{c,n}\right) M^n_{i:n}\right] =
        		\mathbb{E}        [f(X^{c,n}_{T\wedge \tau^i})|X_{t_i}^{c,n}],%= \mathbb{E}        [f(X^{c,n}_{T\wedge \tau^i})|X_i],
        	\end{align*}
        	see \cite{gobetesaim} for details.
        	
        	In the following, we will use the notation: 
        	%denote by $g_{i+1}$ be the conditional density of
        	%the increment $\sigma _{i}\Delta_{i+1} W$ with respect to $\mathcal{F}_{i}$,
        	%i.e., 
        	\begin{align}
        		g_{i+1}(z)=& \frac{1}{\sigma _{i}\sqrt{2\pi \Delta }}\exp \left( -\frac{z^{2}%
        		}{2a_{i}\Delta }\right), \label{def:gi} \\
        		X_{i}^{L}=& X_{i-1}-L,~~~~X_{i}^{L,\sigma }=\frac{X_{i}-L}{\sigma _{i}\sqrt{%
        				\Delta }}.  \label{eq:not}
        	\end{align}%
        	%       In particular, note that $\mathbb{E}\left[ \sigma _{i}\Delta _{i+1}W|%
        	%       \mathcal{F}_{i}\right] =\sigma _{i}\mathbb{E}\left[ \Delta _{i+1}W|\mathcal{F%
        	%       }_{i}\right] =0$.\newline
        	
        	\textbf{\noindent Notation.} %\label{def:2} 
        	\textit{For the following proposition and later on we will use the following
        		random variables 
        		\begin{align*}
        			\bar{m}_{i}:=&1_{\left( {X}_{i}>L\right) }\left( 1+1_{\left( U_{i}\leq
        				p_{i}\right) }\right), \\
        			e_{i}:=&1+b_{i-1}^{\prime }\Delta 1_{\left( U_{i}>p_{i}\right) }+\sigma
        			_{i-1}^{\prime }(1_{\left( U_{i}>p_{i}\right) }\Delta _{i}W-1_{\left(
        				U_{i}\leq p_{i}\right) }\sqrt{\Delta }X_{i-1}^{L,\sigma }), \\
        			h_{i}:=&\frac{b_{i-1}}{a_{i-1}}1_{(U_{i}\leq p_{i})}+X_{i-1}^{L}\partial
        			_{i-1}\left( \frac{b_{i-1}}{a_{i-1}}\right) 1_{(U_{i}\leq p_{i})}.
        		\end{align*}
        	}
        	
        	\noindent Several processes are defined by products of a
        	given sequence of random variables $(a_k)_k$. More precisely, we denote 
        	\begin{align*}
        		A^n_{i:j}:=\prod_{k=i+1}^ja_k.
        	\end{align*}
        	In the above, we may have $(a,A)=(e,E),
        	(\bar{m},\bar{M}),(m,M) $. We will also use the notation $%
        	A^n_{0:i}\equiv A^n_i $ and, by convention, we define $A^n_{i:j}\equiv 1 $
        	if $j=i-1 ,i$. Therefore, we have 
        	\begin{align*} 
        		(\bar{M}^{n}_i,E^{n}_i):=\left (\prod_{j=1}^{i}\bar{m}_{j},%
        		\prod_{j=1}^{i}e_{j} \right ).
        	\end{align*}
        	
        	%As explained previously, we will use a functional notation where the basic
        	%variables will be $(1_{U_i\leq p_i},X_{i-1},X_i) $ which are denoted by the
        	%algebraic variables $(u,x,y)\in \{0,1\}\times [L,\infty)^2 $. In this sense,
        	%we may write 
        	%\begin{eqnarray*}
        	%\bar{m}(u,x,y):= &&1_{\left( y>L\right) } \left( 1+u \right). \\
        	%e(u,x,y):= &&1+b^{\prime }(x)\Delta (1-u)+\sigma ^{\prime }(x)r(u,x,y), \\
        	%h(u,x,y):= &&\frac{b(x)}{a(x)}u+(x-L)\left( \frac{b}{a}\right)^{\prime }(x)u
        	%\\
        	%r(u,x,y):= &&u \frac{y-x-b(x)\Delta}{\sigma(x)}-u\left( \frac{x-L}{\sigma (x)%
        	%}\right) .
        	%\end{eqnarray*}
        	%
        	%Similar formulas will appear later on and in that sense we understand
        	%expressions such as $\partial_if_i $ or the application of formulas in
        	%Section \ref{sectionwithproofs} and in particular the application of
        	%Corollary \ref{cor:23}.
        	
        	%\section{The push forward derivative formula}
        	Now, we give a first important result to obtain the derivative of the semigroup
        	of the killed diffusion. 
        	The main idea is to approximate it with the derivative of the
        	conditional expectation of a function of the approximating Markov chain $%
        	X^{n}$ in each time interval. This leads to an iterative forward formula,
        	which we call the push-forward derivative formula as it will be iterated from $ i=1 $ to $ i=n $:
        	
        	\begin{proposition}
        		(push-forward derivative formula)\label{lem:pf} Let $F_{i}:[L,\infty)%
        		\rightarrow \mathbb{R}$ be a continuous function that satisfies the boundary
        		condition $F_{i}(L)=0$. We also assume that $F_{i}$ is differentiable on $%
        		(L,\infty)$ and right differentiable at the boundary $L$. Then for any $%
        		i=1,...,n$ 
        		\begin{equation}
        			\partial _{i-1}\mathbb{E}_{i-1}\left[ F_{i}m_{i}\right] =\mathbb{E}_{i-1}%
        			\left[ \left( \partial _{i}{F}_{i}e_{i}+F_{i}h_{i}\right) \bar{m}_{i}\right]
        			.  \label{recurrenceformula}
        		\end{equation}%
        		In particular, for ${\tau }_{n}:=\inf \{s>0;X_{s}^{n,x}=L\}$, (\ref%
        		{recurrenceformula}) implies that for $ f\in C^1_b $
        		\begin{equation}
        			\label{fprime}
        			\begin{aligned}
        				\partial _{x}\mathbb{E}\left[ f\left( X_{T\wedge \tau _{n}}^{n,x}\right) %
        				\right] =&\partial _{x}\mathbb{E}\left[ f\left( X_{n}^{n,x}\right) M^n_{n}%
        				\right] \\=&\mathbb{E}\left[ {f}^{\prime }\left( X_{n}^{n,x}\right) E_{n}^{n}%
        				\bar{M}_{n}^{n}\right] +\sum_{i=1}^{n}\mathbb{E}\left[ f_{i}\left(
        				X_{i}^{n,x}\right) E_{i-1}^{n}h_{i}\bar{M}_{i}^{n}\right] .  
        			\end{aligned}
        		\end{equation}
        	\end{proposition}
        	
        	\noindent The proof of this result is included in Section \ref{sec:dr}.%
        	
        	\begin{remark}
        		$\left.\right.$
        		\begin{enumerate}
        			\item The above formula holds true also for $x=L$ if we interpret the derivatives as their limits from the right of $ x=L $. 
        			
        			%\item In the above sense terms that contain $X^L_{i-1} 1_{(U_{i}\leq p_{i})} $$ are associated to the local time in the limit.
        			
        			%\item In the above sense, we will study the weak limit behavior of the
        			%processes appearing in the above formula using results of weak convergence
        			%of stochastic integrals in \cite{KP}.
        			
        			\item The study of the second term in \eqref{fprime} is the most delicate
        			part of the analysis as explained in Section \ref{sec:sc}. This term will lead to the appearance of the random
        			time $\rho_T $ in \eqref{eq:0}.
        		\end{enumerate}
        	\end{remark}
        	
        	{\color{red} }
        	
        	In Lemma \ref{boundedderivative}, we show that the terms on the right hand
        	side of \eqref{fprime} are bounded uniformly in $n$ and, therefore so is $%
        	\partial _{x}\mathbb{E}\left[ f\left( X_{T\wedge \tau _{n}}^{n,x}\right) %
        	\right]$. A result similar to Proposition \ref{lem:pf} is valid for second
        	derivative of $\mathbb{E}_{i-1}\left[ f_{i}m_{i}\right]$ which we use in
        	Section \ref{app:6.2a} to show that the sequence of first derivatives is
        	equi-continuous, and therefore it converges uniformly. 
        	
        	\section{The Girsanov change of measure and preparatory lemmas}
        	
        	\label{sec:4} The subsequent analysis of the identity (\ref{fprime}) starts with an application of Girsanov's transformation of measure to the original space $(\Omega ,%
        	\mathcal{F},\mathbb{P})$ so that it removes the drift in the Euler approximation\ (%
        	\ref{eq:defbX}) as explained in Section \ref{sec:sc}. This is necessary in order to be able to use $\bar{M}^n $
        	as characterizing a change of measure for reflecting processes.\footnote{In fact, \eqref{eq:refl} is only valid for Brownian motion and it is not valid for Brownian motion with drift.} We start by introducing  the (discrete) exponential martingale corresponding to the
        	increments 
        	\begin{equation*}
        		Z_{j}:=\Delta _{j}W+\sigma _{j-1}^{-1}b_{j-1}\Delta ,
        	\end{equation*}%
        	defined as 
        	\begin{equation}
        		\mathcal{\mathcal{K}}_{i}^{n}=\exp \left( \sum_{j=1}^{i}{\kappa }%
        		_{j}^{n}\right) ,\ \ \ \ \ \ \ \ {\kappa }_{j}^{n}=\frac{b_{j-1}}{\sigma
        			_{j-1}}Z_{j}-\frac{1}{2}\left( \frac{b_{j-1}}{\sigma _{j-1}}\right)
        		^{2}\Delta.  \label{Kn}
        	\end{equation}%
        	With these definitions, we let $\widetilde{\mathbb{P}}\equiv \widetilde{\mathbb{P}}^n$ to be given by 
        	\begin{equation*}
        		\left. \frac{d\widetilde{\mathbb{P}}^n}{d\mathbb{P}}\right\vert _{\mathcal{F}_{{i}}}=(\mathcal{K}%
        		_{i}^{n})^{-1}.
        	\end{equation*}%
        	Then, under $\widetilde{\mathbb{P}}^n$, the Euler approximation $X^{n,x}$ satisfies 
        	\begin{equation}
        		X_{i}^{n,x}=X_{i-1}^{n,x}+\sigma _{i-1}{Z}_{i},  \label{eq:defbXunderPtilde}
        	\end{equation}%
        	and, under $\widetilde{\mathbb{P}}^n$, the random variables ${Z}_{i}$, $i=1,...,n$ are
        	i.i.d. random variables with common distribution $N\left( 0,\Delta \right) $. The terms appearing in the identity \eqref{fprime} now become 
        	\begin{align}
        		\mathbb{E}\left[ {f}^{\prime }\left( X_{n}^{n,x}\right) E_{n}^{n}\bar{M}%
        		_{n}^{n}\right] =&\mathbb{\widetilde{E}}\left[ {f}^{\prime }\left(
        		X_{n}^{n,x}\right) {\mathcal{\mathcal{K}}}_{n}^{n}E_{n}^{n}\bar{M}_{n}^{n}%
        		\right]  \label{fprimetilde1} \\
        		\mathbb{E}\left[ f_{i}\left( X_{i}^{n,x}\right) E_{i-1}^{n}h_{i}\bar{M}^n_{i}%
        		\right] =&\mathbb{\widetilde{E}}\left[ f_{i}\left( X_{i}^{n,x}\right) {\mathcal{%
        				\mathcal{K}}}_{i}^{n}E_{i-1}^{n}h_{i}\bar{M}_{i}^{n}\right] .
        		\label{fprmetilde2}
        	\end{align}
        	Here $ \mathbb{\widetilde{E}} $ denotes the expectation under $\widetilde{\mathbb{P}}^n$ and from now on we work under this measure and therefore the
        	Euler approximation $X^{n,x}$ satisfies (\ref{eq:defbXunderPtilde}).
        	
        	\begin{remark}
        		\label{rem:3}
        		Of course, under $\widetilde{\mathbb{P}}^n$, the laws of the random variables $m_{i},\bar{m}%
        		_{i},e_{i},$ will change as all of them depend on $\Delta _{i}W$. In
        		particular, for $\bar{b}=b^{\prime }-\sigma^{\prime }\sigma^{-1}b $, we can rewrite $ e_i $ using $ Z_i $ as 
        		\begin{align}
        			e_{i} =&1+\bar{b}_{i-1}\Delta 1_{\left( U_{i}>p_{i}\right) }+\sigma
        			_{i-1}^{\prime }Z_{i}-\sigma _{i-1}^{\prime }1_{\left( U_{i}\leq
        				p_{i}\right) }\left( Z_{i}+\sqrt{\Delta }X_{i-1}^{L,\sigma }\right) 
        			\notag \\
        			=&1+\bar{b}_{i-1}\Delta 1_{\left( U_{i}>p_{i}\right) }+\sigma
        			_{i-1}^{\prime }\left( Z_{i}-\widetilde{\mathbb{E}}_{i-1}[Z_{i}\bar{m}%
        			_{i}]\right) +\gamma _{i}  \label{eq:einZZ} \\
        			\gamma _{i} :=&\sigma _{i-1}^{\prime }\left( \widetilde{\mathbb{E}}_{i-1}[Z_{i}%
        			\bar{m}_{i}]-1_{\left( U_{i}\leq p_{i}\right) }\left( Z_{i}+\sqrt{\Delta }%
        			X_{i-1}^{L,\sigma }\right) \right). \label{eq:gammainZ}
        		\end{align}
        	\end{remark}
        	
        	Next we define the measure $ {\mathbb{Q}}^n$ such that $\left. \frac{d{\mathbb{Q}}^n}{d\widetilde{\mathbb{P}}^n}\right\vert _{\mathcal{F}_{{n}}}=\bar{M}^n_{n}$. Note that
        	the process $\bar{M}^n_{n}$ is a discrete time positive martingale with mean $ 1 $ (see Lemma \ref{lem:333}).  
        	%\begin{remark}
        	%\label{rem:girref} 
        	Under the measure ${\mathbb{Q}}^n$, the (discrete) Euler
        	approximation $X^{n}_i$ has the same $ \mathcal{F}_{i-1} $-conditional law as the reflected process in that interval. In fact, for any bounded measurable function $ f$, one has
        	\begin{align*}
        		\mathbb{E}_{i-1}[f(X^n_i)\bar{m}_i]=\int_L^\infty f(y)\left(g_i(y-X_{i-1}^n)+g_i(y+X_{i-1}^n+2L)\right)dy.
        	\end{align*}
        	%               \begin{equation}
        	%               X_{i}^{n}=X_{i-1}^{n}+\sigma _{i-1}\left( Z_{i}-\widetilde{\mathbb{E}}%
        	%               _{i-1}[Z_{i}\bar{m}_{i}]\right) +\sigma _{i-1}\widetilde{\mathbb{E}}_{i-1}[Z_{i}%
        	%               \bar{m}_{i}],  \label{eq:defbXunderPtilde}
        	%               \end{equation}%
        	%               has the same law as the value of the continuously interpolated process $%
        	%               X^{n} $ (\emph{after} the Girsanov transformations is applied and the drift
        	%               is removed) evaluated at the final time $t_{n}$ that is reflected at the
        	%               boundary of the domain, see \cite{gobetesaim} for details. In particular,
        	
        	Therefore
        	for any measurable function $F:[L,\infty )\rightarrow \mathbb{R}$, for which
        	the two sides of the identity below are well defined, we have
        	\begin{equation}
        		\widetilde{\mathbb{E}}\left[ F(X_{i})\bar{M}^n_{i}\right] =\mathbb{E}\left[
        		F\left( \mathcal{X}_{t_{i}}\right) \right] ,\ \ \ i=0,...,n.
        		\label{eq:reft}
        	\end{equation}%
        	where the process $(\mathcal{X},\Lambda )$ is the solution of the reflected
        	Euler scheme as defined in Section \ref{app:res}. 
        	% \end{remark}

        	In order to understand the structure of the limit procedure, we define the
        	sequence of approximation process $ E^n   $ as well as its driving processes $R^{n}$ and $\Gamma ^{n}$ as follows
        	\begin{align}
        		(R_{t}^{n},\Gamma _{t}^{n},E_t^n) &:=\sum_{i=0}^{n-1}1_{[0,t_{i+1})}(t)
        		\left(\left( Z_{i}-\widetilde{\mathbb{E}%
        		}_{i-1}[Z_{i}\bar{m}_{i}]\right), \gamma _{i},E^n_{i}
        		\right). \label{rn} 
        		%                \Gamma _{t}^{n} &:=\sum_{i=0}^{n-1}1_{[0,t_{i+1})}(t)}
        		%                \label{gamman}
        	\end{align}
        	Finally, let us introduce the process which will eventually converge to the local
        	time process $B$ appearing in \eqref{y1}, respectively in \eqref{eq:rep} 
        	\begin{equation}
        		B_{t}^{n}:=\sum_{i=0}^{n-1}1_{[0,t_{i+1})}(t)\sigma _{i-1}\widetilde{\mathbb{E}}%
        		_{i-1}[Z_{i}\bar{m}_{i}].  \label{Bn}
        	\end{equation}%
        	In the particular case $ t=T $, we let $ (R_{T}^{n},\Gamma _{T}^{n},B_{T}^{n}) :=(R_{T-}^{n},\Gamma _{T-}^{n},B_{T-}^{n}) $. From (\ref{eq:defbXunderPtilde}), (\ref{eq:einZZ}), (\ref{eq:gammainZ}) and
        	(\ref{Kn}) we deduce that 
        	\begin{align}
        		X_{i}^{n} =&X_{i-1}^{n}+\sigma
        		_{i-1}(R_{t_{i}}^{n}-R_{t_{i-1}}^{n})+(B_{t_{i}}^{n}-B_{t_{i-1}}^{n})
        		\notag \\
        		E_{i}^{n} =&E_{i-1}^{n}\left( 1+\bar{b}_{i-1}\Delta
        		1_{(U_{i}>p_{i})}+\sigma _{i-1}^{\prime
        		}(R_{t_{i}}^{n}-R_{t_{i-1}}^{n})+(\Gamma _{t_{i}}^{n}-\Gamma
        		_{t_{i-1}}^{n})\right)  \label{eq:defbEunderPtilde'} \\
        		\mathcal{\mathcal{K}}_{i}^{n} =&\exp \left( \sum_{j=1}^{i}\left( \frac{%
        			b_{j-1}}{\sigma _{j-1}}(R_{t_{j}}^{n}-R_{t_{j-1}}^{n})+\frac{b_{j-1}}{\sigma
        			_{j-1}^{2}}(B_{t_{i}}^{n}-B_{t_{i-1}}^{n})-\frac{1}{2}\left( \frac{b_{j-1}}{%
        			\sigma _{j-1}}\right) ^{2}\Delta \right) \right) .
        		\label{eq:defbKunderPtilde'}
        	\end{align}

        	We introduce now five lemmas to help us towards obtaining the limit of the
        	sequence $(R^{n},B^{n},\Gamma ^{n})$ and then, implicitly, the limit of the
        	sequence $ (X^n,E^n) $. Their proofs are provided in Section %
        	\ref{sec:gs}. Recall the notation in \eqref{eq:not} which is used in the
        	following result. In the statements, we use the notation:
        	\begin{align*}
        		\vartheta _{i-1}:=& \sigma _{i-1}\sqrt{\Delta }%
        		g_{i}(X_{i-1}^{L})-X_{i-1}^{L,\sigma }\bar{\Phi}\left( X_{i-1}^{L,\sigma
        		}\right) .
        	\end{align*}
        	
        	%\begin{lemma}
        	%\label{lem:3} We have that 
        	%\begin{eqnarray*}
        	%\mathbb{E}_{i-1}\left[ 1_{(X_{i}>L)}m_{i}\right] &=&\bar{\Phi}\left(
        	%-X_{i-1}^{L,\sigma }\right) -\bar{\Phi}\left( X_{i-1}^{L,\sigma }\right) <1
        	%\\
        	%\mathbb{E}_{i-1}\left[ 1_{(X_{i}>L)}\bar{m}_{i}\right] &=&\bar{\Phi}\left(
        	%-X_{i-1}^{L,\sigma }\right) +\bar{\Phi}\left( X_{i-1}^{L,\sigma }\right) =1.
        	%\end{eqnarray*}%
        	%In particular, this implies that the process $\bar{M}$ is a martingale,
        	%whilst the process $M$ is a supermartingale.
        	%\end{lemma}
        	
        	\begin{lemma}
        		\label{lem:333} Under $ \mathbb{Q}^n $ and for 
        		any $k\geq 0$, we have that 
        		\begin{equation}
        			\widetilde{\mathbb{E}}_{i-1}\left[ {Z}_{i}^{k}\bar{m}_{i}\right] =\frac{\Delta ^{%
        					\frac{k}{2}}}{\sqrt{2\pi }}\left( \int_{-X_{i-1}^{L,\sigma }}^{\infty
        			}x_{i}^{k}e^{-\frac{x_{i}^{2}}{2}}dx_{i}+\int_{X_{i-1}^{L,\sigma }}^{\infty
        			}\left( x_{i}-2X_{i-1}^{L,\sigma }\right) ^{k}e^{-\frac{x_{i}^{2}}{2}%
        			}dx_{i}\right). \label{ziid0}
        		\end{equation}
        		
        		Furthermore, the following moments estimates hold
        		\begin{align}
        			\widetilde{\mathbb{E}}_{i-1}\left[ {Z}_{i}\bar{m}_{i}\right] =&2\sqrt{\Delta}\vartheta _{i-1} ,\quad \widetilde{\mathbb{E}}_{i-1}\left[ {Z}_{i}^{2}\bar{m}_{i}\right] =\Delta
        			-4\Delta X_{i-1}^{L,\sigma }\vartheta _{i-1},\label{ziid} \\
        			0\leq \vartheta _{i-1}& \leq \sigma _{i-1}\sqrt{\Delta }g_{i}(X_{i-1}^{L}).\notag
        		\end{align}
        		
        		Finally, for any $m,p,k\in \mathbb{N}$, there exists a constant $%
        		c_{k}>0$ independent of $n$ such that 
        		\begin{align}
        			\widetilde{\mathbb{E}}_{i-1}\left[ {Z}_{i}^{2k}\bar{m}_{i}\right] \leq &
        			c_{k}\Delta ^{k},  \label{ziid3} \\
        			\widetilde{\mathbb{E}}_{i-1}\left[ |{Z}_{i}|^{k}(X_{i-1}-L)^{m}(X_{i}-L)^{p}{\
        				1_{(X_i>L)} }1_{(U_{i}\leq p_{i})}\right] \leq & c_{k,m,p}\Delta ^{\frac{%
        					k+m+p}{2}}\exp \left( -\frac{(X_{i-1}-L)^{2}}{2a_{i-1}\Delta }\right) .\notag
        			%       \label{ziid33}
        		\end{align}
        		
        		In particular, \footnote{%
        			Here and everywhere else $\bar{\Phi}$ is defined to be $\bar{\Phi}=1-\Phi $,
        			where $\Phi $ is the cumulative distribution function of the standard normal
        			distribution.} 
        		\begin{equation}
        			\begin{array}{lll}
        				\widetilde{\mathbb{E}}_{i-1}[m_{i}] & =\bar{\Phi}(-X_{i-1}^{L,\sigma })-\bar{\Phi%
        				}(X_{i-1}^{L,\sigma }) & <1 \\ 
        				\widetilde{\mathbb{E}}_{i-1}[\bar{m}_{i}] & =\bar{\Phi}(-X_{i-1}^{L,\sigma })+%
        				\bar{\Phi}(X_{i-1}^{L,\sigma }) & =1.%
        			\end{array}
        			\label{lem:3}
        		\end{equation}
        	\end{lemma}
        Note that the above result implies that the process $\bar{M}^n$ is a martingale,
        whilst the process $M^n$ is a supermartingale.
        	Furthermore, the above result is the base in order to obtain the following estimates which will be used towards the justification of the semimartingale character of the driving processes $R^{n}$ and $\Gamma ^{n}$.
        	\begin{lemma}
        		\label{cumulativecontrol}There exists a
        		constant $c=c\left( T\right) $ independent of $n$ and $ k $ such that 
        		\begin{equation}
        			\widetilde{\mathbb{E}}\left[ \left( \sum_{i=1}^{n}\widetilde{\mathbb{E}}_{i-1}\left[ 
        			{Z}_{i}\bar{m}_{i}\right] \right) \bar{M}^n_{n}\right] \leq c,~~~~\widetilde{%
        				\mathbb{E}}\left[ \max_{j\leq n-k-1}\left( \sum_{i=j}^{j+k}\widetilde{\mathbb{E}}%
        			_{i-1}\left[ {Z}_{i}\bar{m}_{i}\right] \right) \bar{M}^n_{n}\right] \leq c%
        			\sqrt{\left( k+1\right) \Delta }  \label{control1}
        		\end{equation}
        		In addition, the following bounds hold $\mathbb{Q}^n$-almost surely\footnote{%
        			Due to the absolute continuity of ${\mathbb{Q}}^n$ with respect to $\mathbb{P}$%
        			, bounds (\ref{control2}) and (\ref{control3}) also hold ${\mathbb{Q}}^n$%
        			-almost surely.} 
        		\begin{align}
        			\sum_{i=1}^{n}\widetilde{\mathbb{E}}_{i-1}\left[ {Z}_{i}^{2}\bar{m}_{i}\right]
        			\leq &c,\quad\quad \max_{j\leq n-k-1}\sum_{i=j}^{j+k}\widetilde{\mathbb{E}}_{i-1}\left[
        			{Z}_{i}^{2}\bar{m}_{i}\right] \leq c\left( k+1\right) \Delta .
        			\label{control2} \\
        			\sum_{i=1}^{n}\widetilde{\mathbb{E}}_{i-1}\left[ {Z}_{i}^{4}\bar{m}_{i}\right]
        			\leq &c\Delta, \quad\quad\max_{j\leq n-k-1}\sum_{i=j}^{j+k}\widetilde{\mathbb{E}}%
        			_{i-1}\left[ {Z}_{i}^{4}\bar{m}_{i}\right] \leq c\left( k+1\right) \Delta
        			^{2}.  \label{control3}
        		\end{align}
        	\end{lemma}
        	
        	When dealing with remainders it will be useful to introduce the concept often
        	asymptotically negligible in expectation.
        	\begin{definition}
        		\label{def:2u}    
        		Let $ \Upsilon_i\equiv \Upsilon^n_i(x)\in \mathcal{F}_{t_i} $ be a sequence of $ L^{2q} (\Omega)$ -integrable r.v.'s,
        		we say that
        		the family of random variables $\Upsilon_i\in \mathcal{F}_i $, $i=1,...,n 
        		$ is asymptotically negligible in expectation (under the measure $ \bar{M}^n $) of order $ \Delta^p $ if for any $ q\in\mathbb{N} $ and any compact set $ K\subseteq D $, we have
        		\begin{align}
        			\label{eq:neg}
        			\sup_{x\in K}\tilde{\mathbb{E}}_{0,x}\left[\left|\sum_{i=1}^n\Upsilon_i\right |^q\bar{M}^n_n\right]\leq C\Delta^{pq}.
        		\end{align}
        		Here, as in the case for ${\mathbb{E}}_{0,x}  $, we define $ \tilde{\mathbb{E}}_{0,x}\left[\cdot \right]= \mathbb{E}_{i}[\cdot|X_i=x] $
        		As notation, we will use $ \Upsilon_i=O^E_i(\Delta^p) $ or $ O^{E,\bar{M}}_i(\Delta^p) $ in the case the Radon-Nikodym r.v. $ \bar{M}^n $ is stressed. 
        		The usual big O notation will also be used as in $\Upsilon_i=O_i(\Delta^p) $. The latter means that $ \Upsilon_i\in\mathcal{F}_i $ and $ |\Upsilon_i|\leq C\Delta^p $. We may also abuse slightly this notation using $ \Upsilon_i=O_i(Z_i) $  to mean that $ |\Upsilon_i|\leq C|Z_i| $.           
        	\end{definition}
        	
        	\begin{remark}
        		\label{rem:11} Using the same method of proof as in Lemma \ref{lem:essb}
        		together with Lemmas \ref{lem:333} and \ref{cumulativecontrol} we
        		obtain $\mathbb{Q}^n$-almost surely that for $k,m\in \mathbb{N}$, 
        		\begin{eqnarray*}
        			\tilde{\mathbb{E}}_{i-1}[  (X_{i-1}-L)^{k}1_{(U_{i}\leq p_{i})}] &=&O_{i-1}^{E}(\Delta ^{(k-1)/2}), \\
        			|Z_{i}|^{k+1} &=&O_{i}^{E}(\Delta ^{(k-1)/2}), \\
        			\tilde{\mathbb{E}}_{i-1}[   |Z_{i}|^{k+1}(X_{i-1}-L)^{m}1_{(U_{i}\leq p_{i})}] &=&O_{i-1}^{E}(\Delta
        			^{(k+m-1)/2}) \\
        			\tilde{\mathbb{E}}_{i-1}[|Z_{i}|^{k+1}(X_{i}-L)^{m}1_{(U_{i}\leq p_{i})}] &=&O_{i-1}^{E}(\Delta
        			^{(k+m-1)/2}).
        		\end{eqnarray*}%
        		Also note that $Z_{i}1_{(U_{i}\leq p_{i})}=\frac{X_{i}-L}{\sigma _{i-1}}%
        		1_{(U_{i}\leq p_{i})}-\frac{X_{i-1}-L}{\sigma _{i-1}}1_{(U_{i}\leq p_{i})}-%
        		\frac{b_{i-1}}{\sigma _{i-1}}\Delta 1_{(U_{i}\leq p_{i})}$. Therefore its order can be deduced from the above list.
        	\end{remark}
        	The next lemma gives the limit of the quadratic variation of the driving process $ R^n $.

        	\begin{lemma}
        		\label{lemmaBm} There exists a stochastic process $\theta ^{R}$ that
        		satisfies the control 
        		\begin{equation}
        			\sup_{t\in \left[ 0,T\right] }\widetilde{\mathbb{E}}\left[ \left\vert \theta
        			_{t}^{R}\right\vert \bar{M}^{n}_n\right] \leq c\sqrt{\Delta} \   \label{Bmcontrol}
        		\end{equation}%
        		where the constant $c=c\left( T\right) $ is independent of $n$ such that 
        		\begin{equation*}
        			\sum_{\left\{ i>0,t_{i}\leq t\right\} }(R_{t_{i}}^{n}-R_{t_{i-1}}^{n})^{2}= 
        			\left[ \frac{t}{\Delta }\right] \Delta +\theta _{t}^{R}.
        		\end{equation*}%
        		Similarly, we have that 
        		\begin{equation}
        			\sup_{t\in \left[ 0,T\right] }\widetilde{\mathbb{E}}\left[ \theta _{t}^{\Gamma }%
        			\bar{M}^{n}_n\right] \leq c\sqrt{\Delta },  \label{Rmcontrol}
        		\end{equation}%
        		where 
        		\begin{equation*}
        			\theta _{t}^{\Gamma }:=\sum_{\left\{ i>0,t_{i}\leq t\right\} }(\Gamma
        			_{t_{i}}^{n}-\Gamma _{t_{i-1}}^{n})^{2}.
        		\end{equation*}
        	\end{lemma}
        	
        	%In the language of definition introduced on page 55 of \cite{JacodProtter},
        	%the sequences $\left( \widetilde{\mathbb{E}}\left[ {Z}_{i}^{4}\bar{ m}_{i}\right]
        	%\right) _{i}$ and $\theta _{i}$ as defined in (\ref{thetai}) are
        	%asymptotically negligible, or AN in short.
        	
        	%\subsection{Uniform bounds for moments and derivatives}
        	We need to be able to control uniformly the moments of the various processes appearing
        	in (\ref{fprimetilde1}), respectively (\ref{fprmetilde2}). We do this in the
        	following lemma: 
        	%\textcolor{red}{ I think it may be useful to have an estimate for   }
        	%\begin{align*}
        	%       &\mathbb{{E}}\left[ \max_i\left( E_{i:n}^{n}\right)
        	%       ^{k}\bar{M}_{i:n}^{n}|X_{0}^{n}=x\right] \\
        	%       &\mathbb{{E}}\left[ \max_i\left( E_{i}^{n}\right)
        	%       ^{k}\bar{M}_{i}^{n}|X_{0}^{n}=x\right]
        	%\end{align*}
        	
        	\begin{lemma}
        		\label{moments}For arbitrary $p\geq 1$,  $\sup_{x\geq L}\widetilde{\mathbb{E}}_{i,x}\left[
        		\max_{j\geq i}\left( \bar{M}_{i:j}^{n}\right) ^{p}\right]^{1/p}\leq 2^n $ 
        		and there exists a constant $\mathsf{C}_p >0$ independent of $i$ and $n$ such that 
        		the following moment property is satisfied for $\mathsf{E}=\mathcal{K}^n,
        		E^n, $ 
        		\begin{equation}
        			\sup_{x\geq L}\widetilde{\mathbb{E}}_{i,x}\left[ \max_{j\geq
        				i}\left(e^{-\mathsf{C}_p(t_j-t_i)} \mathsf{E}_{i:j}\right) ^{p}\bar{M}^n_n\right]^{1/p} \leq { \mathsf{C}_p  }.
        			\label{Lnbounds}
        		\end{equation}
        		The same bound is satisfied for $\mathsf{E}=X^{n} $ but the supremum has to
        		be restricted to a compact set.
        	\end{lemma}
        	The constant $ 2 $ appearing in the above estimates is probably non-optimal but it makes proofs easier.
        	%{\color{blue}
        	%       Here, we need to deal with $ \mathcal{K}^n $ and rewrite the estimates as
        	%       \begin{align*}
        	%         \sup_{x\geq L}\widetilde{\mathbb{E}}_{i,x}\left[ \max_{j\geq
        	%               i}\left(e^{-\mathsf{C}_p(t_j-t_i)} \mathsf{E}_{i:j}\right) ^{p}\bar{M}_n\right]^{1/p} \leq { \mathsf{C}_p  }.
        	%       \end{align*}
        	%Then the estimate will follow from the moment estimate Lemma using the expansion $ e^{\kappa^n_i}-1\approx \kappa^n_i+\frac 12(\kappa^n_i)^2+ $
        	%The properties of $ Z_i $ under $ \bar {m}_i$ are known so that one can build the martingale and drift terms. The issue is that here approximations of local time terms will appear and we need to rewrite the moment lemma to adapt to this. For $ e_i $ is done in the usual way using again properties under $ \bar {m}_i$ which again leads to local time terms.
        	%}

        	Now, we proceed with a uniform control of the derivative of the functions $%
        	f_{i}$.
        	
        	\begin{lemma}
        		\label{boundedderivative}Assuming that $f\in C_{b}^{1}([L,\infty ))$ such
        		that $f(L)=0$ and 
        		%\textcolor{red}{ that $f^{\prime}$ is Lipschitz continuous on $[L,\infty)$}, then
        		there exists a constant $\mathsf{C}$ which  depends on $\|f\|_\infty+\|f'\|_\infty $, but is independent of $i$ and $n $ such that 
        		\begin{equation*}
        			\|\partial _{x}f_{i}\| _{\infty }\leq \mathsf{C}e^{\mathsf{C}(T-t_i)}.
        		\end{equation*}
        	\end{lemma}
        	The above two technical lemmas are proven in the Appendix.
        	\section{Convergence of the processes $({X}^{n},R^{n},B^n, \Gamma^n, E^{n},\mathcal{K} ^n)$}
        	
        	Let us recall the definition of the various processes involved in the
        	approximate representation of the derivative of the killed semigroup. In the list below, we review the discrete time versions of the processes, i.e, we specify the processes only at the times $t_i$, $i=0,1,...,n$. However, in the subsequent analysis and whenever needed, we will work with their standard imbedding  of the discrete time versions into the piece-wise constant continuous time processes. In particular, we will show that the piece-wise constant version of the processes $({X}^{n},R^{n},B^n, \Gamma^n, E^{n},\mathcal{K} ^n)$ converge in distribution.
        	
        	\begin{itemize}
        		\item $X^n$ satisfies the recurrence formula 
        		\begin{align}
        			X_{i}^{n,x} =&X_{i-1}^{n,x}+\sigma _{i-1}{Z}_{i}  \notag \\
        			=& X_{i-1}^{n,x}+\sigma _{i-1}({Z}_{i}-\widetilde{\mathbb{E}}_{i-1}\left[ {Z}_{i}\bar{%
        				m}_{i}\right]) +\sigma _{i-1}\widetilde{\mathbb{E}}_{i-1}\left[ {Z}_{i}\bar{m}%
        			_{i}\right] \notag \\
        			=&X_{i-1}^{n,x}+\sigma
        			_{i-1}(R_{t_i}^{n}-R_{t_{i-1}}^{n})+(B _{t_i}^{n}-B _{t_{i-1}}^{n}),  \label{eq:defbXunderPtilde''}
        		\end{align}
        		%\textcolor{red}{Add continuous version}
        		We will show that $X^n$ converges in distribution to a reflected diffusion $%
        		Y $ as defined in (\ref{y1}), see also (\ref{reflected}) below.
        		
        		\item $R^n$ as defined in (\ref{rn}) is the martingale driving the
        		stochastic flow of $X^n$. We will show that $R^n$ converges in distribution
        		to a Brownian motion.
        		
        		\item $\Gamma^{n}$ as defined in (\ref{rn}) is an additional
        		martingale driving the process $E^n$. We will show that $\gamma^n$ vanishes
        		in the limit.
        		
        		\item $B^n$ as defined in (\ref{Bn}) is the predictable part in the
        		(discrete) Doob-Meyer decomposition of $X^n$ under the measure ${\mathbb{Q}}^n$. We will prove that $B^n$ converges in distribution to a local time-type (a.k.a. regulator)
        		process.
        		
        		\item $E^n$ as defined in (\ref{eq:defbEunderPtilde'}) is the derivative of
        		the ``stochastic flow" of the killing of $X^n$. We will prove that $E^n$ converges in
        		distribution to the process $\xi$ defined in (\ref{xi1}) (see also (\ref{exp}%
        		) below).
        		
        		\item $\mathcal{K}^n$ as defined in (\ref{eq:defbKunderPtilde'}) is the
        		Radon-Nikodym derivative of the measure $\widetilde{\mathbb{P}} $ (this is the measure
        		under which $X^n$ has no drift) with respect to the original measure $\mathbb{P}. $
        		We will prove that $\mathcal{K}^n$ converges in distribution to the process $%
        		\mathcal{K}$ defined in (\ref{G}).
        	\end{itemize}
        	
        	The processes $({X}^{n},E^{n},\mathcal{K} ^n)$ are driven/generated by $%
        	(R^{n},B^n, \Gamma^n)$. Note that $(R^{n},B^n, \Gamma^n)$ as~well as $\bar{M}^n_{n}$ have moments of all orders. Therefore so have $({X}^{n},E^{n},\mathcal{%
        		K} ^n)$. As a result $({X}^{n},R^{n},B^n, \Gamma^n, E^{n},\mathcal{K} ^n)$
        	are integrable under ${\mathbb{Q}}^n$.
        	
        	As stated above, the sequence $(X^n,B^{n})$ will converge to the reflected
        	process $(Y,B)$ defined in \eqref{y1} if $b=0$ or that defined in %
        	\eqref{eq:rep} if $b\ne 0$. If $b=0$, the sequence $E^n$ converges in
        	distribution to the exponential martingale $\xi$ appearing in formula %
        	\eqref{eq:-1} as defined in \eqref{xi1}. Moreover, we have that $\mathcal{K}%
        	^n\equiv 0$ (the measures $\mathbb{P}$ and $\widetilde{\mathbb{P}}^n$ coincide). If $b\ne 0$, the
        	sequence $(E^n,\mathcal{K}^n)$ converges in distribution and the
        	product $E^n\mathcal{K}^n$ converges in distribution to the semi-martingale $%
        	\mathcal{E}$ appearing in formula \eqref{eq:0} as defined in the introduction. The
        	sequence $R^{n}$ will converge to the driving Brownian motion $W$ appearing
        	in both \eqref{eq:-1} and \eqref{eq:0}. The sequence $\Gamma^{n}$ will
        	vanish in the limit. We justify all of these limits in what follows. We
        	start with the processes $(R^{n},\Gamma^{n})$:
        	
        	\begin{theorem}
        		\label{WandF} The sequence $(R^{n},\Gamma^{n})$ is a sequence of square integrable
        		martingales under ${\mathbb{Q}}^n$ which converges in distribution to $(W,0)$%
        		, where $W$ is a Brownian motion.
        	\end{theorem}
        	
        	\begin{proof}
        		We deduce the martingale property from \eqref{rn}, \eqref{eq:gammainZ} and Lemma  \ref{lem:333}. In fact,
        		\begin{align*}
        			\widetilde{\mathbb{E}}_{i-1}[(R^n_{t_i}-R^n_{t_{i-1}})\bar{m}_{i}]=&\widetilde{%
        				\mathbb{E}}_{i-1}[(Z_{i}-\widetilde{\mathbb{E}}_{i-1}[Z_{i}\bar{m}_{i}])\bar{m}%
        			_{i}]=0 \\
        			\widetilde{\mathbb{E}}_{i-1}[(\Gamma^n_{t_i}-\Gamma^n_{t_{i-1}})\bar{m}_{i}]=& 
        			\widetilde{\mathbb{E}}_{i-1}[\gamma _{i}\bar{m}_{i}] =0.
        		\end{align*}
        		%               \begin{eqnarray*}
        		%                       
        		%%                      \\
        		%%                      &=&\sigma _{i-1}^{\prime }\left( 2\sigma _{i-1}\Delta g_{i}(X_{i-1}^{L})-2%
        		%%                      \sqrt{\Delta }X_{i-1}^{L,\sigma }\bar{\Phi}\left( X_{i-1}^{L,\sigma }\right)
        		%%                      \right) \\
        		%%                      &&-2\sigma _{i-1}^{\prime }\left( \sigma _{i-1}\Delta g_{i}(X_{i-1}^{L})-2%
        		%%                      \sqrt{\Delta }X_{i-1}^{L,\sigma }\bar{\Phi}\left( X_{i-1}^{L,\sigma }\right)
        		%%                      \right) \\
        		%%                      &&-\sigma _{i-1}^{\prime }\left( 2\sqrt{\Delta }X_{i-1}^{L,\sigma }\bar{\Phi}%
        		%%                      \left( X_{i-1}^{L,\sigma }\right) \right) =0.
        		%               \end{eqnarray*}
        		In the last equality we have also used that $ \widetilde{\mathbb{E}}[Z_i1_{(X_i>L)}]=\sigma_{i-1}\Delta g_{i-1}(X_{i-1}^L) $. The square integrability of the processes $(R^{n},\Gamma^{n})$ follows from
        		the controls in Lemma \ref{cumulativecontrol}. The convergence in
        		distribution follows by applying Theorem 2.2.13 page 56 from \cite%
        		{JacodProtter} by using the controls in Lemma \ref{lemmaBm} and (\ref%
        		{control3}). In other words we deduce that $\left( W^{n},\Gamma ^{n}\right) $
        		converges in distribution to a two dimensional Gaussian process with
        		independent increments and with covariation process 
        		\begin{equation*}
        			\left( 
        			\begin{array}{cc}
        				t & 0 \\ 
        				0 & 0%
        			\end{array}
        			\right) .
        		\end{equation*}
        		which gives our claim.
        		%similarly with the analysis of (\ref{int2}) and (\ref{int3}), we
        		%deduce that 
        		%\begin{equation}
        		%\sum_{\left\{ i,t_{i}\leq t\right\} }\widetilde{\mathbb{E}}\left[ \left( 1_{\left(
        		%U_{i}\leq p_{i}\right) }\left( \frac{X_{i}-L}{\sigma _{i-1}}\right) \right)
        		%^{2}\bar{M}_{t}\right] \leq C\sum_{i}\Delta \widetilde{\mathbb{E}}\left[ e^{-\frac{%
        		%\left( \mathcal{X}_{i-1}-L\right) ^{2}}{2a\left( \mathcal{X}_{i-1}\right)
        		%\Delta }}\right] \leq C\sqrt{\Delta }.  \label{control4}
        		%\end{equation}
        	\end{proof}
        	
        	Next, we move on to the remaining four sequences. 
        	\begin{theorem}
        		\label{relcomptheorem}The sequence of processes $({X}^{n},B^n, E^{n},%
        		\mathcal{K} ^n)$ is relatively compact under ${\mathbb{Q}}^n$
        	\end{theorem}
        	
        	\begin{proof}
        		We apply here{ Theorem 8.6  and Remark 8.7 in Chapter 3  }from \cite{EthierKurtz}. In particular, we use the
        		estimates in Lemma \ref{lemmaBm} and (\ref{control3}). So we need to prove
        		that 
        		\begin{equation*}
        			\sup_{n}\widetilde{\mathbb{E}}\left[ \left\vert q_{t}^{n}\right\vert \bar{M}^n_{n}%
        			\right] <\infty ,
        		\end{equation*}%
        		where $q^{n}$ is replaced by each of the processes in the quadruplet $({X}%
        		^{n},B^{n},E^{n},\mathcal{K}^{n})$ in turn. The control for the processes in
        		the triplet $({X}^{n},E^{n},\mathcal{K}^{n})$ follows from Lemma \ref{moments}. The bounds on the moments of $(B^{n})$ follows from \eqref{Bn} and Lemma \ref%
        		{cumulativecontrol}.
        		
        		Next, we need to find a family $(\gamma _{\delta }^{n})_{0<\delta <1, n\in 
        			\mathbb{N}}$ of nonnegative random variables such that  
        		\begin{equation}\label{filt}
        			\widetilde{\mathbb{E}}\left[ \left\vert q_{t+u}^{n}-q_{t}^{n}\right\vert \bar{M}%
        			^n_{n}|\mathcal{F}_{t}^{n}\right] \leq \widetilde{\mathbb{E}}\left[ \gamma
        			_{\delta }^{n}\bar{M}^n_{n}|\mathcal{F}_{t}^{n}\right]
        		\end{equation}%
        		and $\lim_{\delta \rightarrow 0}\limsup_{n\rightarrow \infty }\widetilde{\mathbb{%
        				E}}\left[ \gamma _{\delta }^{n}\bar{M}^n_{n}\right] =0$ for $t\in \lbrack
        		0,T],u\in \lbrack 0,\delta ]$, where $q^{n}$ is replaced by $X^{n},B^{n}$, $%
        		E^{n}$ and $\mathcal{K}^{n}.$ {As usual, the continuous filtration in \eqref{filt} is identified as $\mathcal{F}_{t}^{n}:=\mathcal{F}_{[nt]\over n}$.}
        		
        		We take then in turn:\\[1mm]
        		$\bullet $ $q^{n}=B^{n}$. For $t\in \lbrack 0,T],u\in \lbrack 0,\delta ]$ we
        		have that 
        		\begin{align*}
        			\left\vert B_{t+u}^{n}-B_{t}^{n}\right\vert \leq \Vert \sigma \Vert_\infty
        			\sum_{\left\{ i,t_{i}\in \left[ t,t+\delta \right] \right\} }\widetilde{\mathbb{E%
        			}}_{i-1}[Z_{i}\bar{m}_{i}].
        		\end{align*}
        		The obvious choice for $\gamma _{\delta }^{n}$ is%
        		\begin{equation*}
        			\gamma _{\delta }^{n}=\max_{j\leq n-\left[ \frac{\delta }{\Delta }\right]
        				-2}\left( \sum_{i=j}^{j+\lceil \frac{\delta }{\Delta }\rceil}\widetilde{%
        				\mathbb{E}}_{i-1}[Z_{i}\bar{m}_{i}]\right)
        		\end{equation*}%
        		with the control following from (\ref{control1})%
        		\begin{equation*}
        			\widetilde{\mathbb{E}}\left[ \gamma _{\delta }^{n}\bar{M}^n_{n}\right] \leq c\sqrt{%
        				\left( \left[ \frac{\delta }{\Delta }\right] +2\right) \Delta }\leq c\sqrt{%
        				\delta +2\Delta }=c\sqrt{\delta +2\frac{T}{n}}
        		\end{equation*}%
        		and obviously $\lim_{\delta \rightarrow 0}\limsup_{n\rightarrow \infty }%
        		\widetilde{\mathbb{E}}\left[ \gamma _{\delta }^{n}\bar{M}^n_{n}\right] =0$.\\[2mm]
        		$\bullet $ $q^{n}=X^{n}$. For $t\in \lbrack 0,T],u\in \lbrack 0,\delta ]$ we
        		have that 
        		\begin{equation*}
        			\left\vert X_{t+u}^{n}-X_{t}^{n}\right\vert \leq \left\vert \sum_{\left\{
        				i,t_{i}\in \left[ t,t+u\right] \right\}
        			}(B_{t_{i}}^{n}-B_{t_{i-1}}^{n})\right\vert
        			+\left\vert \sum_{\left\{ i,t_{i}\in \left[
        				t,t+u\right] \right\} }\sigma
        			_{i-1}(R_{t_{i}}^{n}-R_{t_{i-1}}^{n})\right\vert
        		\end{equation*}%
        		where the first term is controlled from the computation for $q^{n}=B^{n}$.

        		The second term is controlled as follows (we use the martingale property of $%
        		R^{n}$) 
        		\begin{eqnarray*}
        			&&	\tilde{\mathbb{E}}\left[ \left. \left\vert \sum_{\left\{ i,t_{i}\in \left[
        				t,t+u\right] \right\} }\sigma
        			_{i-1}(R_{t_{i}}^{n}-R_{t_{i-1}}^{n})\right\vert \bar{M}_{n}^{n}\right\vert 
        			\mathcal{F}_{t}^{n}\right] ^{2}\\
        			&\leq &\tilde{\mathbb{E}}\left[ \left.
        			\left\vert \sum_{\left\{ i,t_{i}\in \left[ t,t+u\right] \right\} }\sigma
        			_{i-1}(R_{t_{i}}^{n}-R_{t_{i-1}}^{n})\right\vert ^{2}\bar{M}%
        			_{n}^{n}\right\vert \mathcal{F}_{t}^{n}\right]  \\
        			&\leq &\tilde{\mathbb{E}}\left[ \left. \sum_{\left\{ i,t_{i}\in \left[ t,t+u%
        				\right] \right\} }\sigma _{i-1}^{2}(R_{t_{i}}^{n}-R_{t_{i-1}}^{n})^{2}\bar{M}%
        			_{n}^{n}\right\vert \mathcal{F}_{t}^{n}\right]  \\
        			&\leq &\left\vert \left\vert \sigma \right\vert \right\vert_\infty ^{2}\tilde{%
        				\mathbb{E}}\left[ \left. \sum_{\left\{ i,t_{i}\in \left[ t,t+\delta \right]
        				\right\} }(R_{t_{i}}^{n}-R_{t_{i-1}}^{n})^{2}\bar{M}_{n}^{n}\right\vert 
        			\mathcal{F}_{t}^{n}\right].
        		\end{eqnarray*}%
        		From Lemma \ref{lemmaBm} we deduce that 
        		\begin{equation*}
        			\sum_{\left\{ i,t_{i}\in \left[ t,t+\delta \right] \right\}
        			}(R_{t_{i}}^{n}-R_{t_{i-1}}^{n})^{2}\leq \left( \left[ \frac{t+\delta }{%
        				\Delta }\right] -\left[ \frac{t}{\Delta }\right] \right) \Delta +\theta
        			_{t+\delta }^{R}-\theta _{t}^{R}\leq \left( \delta +\frac{T}{n}\right)
        			+2\sup_{t\in \left[ 0,T+\delta\right] }\theta _{t}^{R}.
        		\end{equation*}
        		
        		In this case, we choose 
        		\begin{equation*}
        			\gamma _{\delta }^{n}:=\|\sigma\|_\infty\left( \sup_{t\in \left[ 0,T\right] }\tilde{\mathbb{E}%
        			}\left[ \left. \left( \left( \delta +\frac{T}{n}\right) +2\sup_{t\in \left[
        				0,T+\delta \right] }\theta _{t}^{R}\right) \bar{M}_{n}^{n}\right\vert 
        			\mathcal{F}_{t}^{n}\right] \right) ^{\frac{1}{2}}(\bar{M}_{n}^{n})^{-1},
        		\end{equation*}%
        		which gives the correct control in \eqref{filt} by using (\ref{Bmcontrol'}).

        		$\bullet $ $q^{n}=E^{n}$. We have that for $ S(R,\Gamma)_{i}^{n}:=\sigma _{i-1}^{\prime
        		}(R_{t_{i}}^{n}-R_{t_{i-1}}^{n})+(\Gamma _{t_{i}}^{n}-\Gamma
        		_{t_{i-1}}^{n})$, $i=1,...,n$.
        		\begin{eqnarray*}
        			\left\vert E_{t+u}^{n}-E_{t}^{n}\right\vert &\leq &
        			\Delta \|\bar{b}\|_\infty u\max_{\left\{ i,t_{i}\in \left[ t,t+u\right] \right\}}|E_{i-1}^{n}| +\left\vert \sum_{\left\{ i,t_{i}\in \left[ t,t+u\right] \right\}
        			}E_{i-1}^{n}S(R,\Gamma)_{i}^{n}
        			\right\vert.
        		\end{eqnarray*}%
        		The first term is controlled using Lemma \ref{moments}. For the second one, consider
        		\begin{equation*}
        			\widetilde{\mathbb{E}}\left[ \left. \left\vert \sum_{\left\{ i,t_{i}\in \left[
        				t,t+u\right] \right\} }E_{i-1}^{n}S(R,\Gamma) _{i}^{n}\right\vert \bar{M}^n%
        			_{n}\right\vert \mathcal{F}_{t}^{n}\right].
        		\end{equation*}
        		Using the fact that $\widetilde{\mathbb{E}}_{i-1}\left[ S(R,\Gamma)_{i}^{n}\bar{m%
        		}_{i}\right] =0$ we deduce that 
        		\begin{eqnarray*}
        			&&\tilde{\mathbb{E}}\left[ \left. \left\vert \sum_{\left\{ i,t_{i}\in \left[
        				t,t+u\right] \right\} }E_{i-1}^{n}S(R,\Gamma )_{i}^{n}\right\vert \bar{M}%
        			_{n}^{n}\right\vert \mathcal{F}_{t}^{n}\right] ^{4} \\
        			&&\hspace{2cm}\leq \tilde{\mathbb{E}}\left[ \left. \left( \sum_{\left\{
        				i,t_{i}\in \left[ t,t+u\right] \right\} }E_{i-1}^{n}S(R,\Gamma
        			)_{i}^{n}\right) ^{2}\bar{M}_{n}^{n}\right\vert \mathcal{F}_{t}^{n}\right]
        			^{2} \\
        			&&\hspace{2cm}=\tilde{\mathbb{E}}\left[ \left. \left( \sum_{\left\{
        				i,t_{i}\in \left[ t,t+u\right] \right\} }\left( E_{i-1}^{n}\right)
        			^{2}\left( S(R,\Gamma )_{i}^{n}\right) ^{2}\right) \bar{M}%
        			_{n}^{n}\right\vert \mathcal{F}_{t}^{n}\right] ^{2} \\
        			&&\hspace{2cm}\leq \tilde{\mathbb{E}}\left[ \max_{\left\{ i,t_{i}\in \left[
        				t,t+\delta \right] \right\} }\left( E_{i}^{n}\right) ^{2}\left. \left(
        			\sum_{\left\{ i,t_{i}\in \left[ t,t+\delta \right] \right\} }\left(
        			S(R,\Gamma )_{i}^{n}\right) ^{2}\right) \bar{M}_{n}^{n}\right\vert \mathcal{F%
        			}_{t}^{n}\right] ^{2} \\
        			&&\hspace{2cm}\leq \tilde{\mathbb{E}}\left[ \max_{\left\{ i,t_{i}\in \left[
        				t,t+\delta \right] \right\} }\left( E_{i}^{n}\right) ^{4}\bar{M}_{n}^{n}|%
        			\mathcal{F}_{t}^{n}\right] \times \tilde{\mathbb{E}}\left[ \left. \left(
        			\sum_{\left\{ i,t_{i}\in \left[ t,t+\delta \right] \right\} }\!\!\!\!\!\!\left(
        			S(R,\Gamma )_{i}^{n}\right) ^{2}\right) ^{2}\!\!\!\!\bar{M}_{n}^{n}\right\vert 
        			\mathcal{F}_{t}^{n}\right].
        		\end{eqnarray*}%
        		
        		We choose then%
        		\begin{align*}
        			&\gamma _{\delta }^{n} :=2\Delta \Vert \bar{b}\Vert _{\infty }\sup_{i\geq
        				0}\left\vert E_{0:i}^{n}\right\vert  \\
        			&+\!\!\!\!\sup_{t\in \left[ 0,T\right] }\!\!\!\!\tilde{\mathbb{E}}\left[ \max_{\left\{
        				i=0,...,n\right\} }\left( E_{i}^{n}\right) ^{4}\bar{M}_{n}^{n}|\mathcal{F}%
        			_{t}^{n}\right] ^{\frac{1}{4}}\!\!\!\!\sup_{t\in \left[ 0,T\right] }\tilde{\mathbb{E}%
        			}\left[ \left. \left( \sum_{\left\{ i,t_{i}\in \left[ t,t+\delta \right]
        				\right\} }\!\!\!\!\!\!\left( S(R,\Gamma )_{i}^{n}\right) ^{2}\right) ^{2}\!\!\!\!\bar{M}%
        			_{n}^{n}\right\vert \mathcal{F}_{t}^{n}\right] ^{\frac{1}{4}}\!\!\!\!\left( \bar{M}%
        			_{n}^{n}\right) ^{-1}\!\!\!\!,
        		\end{align*}
        		and the result follows from the bounds in Lemma \ref{moments}, by using \eqref%
        		{Bmcontrol'} and \eqref{bmcontrol''} similar to the case $q^{n}=X^{n}.$\\[2mm]

        		$\bullet $ $q^{n}=\mathcal{\mathcal{K}}^{n}$. We have that 
        		\begin{equation*}
        			\left\vert \mathcal{\mathcal{K}}_{t+u}^{n}-\mathcal{\mathcal{K}}%
        			_{t}^{n}\right\vert =\left\vert \sum_{\left\{ i,t_{i}\in \left[ t,t+u\right]
        				\right\} }^{n}\mathcal{\mathcal{K}}_{i-1}^{n}\left( \exp (\bar{S}(R,B)
        			_{i}^{n})-1\right) \right\vert ,
        		\end{equation*}%
        		where%
        		\begin{equation*}
        			\bar{S}(R,B) _{i}^{n}:=\frac{b_{i-1}}{\sigma _{i-1}}%
        			(R_{t_{i}}^{n}-R_{t_{i-1}}^{n})+\frac{b_{i-1}}{a_{i-1}}%
        			(B_{t_{i}}^{n}-B_{t_{i-1}}^{n})-\frac{1}{2}\left( \frac{b_{i-1}}{\sigma
        				_{i-1}}\right) ^{2}\Delta .
        		\end{equation*}%
        		Observing that $\left\vert \widetilde{\mathbb{E}}_{i-1}\left[ \exp (
        		S(R,B)_{i}^{n})\bar{m}_{i}\right] -1\right\vert \leq C\Delta $, we deduce
        		the result with arguments similar as in the case $q^{n}=E^{n}$. This concludes the proof of
        		the theorem.
        	\end{proof}
        	
        	Now we can put all the above results together to obtain the first convergence characterization result for the probabilistic representation of the first derivative of the killed diffusion process.
        	\begin{theorem}
        		\label{bigT}The law of $({X}^{n},B^{n},E^{n},\mathcal{K}^{n},R^{n},\Gamma
        		^{n})$ under ${\mathbb{Q}}^n$ converges to the law of the process $(Y,B,\xi ,%
        		\mathcal{K},W,0)$. Here $W$ is a standard Brownian motion, $\xi $ and $%
        		\mathcal{K}$ satisfy 
        		\begin{align}
        			\xi _{t}=& 1+\int_0^t\bar{b}(Y_s)\xi_sds+\int_{0}^{t}\sigma ^{\prime }(Y_{s})\xi _{s}dW_{s}  \label{exp}
        			\\
        			\mathcal{K}_{t}=& \exp \left( \int_{0}^{t}ba^{-1}(Y_{s})dY_{s}-\frac{1}{2}%
        			\int_{0}^{t}b^{2}a^{-1}(Y_{s})ds\right)  \label{G}
        		\end{align}%
        		and $(Y,B)$ is the solution of the reflected equation in the domain $%
        		[L,\infty )$. 
        		\begin{align}
        			Y_{t}=&x+\int_{0}^{t}\sigma (Y_{s})dW_{s}+B_{t}, \label{reflected}\\
        			B_t=&\int_0^t1_{(Y_s=L)}d|B|_s.  \notag
        		\end{align}
        	\end{theorem}
        	
        	\begin{proof}
        		From Theorems \ref{WandF} and \ref{relcomptheorem}, the law of the vector  process $({%
        			X}^{n},B^{n},E^{n},\mathcal{K}^{n},R^{n},\Gamma ^{n})$ is relatively
        		compact. We extract a convergent (in law) subsequence from $({X}%
        		^{n},B^{n},E^{n},\mathcal{K}^{n},R^{n},\Gamma ^{n})$ which we re-index
        		and denote its limit as $(Y,B,\xi ,\mathcal{K},W,0).$ Following from (\ref
        		{eq:defbXunderPtilde''}) and (\ref{eq:defbEunderPtilde'}) we deduce that, for  $t\in [t_{i},t_{i+1})$ 
        		\begin{align*}
        			X_{t}^{n} =&X_{t_i}^{n}\notag\\
        			=&x_{0}+\sum_{j=0}^{i-1}\left(\sigma (X_{t_{j}}^{n}) (R_{t_{j+1}}-R_{t_{j}}^{n})+(B_{t_{j+1}}^{n}-B_{t_{j}}^{n})\right),\notag\\
        			=&x_{0}+\int_{0}^{t}\sigma ^{n}(X^{n},s)dR_{s}^{n}+B_{t}^{n},\notag\\
        			E_{t}^{n} =&E_{t_i}^n\\
        			=&1+\sum_{j=0}^{i-1}E_{j}^{n}\left( 
        			\bar{b}_{j}\Delta
        			1_{(U_{j+1}>p_{j+1})}+\sigma _{j}^{\prime
        			}(R_{t_{j+1}}^{n}-R_{t_{j}}^{n})+(\Gamma _{t_{j+1}}^{n}-\Gamma
        			_{t_{j}}^{n})\right),\notag\\
        			=&1+U^n_t+\int_0^t \widetilde{b}^n(E^{n},X^{n},s)ds +\int_{0}^{t}\widetilde{\sigma}^{n}(E^{n},X^{n},s)dR_{s}^{n}+%
        			\int_{0}^{t}E_{s}^{n}d\Gamma _{s}^{n},  
        			%\label{eq:E}
        		\end{align*}
        		where 
        		\begin{eqnarray*}
        			\sigma ^{n}(X^{n},s) &=&\sigma (X_{t_{j}}^{n}),~~~~~s\in \lbrack
        			t_{j}^{n},t_{j+1}^{n}) \\
        			\widetilde{\sigma}^{n}(E^{n},X^{n},s) &=&\sigma_j' (X_{t_{j}}^{n})E_{t_{j}}^{n},~~~~~s\in \lbrack
        			t_{j}^{n},t_{j+1}^{n}) \\
        			\widetilde{b}^{n}(E^{n},X^{n},s) &=&\bar{b}_j ^{
        			}(X_{t_{j}}^{n})E_{t_{j}}^{n}, ~~~~~s\in \lbrack t_{j}^{n},t_{j+1}^{n})\\
        			U^n_t &=& \sum_{j=0}^{i-1}E_{j}^{n} 
        			\bar{b}_{j}\Delta
        			1_{(U_{j+1}\le p_{j+1})}-\int_{t_i}^t \widetilde{b}^n(E^{n},X^{n},s)ds,\ \ \  ~~~~~t\in \lbrack t_{i}^{n},t_{i+1}^{n}).
        		\end{eqnarray*}
        		Again in the above identities we used the piecewise constant versions of the processes involved.

        		By applying Corollary 5.6 in \cite{KurtzProtter} and the fact that $({X}%
        		^{n},B^{n},E^{n},R^{n},\Gamma ^{n})$ converges to $(Y,B,\xi ,W,0)$ and that the process $U^n$ vanishes in the limit, we
        		deduce that $Y$ satisfies equation \eqref{reflected}, where $B$ is an
        		increasing process. A similar argument is applied to the convergence of $%
        		\mathcal{K}_{n}$ to $\mathcal{K}$. It remains to prove next that $(Y,B)$ is
        		the reflected diffusion in the domain $[L,\infty )$. For this it suffices to
        		show that $Y_{t}\geq L$ and that $\int_{0}^{t}1_{(Y_s>L)} dB_{s}=0,t\geq 0$. Then the result will follow, for
        		example, from Theorem 1.2.1 in \cite{ap}. Let $\varphi _{m}$ be the positive
        		continuous function with support in the interval $(-\infty ,L)$ given by 
        		%\begin{equation*}
        		$\varphi _{m}\left( x\right) =\min \left( 1,-m\left( x-L\right) \right)1_{(x<L)}$, $x\in \mathbb{R}$.
        		%\end{equation*}%
        		We deduce that 
        		\begin{equation*}
        			\widetilde{ \mathbb{P}}\left( Y_{t}\leq L-\frac{1}{m}\right) \leq \widetilde{\mathbb{E}}\left[
        			\varphi _{m}\left( Y_{t}\right) \right] =\lim_{n\mapsto \infty }\widetilde{%
        				\mathbb{E}}\left[ \varphi _{m}\left( X_{t}^{n}\right) \bar{M}_{n}^{n}\right]
        			=0.
        		\end{equation*}%
        		Hence the support of $Y$ is a subset of $[L,\infty ).$ Also observe that $B$ is an
        		increasing process and that due to Lemma  \ref{lem:333} and \eqref{Bn}  
        		\begin{align*}
        			&\mathbb{E}\left[ \int_{0}^{T}\psi _{m}\left( Y_{s}\right) dB_{s}\right]\\
        			=&\lim_{n\to \infty }\widetilde{\mathbb{E}}\left[ \left(
        			\sum_{i=0}^{n-1}\psi _{m}\left( X_{i-1}^{n}\right)
        			(B_{t_{i}}^{n}-B_{t_{i-1}}^{n})\right) \bar{M}_{n}^{n}\right] \\
        			=&\lim_{n\to \infty }\widetilde{\mathbb{E}}\left[ \left(
        			\sum_{i=0}^{n-1}\psi _{m}\left( X_{i-1}^{n}\right) \sigma _{i-1}\widetilde{%
        				\mathbb{E}}_{i-1}[Z_{i}\bar{m}_{i}]\right) \bar{M}_{n}^{n}\right] \\
        			\leq &2\|a\|_\infty\Delta\lim_{n\to \infty } \widetilde{\mathbb{E}}\left[ \left(
        			\sum_{i=0}^{n-1}\psi _{m}\left( X_{i-1}^{n}\right) g_{i}(X_{i-1}^{L})\right) 
        			\bar{M}_{n}^{n}\right] \\
        			\leq &C\exp \left( -\frac{1}{16\| a\|_\infty \Delta m^{2}}\right) \sqrt{\Delta }\lim_{n\to \infty }\widetilde{\mathbb{E}}\left[
        			\left( \sum_{i=0}^{n-1}\exp \left( -\frac{\left( X_{i-1}-L\right) ^{2}}{%
        				4a_{i}\Delta }\right) \right) \bar{M}_{n}^{n}\right] \\
        			\leq &C\lim_{n\to \infty }\exp \left( -\frac{n}{16\| a\|_\infty  Tm^{2}}\right) =0,
        		\end{align*}      
        		where $\psi _{m}~$is the positive continuous function with support in the
        		interval $(L+\frac{1}{2m},\infty )$ given by 
        		\begin{equation*}
        			\psi _{m}\left( x\right) =\min \left( 1,m\left( x-L-\frac{1}{2m}\right)
        			\right)1_{(x\geq L+\frac{1}{2m})}
        		\end{equation*}%
        		and the last inequality follows
        		from Lemma \ref{lem:essb}. We deduce that, almost surely, 
        		\begin{equation*}
        			\int_{0}^{T}\psi _{m}\left( Y_{s}\right) dB_{s}=0
        		\end{equation*}%
        		and so 
        		\begin{equation*}
        			0\leq \int_{0}^{t}1_{(Y_s>L)}dB_{s}\leq
        			\lim_{m\rightarrow \infty }\int_{0}^{t}\psi _{m}\left( Y_{s}\right) dB_{s}=0.
        		\end{equation*}
        		
        		Finally, let us note that the law of the process $(Y,B,\xi ,\mathcal{K},W,0)$
        		is uniquely identified by the identities (\ref{exp}), \eqref{G}, the pathwise uniqueness of 
        		equation (\ref{reflected}) and the fact that $W$ is a Brownian motion. This together with the
        		convergence along subsequences, gives us the convergence in law of the whole
        		sequence $({X}^{n},B^{n},E^{n},\mathcal{K}^{n},R^{n},\Gamma ^{n})$ under $%
        		{\mathbb{Q}}^n$ to the process $(Y,B,\xi ,\mathcal{K},W,0)$.
        	\end{proof}
        	
        	%\subsection{Convergence of the derivative characterization}
        	
        	\section{Proof of the probabilistic representation formula and applications}
        	\label{sec:5} In this section, we use Theorem \ref{bigT} in order to obtain the limits of the approximative representation for the first derivative obtained in %
        	\eqref{fprime}. We express the limits first as expectations over a
        	probability space under which $Y$ is driftless, in other words $Y$ satisfies %
        	\eqref{reflected}. Within this probability space, the process $\mathcal{E} $
        	satisfies \eqref{exp}. To emphasize this, we use $\widetilde {\mathbb{E}}$ to denote the
        	corresponding expectation. Finally, via a standard Girsanov transformation,
        	we express the derivatives as expectations over a probability space under
        	which $Y$ has drift $b$, in other words $Y$ satisfies \eqref{eq:rep}. Within
        	this probability space, the process $\mathcal{E}_t =\xi_t\exp\left(\frac{b}a(L)B_t\right)$ satisfies 
        	\begin{align}
        		\mathcal{E}_{t}=& 1+\int_{0}^{t}\mathcal{E}_{s}\left( b^{\prime
        		}(Y_{s})ds+\sigma ^{\prime }(Y_{s})dW_{s}+{\frac{b}{\sigma ^{2}}}\left(
        		L\right) dB _{s}\right) .  \label{e}
        	\end{align} 
        	
        	Note
        	that the Brownian motion $W$ appearing in  equations such as \eqref{eq:rep}, %
        	\eqref{e}, \eqref{exp}, \eqref{reflected} is a `generic' Brownian
        	motion, not necessarily the original Brownian motion appearing in the equation %
        	\eqref{x} satisfied by $X$.
        	
        	\begin{theorem}
        		\label{th:ch} Recall the set up of \eqref{eq:rep}. Then 
        		\begin{equation}  \label{limder}
        			\begin{aligned} \widetilde{\mathbb{E}}_{0,x}\left[ {f}^{\prime }\left(
        				X_{n}^{n}\right) E_{n}^{n}\mathcal{K} ^n_n\bar{M} _{n}^{n}\right]&\to
        				\widetilde{\mathbb{E}}_{0,x}\left[ {f}^{\prime }\left( Y_T\right)
        				\xi_T\mathcal{K}_T\right]=\mathbb{E}_{0,x}\left[ {f}^{\prime }\left(
        				Y_T\right) \mathcal{E}_T\right] \\
        				\sum_{i=1}^{n}\widetilde{\mathbb{E}}_{0,x}\left[ f_{i}\left( X_{i}^{n,x}\right)
        				E_{i-1}^{n}h_{i}\mathcal{K}_i\bar{M}_{i}^{n}\right] &\to
        				{\frac{b}{a}} \left( L\right)\widetilde{\mathbb{E}}_{0,x}\left[f(Y_T)
        				{\xi}_{\rho_T}\mathcal{K}_{T} 1_{(\tau\leq T)} \right]\\
        				&\hspace{1cm}={\frac{b}{a}} \left(
        				L\right)\mathbb{E}_{0,x}\left[f(Y_T) {\mathcal{E}}_{\rho_T} 1_{(\tau\leq T)}
        				\right]. \end{aligned}
        		\end{equation}
        	\end{theorem}
        	
        	\begin{remark}
        		\label{rem:16}To keep the notation consistent, in (\ref{limder}) we kept the
        		tilde in the expression for the two limits, $\widetilde{\mathbb{E}}_{0,x}\left[ {%
        			f}^{\prime }\left( Y_T\right) \xi_T\mathcal{K}_T\right]$ and $\widetilde{\mathbb{%
        				E}}_{0,x}\left[f(Y_T) {\xi}_{\rho_T}\mathcal{K}_{T} 1_{\tau\leq T} %
        		\right]$. The expectation in these two quantities are taken with respect to a
        		probability measure, which we can denote by $\widetilde {\mathbb{P}}$ (via a slight abuse
        		of notation as this is not necessarily the measure $\widetilde {\mathbb{P}}$ defined in
        		Section \ref{sec:4}), under which $(Y,B)$ is the solution of (\ref{reflected}%
        		), the reflected equation in the domain $[L,\infty )$ with no drift term. One
        		can use an equivalent representation of the two limits as $\mathbb{E}_{0,x}%
        		\left[ {f}^{\prime }\left( Y_T\right) \xi_T\mathcal{K}_T\right]$ and $%
        		\mathbb{E}_{0,x}\left[f(Y_T) {\xi}_{T}\mathcal{K}_{T} 1_{\tau\leq
        			T} \right]$, where the expectation $\mathbb{E}_{0,x}$ in the two quantities
        		are taken with respect to a probability measure, which we can denote by $\mathbb{P}$
        		(again, via a slight abuse of notation as this is not necessarily the
        		original measure $\mathbb{P}$), under which $(Y,B)$ is the solution of (\ref{eq:rep}%
        		), the reflected equation in the domain $[L,\infty )$ that incorporates the
        		drift term $b$. The transfer from $\widetilde {\mathbb{P}}$ to $\mathbb{P}$ is done via Girsanov's
        		theorem. More precisely we have that $\displaystyle \left.\frac{d\widetilde {\mathbb{P}}}{d
        			\mathbb{P}}\right|_{{\mathcal{F}}_T}=\frac{1}{\widetilde{K}_T}$, where 
        		\begin{equation*}
        			\widetilde{K}_t:=\exp \left( \int_{0}^{t}b\sigma^{-1}(Y_{s})dW_{s}-\frac{1}{2}%
        			\int_{0}^{t}b^{2}a^{-1}(Y_{s})ds\right),\ \ \ t\ge 0 . 
        		\end{equation*}
        		The equivalent representation in (\ref{limder}) is obtained once we observe
        		that the process $\mathcal{E}$, as defined in the introduction, satisfies the
        		identity $\mathcal{E}_{T}=\frac{\xi_T\mathcal{K}_T}{\widetilde{K}_T}$.
        	\end{remark}
        	
        	In the next proof we will use the following representation for the random variables
        	$h_{i}$ appearing in (\ref{fprime}). 
        	%\textcolor{red}{Bar now denotes main terms while Tilde and Hat
        	%        denote negligible at the end} 
        	\begin{equation*}  
        		%\label{newdecompforhi}
        		h_{i}:={\frac{b}{a}}\left( L\right) \bar{h}_{i}+\widehat{h}_{i},
        	\end{equation*}%
        	where $\bar{h}_{i}:=1_{(U_{i}\leq p_{i})}$ and 
        	\begin{equation}
        		\widehat{h}_{i}:=X^L_{i-1}\left( {\varpi }_{i-1}+\partial _{i-1}\left( \frac{
        			b_{i-1}}{a_{i-1}}\right) \right) 1_{(U_{i}\leq p_{i})}.  \label{habar}
        	\end{equation}%
        	and ${\varpi }_{i-1}=\varpi (\frac{b}{a},X_{i-1})$ is defined as the bounded
        	continuous function (due to the regularity hypotheses on the coefficients $a$
        	and $b$)
        	\begin{equation*}
        		{\varpi }\left( g,x\right) =\left\{ 
        		\begin{array}{cc}
        			\int_0^1g'(\alpha x+(1-\alpha)L)d\alpha; & x>L ,\\ 
        			g^{\prime }\left( L+\right); & x=L.
        		\end{array}
        		\right.
        	\end{equation*}%
        	%       In the above, $g:={\frac{b}{a}}$ and we denote by $g^{\prime }(L)$ the right
        	%       derivative of $f$ at $L$.
        	
        	Note that we can reduce our considertion to  just the term multiplying $\bar{h
        	}_i $ instead of the full formula for $h_i $ as $ \widehat{h}_i $ is negligible by virtue of Lemma \ref{additionalterm}.
        	
        	\begin{proof}[Proof of Theorem \protect\ref{th:ch}] 
        		The first result is immediate from Theorem \ref{bigT} and the moment
        		estimates in Lemma \ref{moments}. The second one, requires a rewriting of
        		the expectation using a path decomposition. For this, 
        		define $\bar{\rho}_T^{n}=\sup\{t_i;U_i<p_i\} $ and therefore $1_{(\bar{\rho}%
        			_T^{n}=t_i)}=1_{(U_i\leq p_i,U_{i+1}>p_{i+1},...,U_n>p_n)} $. In particular,
        		we define $\bar{\rho}_T^n=0 $ and $E^n_{-\Delta}=0 $ if $\{t_i;U_i<p_i\}
        		=\emptyset $. Furthermore, if we let  $\bar{\tau}_T^n=\inf\{t_i;U_i<p_i \} $ and using the
        		tower property for conditional expectations we obtain that 
        		%                \textcolor{red}{ Is there a factor 2 missing in this formula? $\bar  h_im_i $  No it is correct! Maybe a remark is needed?$ \mathbb{P}(\min_{s\leq \Delta}W_s\leq \frac{L-x}{\Delta})=\textcolor{red}{ 2  }\bar\Phi(\frac{x-L}\sigma) $! }
        		\begin{align*}
        			\sum_{i=1}^{n}\widetilde{\mathbb{E}}_{0,x}\left[ f_{i}\left( X_{i}^{n}\right)
        			E_{i-1}^{n}\bar{h}_{i}\mathcal{K}^n_i\bar{M}_{i}^{n}\right] = &{\frac{b}{a}}%
        			\left( L\right)\widetilde{\mathbb{E}}_{0,x}\left[f(X^{n}_n){\mathcal{K}}%
        			_n\sum_{i=1}^{n}{E}^n_{i-1}1_{(\bar{\rho}^n_T=t_i)}\bar{M}^n_{i-1}\right] \\
        			=&{\frac{b}{a}}\left( L\right)\widetilde{\mathbb{E}}_{0,x}\left[f(X^{n}_n){%
        				\mathcal{K}}_n1_{(\bar{\tau}_T^{n}<T)} {E}^n_{\bar{\rho}_T^{n}-\Delta}\bar{M}^n%
        			_n\right].
        		\end{align*}
        		Then,
        		using Proposition \ref{rdthe}, $( \bar{\rho}_T^{n},\bar{\tau}_T^{n}) $
        		converges to $( \rho_T,\tau)$, where $\rho_T:=\sup\{s<T: Y_s=L\}$. This fact follows from the
        		convergence in distribution of the sequence $({X}^{n},R^{n},B^n, \Gamma^n,
        		E^{n},\mathcal{K} ^n)$ under ${\mathbb{Q}}^n$ to $(Y,W,0,\mathcal{E},\mathcal{K%
        		},\ell )$. The conclusion follows because $P(\tau=T)=0$.
        	\end{proof}
        	
        	We are now ready to prove that the limit deduced in Theorem \ref{th:ch}
        	corresponds to the derivative of $\mathbb{E}[f(X_{T})1_{(\tau >T)}]$ in \eqref{eq:0}. To do so, we re-introduce the explicit dependence on the
        	starting value $x$ of the various processes needed to obtain the
        	representation of the derivative and define the functions $\varphi
        	_{n},\varphi ,\psi :[T,\infty )\mapsto \mathbb{R}$ as follows 
        	\begin{align*}
        		\varphi \left( x\right) =&\mathbb{E}_{0,x}[f(X^x_{T\wedge \tau })]. \\
        		\psi \left( x\right) =&\mathbb{E}_{0,x}[f^{\prime }(Y_{T}){\mathcal{E}}%
        		_{T}]+{\frac{b}{\sigma ^{2}}}\left( L\right) \mathbb{E}_{0,x}\left[ f(Y_{T}){%
        			\mathcal{E}}_{\rho _{T}\vee 0}1_{(\tau \leq T)}\right] . \\
        		\varphi _{n}\left( x\right) =&\mathbb{E}_{0,x}[f(X_{T\wedge \tau^{n}
        		}^{n})].
        	\end{align*}

        	\begin{theorem}
        		\label{th:main} We assume that $b, \sigma\in C^2_b$  and that $\sigma$ is uniformly elliptic. Let $f\in C_{b}^{1}([L,\infty ))$ such that $f(L)=0$. Then $P_Tf$ is differentiable on $[L,\infty)$ and its derivative has the representation  \eqref{eq:0}. That is,
        		\begin{equation}\label{eq:0'}
        			\partial _{x}P_Tf(x)=\partial _{x}\mathbb{E}_{0,x}[f(X^x_{T\wedge \tau })]= \mathbb{E}%
        			_{0,x}[f^{\prime }(Y_{T}){\mathcal{E}}_{T}]+{\frac{b}{\sigma ^{2}}}\left(
        			L\right) \mathbb{E}_{0,x}\left[ f(Y_{T}){\mathcal{E}}_{\rho _{T}\vee 0}1_{(\tau
        				\leq T)}\right] .
        		\end{equation}
        	\end{theorem}
        	
        	\begin{proof}The argument we use is standard given all the previous results. Following  \cite{gobetesaim}, we
        		have that $\lim_{n\rightarrow \infty}\varphi^n=\varphi$ uniformly on any
        		compact set $K\subseteq [L,\infty) $.\footnote{%
        			In fact, pointwise convergence of $\varphi_n$ suffices for the argument.} 
        		Let us restrict first to the case that $f\in C_{b}^{2}([L,\infty ))$. We proved in Theorem \ref{th:ch} that $\partial_x\varphi _{n}$ converges
        		pointwise to $\psi$. Moreover the second derivative $\partial^2_x\varphi _{n}$  is uniformly bounded. This is proved separately in Section \ref%
        		{app:6.2a}. As a result, the convergence of  $\partial_x\varphi _{n}$ to $\psi$ is uniform on any compact set $K\subseteq [L,\infty) $. Since the differentiation operator is a closed operator in the uniform topology, we deduce immediately that the limit $\varphi$ is differentiable and its limit is indeed $\psi$.
        		The generalization to $f\in C_{b}^{1}([L,\infty ))$ is done by taking the limit of the identity \eqref{eq:0'} along a sequence $f_n\in C_{b}^{2}([L,\infty ))$ such that $\lim_{n\rightarrow \infty}f'_n =f'$ uniformly on $[L,\infty) $ and $\lim_{n\rightarrow \infty}f_n =f$ pointwise (and therefore $\lim_{n\rightarrow \infty}f_n =f$ uniformly on any compact $K\in [L,\infty) $).  
        	\end{proof}
        	
        	We deduce from formula \eqref{eq:0} the following alternative probabilistic representation of the derivative of the diffusion semigroup: 
        	
        	\begin{theorem}
        		\label{th:36}  Let $f\in C_{b}^{1}([L,\infty ))$, then $\partial_xP_Tf(x)=\mathbb{E}[f^{\prime
        		}(Y_{T})\Psi_{T}] $, $ x\geq L $, where $ \Psi_{T} $ is defined in \eqref{def:psi}.
        	\end{theorem}
        	\begin{proof}
        		Let us restrict first to the case that $f\in C_{b}^{2}([L,\infty ))$, $ x>L $. Recall the last term in the iterated formula \eqref{fprime} and rewrite
        		it using a Taylor expansion with integral remainder and recalling that $%
        		f_{i}(L)=0$,
        		\begin{align*}
        			&	\sum_{i=1}^{n}\mathbb{E}\left[ f_{i}\left( X_{i}^{n,x}\right)
        			E_{i-1}^{n}h_{i}\bar{M}_{i}^{n}\right] = \sum_{i=1}^{n}\mathbb{E}\left[
        			f_{i}^{\prime }(L)\left( X_{i}^{n,x}-L\right) E_{i-1}^{n}h_{i}\bar{M}_{i}^{n}%
        			\right]  \\
        			& +\int_{0}^{1}d\alpha \mathbb{E}\left[ \sum_{i=1}^{n}f_{i}^{\prime \prime
        			}(\alpha X_{i}^{n,x}+(1-\alpha )L)(1-\alpha )\left( X_{i}^{n,x}-L\right)
        			^{2}E_{i-1}^{n}h_{i}\bar{M}_{i}^{n}\right] .
        		\end{align*}%
        		
        		Using the same arguments as in Lemma \ref{additionalterm} and the fact that $%
        		\Vert f_{i}^{\prime \prime }\Vert _{\infty }\leq \mathsf{M}e^{\mathsf{M}%
        			(T-t_{i})}$ (see Proposition \ref{prop:30}) we obtain that the last term in the above expression converges to zero. For the first
        		term, one uses arguments as in the proof of Theorem \ref{bigT}, including
        		the use of Girsanov's change of measure as in Section \ref{sec:4} to obtain
        		that 
        		\begin{align*}
        			\sum_{i=1}^{n}\mathbb{E}\left[ f_{i}^{\prime }(L)\left( X_{i}^{n,x}-L\right)
        			E_{i-1}^{n}h_{i}\bar{M}_{i}^{n}\right] \rightarrow & \frac{b}{a}(L)\widetilde{%
        				\mathbb{E}}\left[ \int_{0}^{T}\partial _{x}P_{T-s}f(L)\mathcal{K}_{s}%
        			\mathcal{E}_{s}dB_{s}\right] \\
        			=& \frac{b}{a}(L)\mathbb{E}\left[ \int_{0}^{T}\partial _{x}P_{T-s}f(L)%
        			\mathcal{E}_{s}dB_{s}\right] .
        		\end{align*}
        		
        		Note that we have used here that $B$ is the limit process for the process
        		whose increments are 
        		\begin{equation}
        			\label{eq:R}
        			\widetilde{\mathbb{E}}_{i-1}\left[ \left( X_{i}^{n,x}-L\right) h_{i}\bar{m}_{i}%
        			\right] =2\widetilde{\mathbb{E}}_{i-1}\left[ \left( \sigma
        			_{i-1}Z_{i}+X_{i-1}^{L}\right) 1_{(U_{i}\leq p_{i})}\right] =\Delta _{i}B^{n}.
        		\end{equation}%
        		Putting all the
        		limits together in \eqref{fprime}, one arrives at the linear equation 
        		\begin{equation*}
        			\partial _{x}P_{T}f(x)=\mathbb{E}[f^{\prime }(Y_{T}){\mathcal{E}}_{T}]+\frac{%
        				b}{a}(L)\mathbb{E}\left[ \int_{0}^{T}\partial _{x}P_{T-s}f(L)\mathcal{E}%
        			_{s}dB_{s}\right] .
        		\end{equation*}
        		
        		In other words, the function $\varsigma :[0,T]\times \lbrack L,\infty
        		)\rightarrow \mathbb{R}$ defined as $\varsigma \left( t,x\right) =\partial
        		_{x}P_{t}f(x)$ satisfies the equation%
        		\begin{equation}
        			\varsigma \left( t,x\right) =\mathbb{E}[f^{\prime }(Y_{t}){\mathcal{E}}_{t}]+%
        			\frac{b}{a}(L)\mathbb{E}\left[ \int_{0}^{t}\varsigma \left( t-s,L\right) 
        			\mathcal{E}_{s}dB_{s}\right] .  \label{eqsur}
        		\end{equation}
        		We note that the function $\widetilde{\varsigma}:[0,T]\times \lbrack L,\infty
        		)\rightarrow \mathbb{R}$ defined as $\widetilde{\varsigma}\left( t,x\right) =%
        		\mathbb{E}[f^{\prime }(Y_{t})e^{\frac{b}{a}(L)B_{t}}{\mathcal{E}}_{t}]$ also
        		satisfies equation (\ref{eqsur}). To see this, observe that  by direct integration
        		\begin{eqnarray*}
        			\widetilde{\varsigma}\left( t,x\right)  &=&\mathbb{E}[f^{\prime }(Y_{t}){%
        				\mathcal{E}}_{t}]+\frac{b}{a}(L)\mathbb{E}\left[ \int_{0}^{t}\mathbb{E}
        			_{s,Y_{s}}[f^{\prime }(Y_{t})e^{\frac{b}{a}(L)(B_{t}-B_{s})}{\mathcal{E}}
        			_{s,t}]\mathcal{E}_{s}dB_{s}\right]  \\
        			&=&\mathbb{E}[f^{\prime }(Y_{t}){\mathcal{E}}_{t}]+\frac{b}{a}(L)\mathbb{E}%
        			\left[ \int_{0}^{t}\widetilde{\varsigma}\left( t-s,L\right) \mathcal{E}_{s}dB_{s}%
        			\right] .
        		\end{eqnarray*}%
        		We show next that (\ref{eqsur}) has a unique solution and therefore that $$%
        		\partial _{x}P_{T}f(x)=\mathbb{E}[f^{\prime }(Y_{T})e^{\frac{b}{a}(L)B_{T}}{%
        			\mathcal{E}}_{T}].$$
        		
        		Assume that there are two solutions of (\ref{eqsur}) which are
        		uniformly bounded in $(t,x)\in \lbrack 0,T]\times \lbrack L,\infty )$.
        		Consider the difference and denote it by $G_{t}(x)$, $(t,x)\in \lbrack
        		0,T]\times \lbrack L,\infty )$. Taking $x=L$, one obtains the estimate: 
        		\begin{equation*}
        			\sup_{t\leq T}G_{t}(L)\leq \frac{b}{a}(L)\sup_{t\leq T}G_{t}(L)\mathbb{E}%
        			\left[ B_{T}\max_{s\leq T}\mathcal{E}_{s}\right] .
        		\end{equation*}%
        		This gives a contradiction if $T$ is smaller than some sufficiently small
        		threshold time value $T_{0}>0$. Using again the equation satisfied by $G_{t}(x)$, we obtain the uniqueness
        		for all $x\geq L$. The uniqueness for arbitrary $T$ is obtained by using the
        		Markov property and recursively merging the uniqueness in intervals
        		of the form $\left[ kT_{0},\left( k+1\right) T_{0}\right] $ and, in each
        		interval, replacing $f$ by $P_{kT_{0}}f$. 
        		
        		The generalization to $f\in C_{b}^{1}([L,\infty ))$ is done as in the proof of Theorem \ref{th:main}. The result in the case $ x=L $ is obtained by taking limits.
        		%
        		%In order to do this, suppose that there are two solutions so that $ \sup_{s\leq T}|\partial_xP_Tf(x)|<\infty. $ Then choosing $ x=L$
        		%and taking     The iteration of the above equation together with the tower property for conditional expectations, the fact that $ B $ only increases on the set $ \{s;Y_s=L\} $ can be used for an iterative procedure. For example, the first step gives:
        		%       \begin{align*}
        		%                       \partial_xP_Tf(x)=&\mathbb{E}[f^{\prime }(Y_{T}){\mathcal{E}}_{T}]+\frac{b}{a}(L)\mathbb{E}\left[ \int_0^T\mathbb{E}_{s,L}\left[f^{\prime }(Y_{T}){\mathcal{E}}_{s,T}\right]\mathcal{E}_sdB_s\right]\\
        		%&+\left(\frac{b}{a}\right)^2(L)\mathbb{E}\left[ \int_0^T\int_s^T\partial_xP_{T-u}f(Y_s)\mathcal{E}_{u,s}dB_u\mathcal{E}_{s}dB_s\right]\\
        		%=&\mathbb{E}[f^{\prime }(Y_{T}){\mathcal{E}}_{T}]+\frac{b}{a}(L)\mathbb{E}\left[ f^{\prime }(Y_{T})\mathcal{E}_TB_T\right]+\left(\frac{b}{a}\right)^2(L)\mathbb{E}\left[ \int_0^T\partial_xP_{T-u}f(Y_u)\mathcal{E}_{u}B_udB_u\right].
        		%       \end{align*}
        		%The iteration and limit procedure with Lemmas \ref{boundedderivative} and \ref{moments} as well as moment estimates for $ B_T  $  give:
        		%       \begin{align*}
        		%               \partial_xP_Tf(x)=\mathbb{E}[f^{\prime }(Y_{T})e^{\frac{b}{a}(L)B_T}{\mathcal{E}}_{T}].
        		%       \end{align*}
        		%Finally, applying It\^o's formula for the process $\check{\mathcal{E}}_{t}=e^{%
        		%\frac{b}{a}(L)B_t}{\mathcal{E}}_{t} $, one obtains the result. %
        		%\textcolor{red}{ Still, did not understand why you prefer this writing...  }
        	\end{proof}
        	%\section{A Bismut-Elworthy-Li type formula}
        	\begin{remark}
        		One may prove Theorem \ref{th:36} without the use of Theorem \ref{th:main} by applying the Arzel\`a-Ascoli Theorem to $ (\partial_x\varphi_n )_{n\in\mathbb{N}} $ which will imply that any subsequence of $ (\partial_x\varphi_n )_{n\in\mathbb{N}} $ contains a subsequence that converges to a solution of \eqref{eqsur} which has a unique solution. In this case, one may consider  Theorem \ref{th:main} as an alternative expression for $\partial _{x}P_Tf(x)  $.
        		
        		We also remark that one can prove the above result using directly Lamperti transform in the particular case that $ b=\frac 12\sigma\sigma' $. This proof is restricted to the one dimensional case and does not seem to be applicable to multi-dimensional situations. 
        	\end{remark}
        	Finally, we use the arguments that we have developed so far
        	to introduce an integration by parts  formula that extends the classical    
        	Bismut-Elworthy-Li (BEL) formulation (see \cite{ElLi}) to the killed diffusion framework.

        	\begin{theorem}
        		\label{th:19}
        		Let $f:[L,\infty)\to\mathbb{R} $ be a measurable and bounded function such
        		that $f(L)=0 $ . Then 
        		%                
        		%                 \begin{align*}
        		%                       \partial_xP_Tf(x)=\mathbb{E}\left[f(Y_T)\frac{I_{0,T}}{T}\exp\left(-\int_0^T\frac{I_{u,T}}{T-u}dB_u\right)\right].
        		%                \end{align*}
        		we have for $ x\geq L $,
        		\begin{align*}
        			T\partial _{x}P_Tf(x)={\mathbb{E}}\left[f(Y_T)\int_{\rho_T\vee 0}^T
        			{\Psi}_s\sigma^{-1}(Y_s)dW_s\right].
        		\end{align*}
        		
        		%            \begin{align*}
        		%            \int_{0}^T
        		%            \textcolor{red}{ 1_{(\tau(s)>T)} }
        		%            {\Psi}_s\sigma^{-1}(Y_s)dW_s:=&L^2-\lim_{\Delta\downarrow 0}\sum_{i=1}^n1_{(\tau^n(t_i)>T)}\frac{\Delta_iW}{\sigma_{i-1}}\widehat{E}%
        		%            _{i-1}^{n,x}
        		%            \end{align*}
        		%        Here,
        		%            Here $I_{u,T}(Y_u):= \int_u^T\frac{\Psi_s\Psi_u^{-1}}{\sigma(Y_s)}
        		%            dW_s$. 
        		%        \begin{align*}
        		%                \partial_x\mathbb{E}_{0,x}[f(X_{T\wedge \tau})]=&\frac 1T\mathbb{E}_{0,x}%
        		%                \left[ {f}(Y_T)\left(\int_0^T \sigma^{-1}(Y_s)\Psi%
        		%                _{s}dW_s+\int_0^T \mu(s,T)dB_s\right)\right]. \\
        		%                \mu(s,T):=&\frac{3}{2} \Psi_{s}a^{-1}(Y_s)-{{I}_{0,T}}\frac{%
        		%                        \widetilde{I}_{s,T}}{T-s}e^{-\int_0^s\frac{I_{u,T}}{T-u}d{B}_u} \\
        		%                I_{u,T}:=&\int_u^T \sigma^{-1}(Y_s)\Psi_{u,s}dW_s  \notag \\
        		%                \widetilde{I}_{u,T}:=&I_{u,T}+\frac{3}{2} \int_u^T \Psi%
        		%                _{u,s}a^{-1}(Y_s)dB_s  \notag
        		%        \end{align*}
        		%       \begin{align*}
        		%\partial_x\mathbb{E}_{0,x}[f(X_{T\wedge \tau})]=&\frac 1T\mathbb{E}_{0,x}\left[ {f}(Y_T)\left(
        		%\int_0^T \sigma^{-1}(Y_s)\Psi_{u,s}dW_s
        		%+\frac{3}{2}
        		%\int_0^T \Psi_{u,s}a^{-1}(Y_s)dB_s
        		%\right)\right].
        		%\end{align*}
        	\end{theorem}
        	
        	We remark that the random variable which appears in the above BEL formula
        	corresponds to the one in the classical BEL diffusion  formula in the case when $B_T=0 $ and that it can be applied to the case that $ x=L $ in comparison with the result in \cite{Mall} cited in the Introduction. 
        	
        	\begin{proof}
        		The idea of the proof is  to  start from an approximation to the result in
        		Theorem \ref{th:36} but considering the interval $[0,t_j] $ instead of $%
        		[0,t_n] $ through the Markov property. 
        		Define $\widehat{E}%
        		_j^{n,x}:=\prod_{i=1} ^j\widehat{e}_i$ with
        		\begin{align*}
        			\widehat{e}_i:=1+{b}'_{i-1}
        			1_{(U_i>p_i)}\Delta+\sigma^{\prime
        			}_{i-1}1_{(U_i>p_i)}\Delta_iW
        			+\frac{b}{a}(L)X^L_i1_{(U_i\leq p_i)},
        		\end{align*}  
        		
        		Using \eqref{eq:R},  gives 
        		\begin{align*}
        			&\partial _{x}P_Tf(x) =\widetilde{\mathbb{E}}\left[\partial_if_i \widehat{E}_{i}^{n}%
        			\mathcal{K}^n_i \bar{M}_{i}^{n}\right]+o(1). %\\
        			%                &\widehat{e}_j:=1+\left(b^{\prime }-\frac{\sigma^{\prime }b}{\sigma}%
        			%                1_{(U_j>p_j)}\right)_{j-1}\Delta+\sigma^{\prime
        			%                }_{j-1}\Delta_j(R^n+\Gamma^n)+\frac{b}{a}(L)X^L_i1_{(U_i\leq p_i)}.
        		\end{align*}
        		Now applying an integration by parts in the interval $[t_{i-1},t_i] $ in
        		the setting of Section \ref{sec:4}, one obtains  that 
        		\begin{align}
        			&\widetilde{\mathbb{E}}_{i-1}\left[ \partial_if_i \widehat{E}_{i}^{n}\mathcal{K}^n_i
        			\bar{M}_{i}^{n}\right]= \widetilde{\mathbb{E}}_{i-1}\left[ f_i \varkappa_i \widehat{E%
        			}_{i-1}^{n} \mathcal{K}^n_i\bar{M}_{i}^{n}\right]+\frac{O_{i-1}^E(1)}{\Delta}
        			\notag\\
        			&\varkappa_i:=\frac{Z_i}{\sigma_{i-1}\Delta}+\frac{X^{L,\sigma}_{i-1}}{%
        				\sigma_{i-1}\sqrt{\Delta}}1_{(U_i\leq p_i)}-\left(\frac{b}{a}\right)_{i-1}+\frac{\sigma^{\prime }_{i-1}(Z_i^2-\Delta)}{%
        				\sigma_{i-1}\Delta}1_{(U_i> p_i)}.
        			\label{eq:BEL}
        		\end{align}  
        		
        		The above formula gives at any time $i$ an approximative discrete time BEL formula which can
        		be multiplied by $\Delta $ and then summed over all $i =1,...,n$ . Using this argument one obtains
        		approximations to integrals as follows  
        		\begin{align*}
        			T\partial _{x}P_Tf(x)=&\sum_{i=1}^n\Delta \widetilde{\mathbb{E}}\left[
        			f_i\varkappa_i \widehat{E}_{i-1}^{n}\mathcal{K}^n_{i} \bar{M}_{i}^{n}%
        			\right] +o(1)\\=& \widetilde{\mathbb{E}}\left[
        			f(X_{n}) \mathcal{K}^n_{n}\sum_{i=1}^n
        			\left(\frac{Z_i}{\sigma_{i-1}}-\left(\frac{b}{a}\right)_{i-1}\Delta\right) \widehat{E}_{i-1}^{n} \bar{M}_{i}^{n}M^n_{i:n}
        			\right] +o(1).
        		\end{align*}
        		In particular, note that we have used for the last term in \eqref{eq:BEL}, the fact that $ Z_i^2-\Delta $ are increments of discrete time martingales with quadratic variation that converge to zero plus smaller order terms. On the other hand, for the second term in \eqref{eq:BEL}, we have that as $ f_i(L)=0 $ then
        		\begin{align*}
        			\tilde{\mathbb{E}}_{i-1}\left[ \sqrt{\Delta}f_i \frac{X^{L,\sigma}_{i-1}}{\sigma_{i-1}}e^{\kappa^n_i}1_{(U_i\leq p_i)}\right]=\tilde{\mathbb{E}}_{i-1}\left[\sqrt{\Delta}(f_i-f_i(L))\frac{X^{L,\sigma}_{i-1}}{\sigma_{i-1}}e^{\kappa^n_i}1_{(U_i\leq p_i)}\right]=O^E_{i-1}(\sqrt{\Delta}).
        		\end{align*}

        		The limit is obtained using the weak convergence arguments in Section \ref{sec:5} including  the proof of Theorem \ref{th:ch} and Proposition \ref{rdthe} as well as the reductions that can be obtained using Remark \ref{rem:11}. For example,
        		\begin{align*}
        			&\tilde{\mathbb{E}}_{i-1}\left[ (f_i-f_{i-1})\frac{Z_i}{\sigma_{i-1}}
        			e^{\kappa^n_i}\bar{m}_i\right]=\tilde{\mathbb{E}}_{i-1}\left[ \partial_{i-1}f_{i-1}Z_i^2
        			\bar{m}_i\right]+O^E_{i-1}(\sqrt{\Delta})\\
        			&\tilde{\mathbb{E}}_{i-1}\left[ f_{i-1}\frac{Z_i}{\sigma_{i-1}}
        			e^{\kappa^n_i}\bar{m}_i\right]=\tilde{\mathbb{E}}_{i-1}\left[ f_{i-1}\frac{\Delta_iB}{a_{i-1}}
        			\bar{m}_{i-1}\right]+\tilde{\mathbb{E}}_{i-1}\left[ f_{i-1}\frac{b_{i-1}}{a_{i-1}}Z_i^2
        			\bar{m}_i\right]+O^E_{i-1}(\sqrt{\Delta}).
        		\end{align*}
        		
        		In order to study the limit of the remaining non-negligible terms we will need to change the probabilistic characterization to a reflected one using the setting in Section \ref{app:res}.   As in that section, we will use the expectation notation $ \mathbb{E} $ to denote this setting. Therefore we need to prove that the $ L^2 $ limit of $ \xi^n_n $ exists where 
        		\begin{align*}
        			\xi^n_n:=\sum_{i=\bar{\rho}^n_Tn+1}^n\frac{\Delta_iW}{\sigma_{i-1}}\widehat{E}%
        			_{i-1}^{n,x}.
        		\end{align*}
        		The convergence of the above expression is obtained using the same argument as in the proof of Theorem \ref{th:main}.
        	\end{proof}
        	%        \begin{remark}In the proof, one observes that the result is an alternative expression for the following ``heuristic'' expression
        	%               $ {\mathbb{E}}\left[f(Y_T)\int_{0}^T
        	%               \textcolor{red}{ 1_{(\tau(s)>T)} }
        	%               {\Psi}_s\sigma^{-1}(Y_s)dW_s\right]$ where $ \tau(s):=\inf\{u>s:Y_u=L\} $. It is heuristic because $  1_{(\tau(s)>T)} $ is not adapted.
        	%        \end{remark}
        	\section{Appendix}
        	
        	\label{sectionwithproofs}
        	
        	\subsection{Proof of Proposition \protect\ref{lem:pf}}
        	
        	\label{sec:dr} We split the proof of the proposition in several parts in this proof will also
        	be used in order to obtain the representation for the second
        	derivatives.
        	
        	The first is the technical lemma which is useful in itself. It also gives
        	some intuition on the derivative of the random variables $m_{i}$ (or $ 1-p_t $ in Section \ref{sec:sc}).
        	
        	In the formulas below we use the derivative with respect to the increments $%
        	\Delta _{i}W$ of the Brownian motion process $W$ in the same sense as $\partial
        	_{i}\equiv \partial _{X_{i}}$. That is, we denote by $D_{i}$ the derivative
        	with respect to the variable $\Delta _{i}W$. Hence, for $f(x,y)$, $x,y\in 
        	\mathbb{R}$, we let $D_{i}f(X_{i-1},\Delta _{i}W)\equiv \partial
        	_{y}f(X_{i-1},\Delta _{i}W)$.
        	
        	\begin{lemma}
        		\label{lem:24} Let $f:\mathbb{R}^2\to \mathbb{R}$ be a Lipschitz function,
        		then for $G_i=f(X_{i-1},\Delta_i W) $, we have 
        		\begin{align*}
        			\mathbb{E}_{i-1}\left[ G_{i}\partial _{i-1} p_{i}\right] =&-2%
        			\mathbb{E}_{i-1}\left[ D_iG_{i}1_{(U_{i}\leq p_{i})}\partial_{i-1}\left( 
        			\frac{ X_{i-1}-L}{\sigma _{i-1}} \right) \right] \\
        			&-2\mathbb{E}_{i-1}\left[ G_{i}1_{(U_{i}\leq p_{i})}\partial_{i-1}\left(%
        			\frac{ b_{i-1}\left(X_{i-1}-L\right) }{a _{i-1}} \right)\right] .
        		\end{align*}
        	\end{lemma}
        	
        	\begin{proof}
        		First note that $ \partial _{i-1}p_{i}= -2p_{i}A_{i}  $ with 
        		\begin{align}
        			%               \partial _{i-1}1_{(U_{i}\leq p_{i})}=& -2\delta _{p_{i}}(U_{i})p_{i}A_{i}, 
        			%               \notag \\
        			A_{i}:=& \partial _{i-1}\left( \frac{(X_{i}-L)(X_{i-1}-L)}{a_{i-1}\Delta }%
        			\right)  \notag \\
        			=& \left( \frac{(X_{i-1}-L)}{a_{i-1}\Delta }\partial _{i-1}X_{i}+\frac{%
        				(X_{i}-L)}{a_{i-1}\Delta }-\frac{a_{i-1}^{\prime }(X_{i-1}-L)(X_{i}-L)}{%
        				a_{i-1}^{2}\Delta }\right)  %\label{def:ai}
        			\notag\\
        			=& B_{i}+C_{i}  \notag \\
        			B_{i}:=& \partial _{i-1}\left( \frac{(X_{i-1}-L)^{2}}{a_{i-1}\Delta }\right)
        			+\partial _{i-1}\left( \frac{b_{i-1}\left( X_{i-1}-L\right) }{a_{i-1}}\right)
        			\notag\\
        			%\label{eq:54a} \\
        			C_{i}:=& \partial _{i-1}\left( \frac{X_{i-1}-L}{\sigma _{i-1}}\right) \frac{{%
        					\ \ \Delta _{i}W}}{\Delta }.  \notag
        		\end{align}%
        		%               First, using the fact that the density of $U_{i}$ is uniform, $1_{\left( {X}%
        		%                       _{i},X_{i-1}>L\right) }p_{i}\in (0,1)$ a.s. and independent of $\Delta _{i}W$%
        		%               , we have that $\mathbb{E}[\delta _{p_{i}}(U_{i})|\mathcal{F}%
        		%               _{i-1},\Delta_iW]=1$. For this, we need to insure that $p_{i}\in (0,1)$
        		%               which is true because $X_{i-1},X_{i}>L$ . Furthermore using the definition 
        		%of the Euler scheme, we deduce that: \textcolor{red}{ Is the explicit formula below ever used???if not just delete it  }
        		%               \begin{align}
        		%               A_i
        		%               \end{align}%
        		Therefore, 
        		\begin{equation*}
        			\mathbb{E}_{i-1}\left[ G_{i}\partial _{i-1} p_{i}\right] =-2%
        			\mathbb{E}_{i-1}\left[ G_{i}p_{i}\left(
        			B_{i}+C_{i}\right) \right] =-2\mathbb{E}_{i-1}\left[ G_{i}p_{i}\left(
        			B_{i}+C_{i}\right) \right] .  
        			%\label{eq:54b}
        		\end{equation*}%
        		Note that the terms collected in $B_{i}$ are $\mathcal{F}_{{i-1}}$%
        		-measurable, hence non-random under $\mathbb{E}_{i-1}$, while $C_{i}$ are $%
        		\mathcal{F}_{{i}}$-measurable, hence random under $\mathbb{E}_{i-1}$. We
        		will use the integration by parts formula for the conditional law of  $C_{i}$. That is, we have
        		for a Lipschitz function $g $ 
        		\begin{equation}
        			\label{eq:ibp}
        			\mathbb{E}_{i-1}\left[ g({\Delta _{i}W})\frac{{\Delta _{i}W}}{\Delta }\right]
        			=\mathbb{E}_{i-1}\left[ D_{i}g({\Delta _{i}W})\right] =\mathbb{E}_{i-1}\left[
        			g^{\prime }({\Delta _{i}W})\right] .
        		\end{equation}%
        		
        		To finish the proof, we compute 
        		\begin{equation*}
        			D_{i}p_{i}=D_{i}e^{-2\frac{\left( X_{i-1}-L\right) \left( X_{i}-L\right) }{%
        					a_{i-1}\Delta }}=-2\frac{(X_{i-1}-L)}{\sigma _{i-1}\Delta }p_{i}.
        		\end{equation*}
        		
        		Then using \eqref{eq:ibp}, we have 
        		\begin{align*}
        			2\mathbb{E}_{i-1}\left[ G_{i}p_{i}C_{i}\right] =&-4\mathbb{E}_{i-1}\left[
        			G_{i}\partial _{i-1}\left( \frac{X_{i-1}-L}{\sigma _{i-1}}\right) \frac{%
        				(X_{i-1}-L)}{\sigma _{i-1}\Delta }p_{i}\right]
        			\\&+2\mathbb{E}_{i-1}\left[
        			D_{i}G_{i}p_{i}\partial _{i-1}\left( \frac{X_{i-1}-L}{\sigma _{i-1}}\right) %
        			\right] .
        		\end{align*}
        		The first term cancels with the corresponding terms appearing in $B_{i}$ and
        		therefore we have 
        		\begin{align*}
        			2\mathbb{E}_{i-1}\left[ G_{i}p_{i}A_{i}\right] =&2\mathbb{E}_{i-1}\left[
        			D_{i}G_{i}1_{(U_{i}\leq p_{i})}\partial _{i-1}\left( \frac{X_{i-1}-L}{\sigma
        				_{i-1}}\right) \right]\\
        			& +2\mathbb{E}_{i-1}\left[ G_{i}1_{(U_{i}\leq
        				p_{i})}\partial _{i-1}\left( \frac{b_{i-1}\left( X_{i-1}-L\right) }{a_{i-1}}%
        			\right) \right] .
        		\end{align*}
        		The claim follows.
        	\end{proof}
        	
        	We apply the above general lemma for the iteration of derivatives in different forms. Therefore, throughout this section we use
        	random variables ${\lambda }_{i}$ of the form%
        	\begin{equation}
        		{\lambda }_{i}:=\lambda ^{0}(X_{i-1},\Delta _{i}W)+1_{(U_{i}\leq
        			p_{i})}\lambda ^{1}(X_{i-1},\Delta _{i}W),  \label{assumptionH}
        	\end{equation}%
        	where $\lambda ^{0},\lambda ^{1}\in C_{b}^{1}(\mathbb{R}^{2},\mathbb{R})$.
        	We will also use the following operator notation for terms that often appear in the formulas:
        	\begin{equation*}
        		{{\bar{\Delta}}{\lambda }_{i} }:=\left( {\lambda }^{0}+2{%
        			\lambda }^{1}\right) (X_{i-1},\Delta _{i}W).
        	\end{equation*}
        	
        	\begin{lemma}
        		\label{lem:22} Let $f:\mathbb{R}\to\mathbb{R} $ be a Lipschitz function and $%
        		\lambda _{i}$ be of the form (\ref{assumptionH}) with $\lambda ^{0},\lambda ^{1}\in C_{b}^{1}(\mathbb{R}^{2},\mathbb{R})$. Then the following
        		derivative formula is satisfied almost surely
        		\begin{align*}
        			\partial _{i-1}\mathbb{E}_{{i-1}}\left[ {f}(X_{i})\lambda _{i}\bar{m}_{i}%
        			\right] =&\mathbb{E}_{i-1}\left[ \left( f^{\prime }(X_{i})\bar{A}_{i}^{1}({%
        				\lambda _{i}})+\left( f(X_{i})-f(L)\right) \bar{A}_{i}^{0}({\lambda _{i}}%
        			)\right) \bar{m}_{i}\right]\\
        			& +f(L)\partial _{i-1}\mathbb{E}_{{i-1}}\left[ {%
        				\lambda _{i}}\bar{m}_{i}\right] .
        		\end{align*}
        		Here, 
        		\begin{align*}
        			&\bar{A}_{i}^{1}({\lambda _{i}}):= 1_{(U_{i}> p_{i})}+(\partial _{i-1}X_{i}{%
        				\lambda _{i}}-1)+\frac{\sigma^{\prime }_{i-1}}{\sigma _{i-1}}{\bar{\Delta}}{%
        				\lambda _{i}}\left( {X_{i-1}-L}\right) 1_{(U_{i}\leq p_{i})}-({\bar{\Delta}}{%
        				\lambda _{i}}-1) 1_{(U_{i}\leq p_{i})} \\
        			&\bar{A}_{i}^{0}({\lambda _{i}}):= \partial _{i-1}\lambda
        			_{i}^{0}\\&+1_{(U_{i}\leq p_{i})}\left( \partial _{i-1}{\lambda _{i}}%
        			^{1}-D_{i}({\bar{\Delta}}{\lambda _{i}})\partial _{i-1}\left( \frac{X_{i-1}-L}{%
        				\sigma _{i-1}}\right) -{\bar{\Delta}}{\lambda }_{i}\partial _{{i-1}}\left( 
        			\frac{b_{{i-1}}\left( X_{i-1}-L\right) }{a_{{i-1}}}\right) \right).
        		\end{align*}
        	\end{lemma}
        	
        	%\textcolor{red}{ 
        	%In the above formula, we remark that $1_{(U_{i}<p_{i})} $ in the argument of 
        	%${e}_{i} $ is considered as an independent algebraic variable and therefore
        	%the notation $\partial_{{i-1}}{e}_{i} $ indicates differentiation with
        	%respect to the second variable in ${e}(u,x,y) $. Similarly, $ \partial_if_i=f'(X_i) $.  }
        	
        	\begin{proof}
        		First, we assume that the following boundary condition is satisfied $f(L)=0$%
        		. In this case, the proof of this lemma is straightforward algebra. In fact,
        		the proof follows from the following decomposition formula: 
        		\begin{equation*}
        			{\lambda }_{i}\bar{m}_{i}=\left( {\lambda }^{0}(X_{i-1},\Delta
        			_{i}W)+{\bar{\Delta}}{\lambda }_i(X_{i-1},\Delta _{i}W)1_{(U_{i}\leq
        				p_{i})}\right) 1_{(X_{i}>L)}.
        		\end{equation*}
        		
        		To the right of the above formula, we apply Lemma \ref{lem:24} to the random variables 
        		$G_i=f(X_i){\lambda }^{0}(X_{i-1},\Delta _{i}W)1_{(X_{i}>L)}$ and $%
        		G_i=f(X_i){\bar{\Delta}}{ \lambda }_i(X_{i-1},\Delta _{i}W) 1_{(X_{i}>L)}p_i$. In
        		order to do this, one remarks that the function $(f(x)-f(L))1_{(x>L)}$ is
        		Lipschitz and its composition with $x+b(x)\Delta+\sigma(x)\Delta_iW $ is
        		also Lipschitz. Furthermore, $\partial_x\left(f(x)
        		1_{(x>L)}\right)=f^{\prime }(x) 1_{(x>L)}$ almost everywhere. Therefore, using that $ 1_{(U_i\leq p_i)}\bar{m}_i=21_{(X_i>L)}1_{(U_i\leq p_i)} $, we
        		obtain 
        		%needs to use smooth approximations to $1_{(X_{i}>L)}  $, say $ \Phi(k^{-1}(X_i-L)) $ for $ k\in\mathbb{N} $. 
        		%Given that the function $ f $ is smooth and  satisfies the boundary condition $ f(L)=0 $, we have that many terms vanish in the limit. That is, for $j=0,1,$ we
        		%have that 
        		%\begin{align*}
        		%\lim_{k\to\infty}\mathbb{E}_{i-1}\left[ f(X_{i}){\bar{\Delta}}{\lambda _{i}}k^{-1}\Phi(k^{-1}(X_i-L)) \sigma _{i-1}1_{(U_{i}\leq p_{i})}\partial _{i-1}\left( \frac{%
        		%X_{i-1}-L}{\sigma _{i-1}}\right) \bar{m}_{i}\right] =& 0 \\
        		%\lim_{k\to\infty}\mathbb{E}_{i-1}\left[ f(X_{i})\lambda _{i}^{j}k^{-1}\Phi(k^{-1}(X_i-L))\sigma
        		%_{i-1}1_{(U_{i}\leq p_{i})}\bar{m}_{i}\right] =& 0
        		%\end{align*}
        		\begin{align*}
        			\partial _{i-1}\mathbb{E}_{{i-1}}\left[ {f}(X_{i}){\lambda _{i}}\bar{m}_{i}%
        			\right] =& \mathbb{E}_{i-1}\left[ \left( \partial _{i-1}(f(X_{i})\lambda
        			_{i}^{0})+1_{(U_{i}\leq p_{i})}\partial _{i-1}(f(X_{i})\lambda
        			_{i}^{1})\right. \right.\\
        			&-D_{i}(f(X_{i}){\bar{\Delta}}{\lambda _{i}})\partial _{i-1}\left( 
        			\frac{X_{i-1}-L}{\sigma _{i-1}}\right) 1_{(U_{i}\leq p_{i})}\\
        			& \left. \left. -f(X_{i}){\bar{\Delta}}{\lambda }_{i}\partial _{{i-1}}\left( 
        			\frac{b_{{i-1}}\left( X_{i-1}-L\right) }{a_{{i-1}}}\right) 1_{(U_{i}\leq
        				p_{i})}\right) \bar{m}_{i}\right] .
        		\end{align*}%
        		After some algebraic rearrangement we get the result in the case that $%
        		f(L)=0$. In the general case, is enough to apply the above result to $%
        		g(x):=f(x)-f(L)$ in order to obtain the claim.
        	\end{proof}
        	
        	The result below will be applied in order to compute the first
        	derivative which means that we choose $({\lambda }_{i}^{0},{\lambda }%
        	_{i}^{1})=(1,0)$. Another case will appear when computing the second
        	derivative in Section \ref{app:6.2a}.
        	
        	\begin{corollary}
        		\label{cor:23} Let $\lambda _{i}^{0}=\lambda ^{0}(X_{i-1},\Delta _{i}W),$
        		where $\lambda ^{0}\in C_{b}^{1}(\mathbb{R}^{2},\mathbb{R})$ and assume that 
        		$f(L)=0$. Then we have 
        		\begin{equation*}
        			\partial _{{i-1}}\mathbb{E}_{{i-1}}\left[ {f}(X_{i}){\lambda _{i}^{0}m}_{i}%
        			\right] =\mathbb{E}_{{i-1}}\left[ \left( f^{\prime }(X_{i}){A}_{i}^{1}({%
        				\lambda _{i}^{0}})+f(X_{i}){A}_{i}^{0}({\lambda _{i}^{0}})\right) \bar{m}_{i}%
        			\right].
        		\end{equation*}%
        		Here, 
        		\begin{align*}
        			A_{i}^{1}({\lambda _{i}^{0}}):=& \partial _{i-1}X_{i}{\lambda _{i}^{0}}%
        			1_{(U_{i}>p_{i})}+\sigma _{i-1}\lambda _{i}^{0}\partial _{i-1}\left( \frac{%
        				X_{i-1}-L}{\sigma _{i-1}}\right) 1_{(U_{i}\leq p_{i})} \\
        			{A}_{i}^{0}({\lambda _{i}^{0}}):=& \partial _{i-1}\lambda
        			_{i}^{0}1_{(U_{i}>p_{i})}+D_{i}\lambda _{i}^{0}\partial _{i-1}\left( \frac{%
        				X_{i-1}-L}{\sigma _{i-1}}\right) 1_{(U_{i}\leq p_{i})}\\&+{\lambda }%
        			_{i}^{0}\partial _{{i-1}}\left( \frac{b_{{i-1}}\left( X_{i-1}-L\right) }{a_{{%
        						i-1}}}\right) 1_{(U_{i}\leq p_{i})}.
        		\end{align*}
        	\end{corollary}
        	
        	\begin{proof}
        		As in the proof of Lemma \ref{lem:22}, one obtains first that 
        		\begin{align*}
        			\partial _{{i-1}}\mathbb{E}_{{i-1}}\left[ {f}(X_{i}){\lambda _{i}^{0}m}_{i}%
        			\right] =& \mathbb{E}_{{i-1}}\left[ \partial _{i-1}(f(X_{i}){\lambda _{i}^{0}%
        			})m_{i}\right. \\&+2D_{i}(f(X_{i})\lambda _{i}^{0})\partial _{i-1}\left( \frac{X_{i-1}-L%
        			}{\sigma _{i-1}}\right) 1_{(X_i>L)}1_{(U_{i}\leq p_{i})}\\
        			& \left. +2f(X_{i}){\lambda }_{i}^{0}\partial _{{i-1}}\left( \frac{b_{{i-1}%
        				}\left( X_{{i-1}}-L\right) }{a_{{i-1}}}\right) 1_{(X_i>L)}1_{(U_{i}\leq p_{i})}\right] .
        		\end{align*}%
        		Then the result follows using that, for any random variables $%
        		A_{i},B_{i}$ that are $\mathcal{F}_{i}$-measurable, we have 
        		\begin{equation*}
        			\mathbb{E}_{i-1}[A_{i}m_{i}+B_{i}1_{(U_{i}\leq p_{i})}]=\mathbb{E}_{i-1}%
        			\left[ \left( A_{i}1_{(U_{i}>p_{i})}+\frac{1}{2}B_{i}1_{(U_{i}\leq
        				p_{i})}\right) \bar{m}_{i}\right] .
        		\end{equation*}
        	\end{proof}
        	
        	\begin{proof}[Proof of Proposition \protect\ref{lem:pf}]
        		The proof uses backward induction starting at $i= n $ and is simply an algebraic rearrangement of
        		Corollary \ref{cor:23} with $f:=f_{i}$ and $e_{i}\equiv 1$.  After the explicit
        		evaluation of the derivatives $\partial _{i-1}\left( \frac{X_{i-1}-L}{\sigma
        			_{i-1}}\right) $ and $\partial _{{i-1}}\left( \frac{b_{{i-1}}\left( X_{{i-1}%
        			}-L\right) }{a_{{i-1}}}\right) $ in the above formulas and factorization of
        		the random variable $\bar{m}_{i}$, one obtains the result. The integrability of the obtained formulas follows by the tower property. 
        	\end{proof}
        	
        	\subsection{Proofs of the lemmas in Section \protect\ref{sec:4}}
        	
        	\label{sec:gs}

        	\begin{proof}[Proof of Lemma \protect\ref{lem:333}]
        		For $k\in \mathbb{N}_{0}$, we do the calculation in parts. First, using \eqref{eq:defbXunderPtilde}, the $ \mathcal{F}_{i-1} $ conditional law of $ X_i $ and doing
        		appropriate change of variables one obtains: 
        		\begin{equation*}
        			\widetilde{\mathbb{E}}_{i-1}\left[ {Z}_{i}^{k}1_{\left( {X}_{i}>L\right) }\right]
        			=\widetilde{\mathbb{E}}_{i-1}\left[ \left( \frac{X_{i}-X_{i-1}}{\sigma _{i-1}}%
        			\right) ^{k}1_{\left( {X}_{i}>L\right) }\right] ={\Delta ^{(k+1)/2}}%
        			\int_{-X_{i-1}^{L,\sigma }}^{\infty }z^{k}\frac{e^{-\frac{z^{2}}{2}}%
        			}{\sqrt{2\pi \Delta }}dz.
        		\end{equation*}%
        		In a similar fashion, one gets 
        		\begin{equation*}
        			\widetilde{\mathbb{E}}_{i-1}\left[ {Z}_{i}^{k}1_{\left( {X}_{i}>L\right)
        			}1_{\left( U_{i}\leq p_{i}\right) }\right] ={\Delta ^{(k+1)/2}}%
        			\int_{X_{i-1}^{L,\sigma }}^{\infty }\left( z-\frac{2\left(
        				X_{i-1}-L\right) }{\sigma _{i-1}\sqrt{\Delta }}\right) ^{k}\frac{e^{-\frac{%
        						z^{2}}{2}}}{\sqrt{2\pi \Delta }}dz,
        		\end{equation*}%
        		%
        		%
        		%
        		%
        		%
        		%
        		%
        		%
        		%
        		%
        		%\begin{eqnarray*}
        		%\widetilde{\mathbb{E}}_{i-1}\left[ {Z}_{i}^{k}\bar{m}_{i}\right] &=&\widetilde{\mathbb{E}}_{i-1}%
        		%\left[ {Z}_{i}^{k}1_{\left( {X}_{i}>L\right) }\left( 1+1_{\left( U_{i}\leq
        		%p_{i}(X_{i-1},X_{i})\right) }\right) \right] \\
        		%&=&\widetilde{\mathbb{E}}_{i-1}\left[ \left( \frac{X_{i}-X_{i-1}}{\sigma _{i-1}}\right)
        		%^{k}1_{\left( {X}_{i}>L\right) }\left( 1+p_{i}(X_{i-1},X_{i})\right) \right]
        		%\\
        		%&=&\frac{1}{\sigma _{i-1}\sqrt{2\pi \Delta }}\int_{L}^{\infty }\left( \frac{%
        		%x_{i}-X_{i-1}}{\sigma _{i-1}}\right) ^{k}e^{-\frac{\left(
        		%x_{i}-X_{i-1}\right) ^{2}}{2a_{i-1}\Delta }}\left( 1+e^{-2\frac{\left(
        		%X_{i-1}-L\right) \left( x_{i}-L\right) }{a_{i-1}\Delta }}\right) dx_{i} \\
        		%&=&\frac{\Delta ^{k/2}}{\sigma _{i-1}\sqrt{2\pi \Delta }}\int_{L}^{\infty
        		%}\left( \frac{x_{i}-X_{i-1}}{\sigma _{i-1}\sqrt{\Delta }}\right) ^{k}\left(
        		%e^{-\frac{\left( x_{i}-X_{i-1}\right) ^{2}}{2a_{i-1}\Delta }}+e^{-\frac{%
        		%\left( x_{i}-(2L-X_{i-1}\right) )^{2}}{2a_{i-1}\Delta }}\right) dx_{i} \\
        		%&=&\frac{\sigma _{i-1}\Delta ^{(k+1)/2}}{\sigma _{i-1}\sqrt{2\pi \Delta }}%
        		%\left( \int_{-X_{i-1}^{L,\sigma }}^{\infty }x_{i}^{k}e^{-\frac{x_{i}^{2}}{2}%
        		%}dx_{i}+\int_{X_{i-1}^{L,\sigma }}^{\infty }\left( x_{i}-\frac{2\left(
        		%X_{i-1}-L\right) }{\sigma _{i-1}\sqrt{\Delta }}\right) ^{k}e^{-\frac{%
        		%x_{i}^{2}}{2}}dx_{i}\right)
        		%\end{eqnarray*}
        		which gives us \eqref{lem:3} and (\ref{ziid0}) after explicit computation of
        		the above integrals. Similarly, for $k=1$, \eqref{ziid} follows from the
        		equality 
        		\begin{equation*}
        			\widetilde{\mathbb{E}}_{i-1}\left[ {Z}_{i}\bar{m}_{i}\right] =\widetilde{\mathbb{E}}%
        			_{i-1}\left[ {Z}_{i}1_{\left( {X}_{i}>L\right) }\right] +\widetilde{\mathbb{E}}%
        			_{i-1}\left[ {Z}_{i}1_{\left( {X}_{i}>L\right) }1_{\left( U_{i}\leq
        				p_{i}\right) }\right] .
        		\end{equation*}%
        		%
        		%
        		%
        		%
        		%
        		%
        		%
        		%
        		%
        		%
        		%we get that 
        		%\begin{equation*}
        		%\widetilde{\mathbb{E}}_{i-1}\left[ {Z}_{i}\bar{m}_{i}\right] =2\sigma _{i-1}\Delta
        		%g_{i}(X_{i-1}^{L})-\frac{2(X_{i-1}^{L})}{\sigma _{i-1}\sqrt{2\pi }}%
        		%\int_{X_{i-1}^{L,\sigma }}^{\infty }e^{-\frac{x_{i}^{2}}{2}}dx_{i}
        		%\end{equation*}
        		
        		For $k=2$, one gets the result in a similar way. Moreover, for even $k$, we
        		have the following upper bound %\begin{equation*}
        		%\widetilde{\mathbb{E}}_{i-1}\left[ {Z}_{i}^{2}\bar{m}_{i}\right] =\Delta -4\Delta \left(
        		%X_{i-1}-L\right) g_{i}(X_{i-1}^{L})+\Delta \left( \frac{2\left(
        		%X_{i-1}-L\right) }{\sigma _{i-1}\sqrt{\Delta }}\right) ^{2}\bar{\Phi}\left(
        		%X_{i-1}^{L,\sigma }\right) ,
        		%\end{equation*}%
        		%which gives us (\ref{ziid}). Moreover, for even $k$ 
        		\begin{eqnarray*}
        			\widetilde{\mathbb{E}}_{i-1}\left[ {Z}_{i}^{k}\bar{m}_{i}\right] &=&\frac{\Delta
        				^{\frac{ k}{2}}}{\sqrt{2\pi }}\left( \int_{X_{i-1}^{L,\sigma }}^{\infty } z^{k}e^{-\frac{z^{2}}{2}
        			}dz+\sum_{j=1}^{k}\binom{k}{j}\left( X_{i-1}^{L,\sigma }\right)
        			^{j}2^{j}\int_{X_{i-1}^{L,\sigma }}^{\infty }z^{k-j}e^{-\frac{%
        					z^{2} }{2}}dz\right) \\
        			&\leq &\frac{\Delta ^{\frac{k}{2}}}{\sqrt{2\pi }}\left( \int z^{k}e^{- 
        				\frac{z^{2}}{2}}dz+\sum_{j=1}^{k}\binom{k}{j}2^{j}\int_{X_{i-1}^{L,
        					\sigma }}^{\infty }z^{k}e^{-\frac{z^{2}}{2}}dz\right) \\
        			&\leq &\frac{3^{k}\Delta ^{\frac{k}{2}}}{\sqrt{2\pi }}\int z^{k}e^{- 
        				\frac{z^{2}}{2}}dz.
        		\end{eqnarray*}
        		Hence (\ref{ziid3}) holds with $c_{k}=\frac{3^{2k}}{\sqrt{2\pi }}\int
        		z^{2k}e^{-\frac{z^{2}}{2}}dz$. Other estimates follow in a
        		similar manner. The bounds for 
        		$ \vartheta_{i-1} $ follow because for  
        		$a=X_{i-1}^{L,\sigma }\geq 0$, we have using \eqref{def:gi} 
        		\begin{equation}
        			\label{eq:A}
        			a\bar{\Phi}\left( a\right) =\frac{a}{\sqrt{2\pi }}\int_{a}^{\infty }e^{-%
        				\frac{x_{i}^{2}}{2}}dx_{i}\leq \frac{1}{\sqrt{2\pi }}\int_{a}^{\infty
        			}x_{i}e^{-\frac{x_{i}^{2}}{2}}dx_{i}=\frac{1}{\sqrt{2\pi }}e^{-\frac{a^{2}}{2%
        			}}={ \sqrt{\Delta }\sigma _{i-1}g_{i}(X_{i-1}).  }
        		\end{equation}
        	\end{proof}
        	
        	The following Lemma is needed for the proof of Lemma \ref{cumulativecontrol}
        	and other subsequent results.
        	
        	\begin{lemma}
        		\label{lem:essb} Let $F:\mathbb{R}\mapsto (0,\infty )$ be a positive valued
        		function with Gaussian decay at infinity. In other words, there exists $c>0$
        		such that 
        		\begin{equation}
        			F(x)\leq e^{-\frac{\left\vert x\right\vert ^{2}}{2c}}.  \label{eq:Fbu'}
        		\end{equation}%
        		For, $j,k\in \mathbb{N}$, $j+k<n$, $\ $define 
        		\begin{equation*}
        			\Xi _{jk}:=\Delta ^{\frac{1}{2}}\sum_{i=j}^{j+k}\widetilde{\mathbb{E}}\left[
        			\left. F(X_{i}^{L,\sigma })\bar{M}^n_{n}\right\vert \mathcal{F}_{{j}}\right] 
        			(\bar{M}^n_{j})^{-1}.
        		\end{equation*}%
        		Then, for $q\in \mathbb{N}$, there exists a
        		constant $C$ independent of $n$ and $T$ such that 
        		\begin{equation}
        			\sup_{x\geq L}\widetilde{\mathbb{E}}_{0,x}\left[ \max_{j\leq n-k-1}\left( \Xi _{jk}\right) ^{q}%
        			\bar{M}^n_{n}\right] ^{\frac{1}{q}}\leq C\sqrt{\left( k+1\right) \Delta }.
        			\label{eq:condb}
        		\end{equation}%
        		In particular, we have that 
        		\begin{equation}
        			\Delta ^{\frac{1}{2}}\sup_{x\geq L }\widetilde{\mathbb{E}}_{0,x}\left[ \left(
        			\sum_{i=0}^{n}F(X_{i}^{L,\sigma })\right) ^{q}\bar{M}^n_{n}\right] ^{\frac{1}{q%
        			}}\leq C\sqrt{T}.  \label{eq:essb}
        		\end{equation}
        	\end{lemma}
        	
        	\begin{proof}
        		Following \eqref{eq:reft}, we have that  
        		\begin{equation*}
        			\widetilde{\mathbb{E}}_{0,x}\left[ \max_{j\leq n-k-1}\left( \Xi _{jk}\right) ^{q}%
        			\bar{M}^n_{n}\right] ^{\frac{1}{q}}=\widetilde{\mathbb{E}}_{0,x}\left[ \max_{j\leq
        				n-k-1}\left( \widetilde{\Xi}_{jk}\right) ^{q}\right] ^{\frac{1}{q}},
        		\end{equation*}%
        		where 
        		\begin{equation}
        			\widetilde{\Xi}_{jk}=\Delta ^{\frac{1}{2}}\sum_{i=j}^{j+k}\widetilde{\mathbb{E}}%
        			_{0,x}\left[ \left. {F}\left( \frac{\mathcal{X}_{i}-L}{\sigma (%
        				\mathcal{X}_{i})\sqrt{\Delta }}\right) \right\vert \widetilde{\mathcal{F}}%
        			_{t_{j}}\right] ,  
        			\label{eq:cndide}
        		\end{equation}%
        		and the process $(\mathcal{X},\Lambda )$ is the solution of the reflected
        		Euler scheme as defined in Section \eqref{app:res}, starting from $\mathcal{X}%
        		_{0}=x$. In \eqref{eq:cndide}, the filtration $(\widetilde{\mathcal{F}}_{t})_t$ is the
        		filtration generated by the reflected Euler scheme $(\mathcal{X},\Lambda )$
        		and, via a slight abuse of notation\footnote{%
        			Note that $\widetilde{\mathbb{E}}_{0,x}$ also denotes the expectation with
        			respect to which the process $X^n$ starts from $x\geq L$ at time 0.} $\widetilde{%
        			\mathbb{E}}_{0,x}\left[ \cdot \right] $ also denotes the expectation with respect to
        		which the process $(\mathcal{X},\Lambda )$ starts from $x\geq L$ at time 0.
        		For $t\in \left[ t_{i},t_{i+1}\right] $, we have that 
        		\begin{equation}
        			\widetilde{\mathbb{E}}_{0,x}\left[ \left. \exp \left( -\frac{\left( \mathcal{X}%
        				_{t}-L\right) ^{2}}{2a\left( \mathcal{X}_{i}\right) c\Delta }\right)
        			\right\vert \widetilde{\mathcal{F}}_{t_{i}}\right] =\int_{L}^{\infty
        			}\exp \left( -\frac{\left( y_{t}-L\right) ^{2}}{2a\left( \mathcal{X}%
        				_{i}\right) c\Delta }\right) P^i_{\mathcal{X}_{t}}\left( dy_{t}\right) ,
        			\label{int1}
        		\end{equation}%
        		where $P^i_{\mathcal{X}_{t}}\left( dy_{t}\right) $ is the conditional law of $%
        		\mathcal{X}_{t}$ given $\widetilde{\mathcal{F}}_{t_{i}}$ i.e., 
        		\begin{equation*}
        			P^i_{\mathcal{X}_{t}}\left( dy_{t}\right) =\frac{1}{\sqrt{2\pi a\left( 
        					\mathcal{X}_{i}\right) \left( t-t_{i}\right) }}\left( \!\!\exp \left( -\frac{%
        				\left( y_{t}-\mathcal{X}_{i}\right) ^{2}}{2a\left( \mathcal{X}_{i}\right)
        				\left( t-t_{i}\right) }\right) \!\!+\exp \left(\!\! -\frac{\left( y_{t}-\left( 2L-%
        				\mathcal{X}_{i}\right) \right) ^{2}}{2a\left( \mathcal{X}_{i}\right) \left(
        				t-t_{i}\right) }\right) \right) \!\!dy_{t}.
        		\end{equation*}%
        		We integrate in (\ref{int1}) to get 
        		\begin{align*}
        			\widetilde{\mathbb{E}}_{0,x}\left[ \left. \exp \left( -\frac{\left( \mathcal{X}%
        				_{t}-L\right) ^{2}}{2a\left( \mathcal{X}_{i}\right) c\Delta }\right)
        			\right\vert \widetilde{\mathcal{F}}_{t_{i}}\right] =&\frac{\sqrt{c\Delta }}{%
        				\sqrt{c\Delta +t-t_{i-1}}}e^{-\frac{\left( L-\mathcal{X}_{i}\right) ^{2}}{%
        					2a\left( \mathcal{X}_{i}\right) \left( c\Delta +t-t_{i}\right) }}\left( \bar{%
        				\Phi}\left( b_{t}\right) +\bar{\Phi}\left( -b_{t}\right) \right) \\
        			\geq &\frac{\sqrt{c\Delta }}{\sqrt{c\Delta +t-t_{i}}}e^{-\frac{\left( L-%
        					\mathcal{X}_{i}\right) ^{2}}{2a\left( \mathcal{X}_{i}\right) c\Delta }}\geq 
        			\frac{\sqrt{c}}{\sqrt{c+1}}F\left( \frac{\mathcal{X}_{i}-L}{\sigma
        				\left( \mathcal{X}_{i}\right) \sqrt{\Delta }}\right)
        		\end{align*}
        		where $b_{t}= \frac{\sqrt{c\Delta}(L-\mathcal{X}_{i})}{\sqrt{a\left( \mathcal{X}_{i}\right)(c\Delta +t-t_{i})(t-t_{i})}}$%
        		. Hence if we denote by $F_{1}(x)=e^{-x^{2}/(2c)}$ , we have
        		\begin{align*}
        			\widetilde{\mathbb{E}}_{0,x}\left[ \left. \int_{t_{i}}^{t_{i+1}}F_{1}\left( 
        			\frac{\mathcal{X}_{t}-L}{\sigma \left( \mathcal{X}_{i}\right) \sqrt{%
        					\Delta }}\right) dt\right\vert \widetilde{\mathcal{F}}_{t_{i}}\right]
        			=&\int_{t_{i}}^{t_{i+1}}\widetilde{\mathbb{E}}_{0,x}\left[ \left. \exp \left( -%
        			\frac{\left( \mathcal{X}_{t}-L\right) ^{2}}{2a\left( \mathcal{X}_{i}\right)
        				c\Delta }\right) \right\vert \widetilde{\mathcal{F}}_{t_{i}}\right] dt \\
        			\geq &F\left( \frac{\mathcal{X}_{i}-L}{\sigma \left( \mathcal{X}%
        				_{i}\right) \sqrt{\Delta }}\right) \frac{\sqrt{c}}{\sqrt{c+1}}\Delta .
        		\end{align*}
        		We use the occupation time formula so that 
        		\begin{align}
        			a\left( \mathcal{X}_{t_{i}}\right) \int_{t_{i}}^{t_{i+1}}F_{1}\left( \frac{%
        				\mathcal{X}_{t}-L}{\sigma \left( \mathcal{X}_{i}\right) \sqrt{\Delta }}%
        			\right) dt=&\int_{t_{i}}^{t_{i+1}}F_{1}\left( \frac{\mathcal{X}_{t}-L}{%
        				\sigma \left( \mathcal{X}_{i}\right) \sqrt{\Delta }}\right) d\left[ \mathcal{X}%
        			\right] _{t}  \notag \\
        			=&\int_{L}^{\infty }F_{1}\left( \frac{y-L}{\sigma \left( \mathcal{X}%
        				_{i}\right) \sqrt{\Delta }}\right) \bar{\Lambda}_{t_{i}\mapsto
        				t_{i+1}}^{y}dy.  \label{eq:ineq}
        		\end{align}

        		In (\ref{eq:ineq}), we use the process $\bar{\Lambda}^{y}$ which denotes the
        		local time at $y$ of the solution of the reflected Euler scheme $(\mathcal{X}%
        		,\Lambda )$ as defined in Section \ref{app:res}. Hence: 
        		\begin{eqnarray}
        			\label{eq:B}
        			&&\Delta ^{\frac{1}{2}}F\left( \frac{\mathcal{X}_{i}-L}{%
        				\sigma (\mathcal{X}_{i})\sqrt{\Delta }}\right) \leq \sqrt{c^{-1}+1}\frac{1}{\Delta ^{\frac{1}{2}}}\left\vert
        			\left\vert \frac{1}{a}\right\vert \right\vert _{\infty }a\left( \mathcal{X}%
        			_{t_{i}}\right) \widetilde{\mathbb{E}}%
        			_{0,x}\left[ \left. \int_{t_{i}}^{t_{i+1}}F_{1}\left( \frac{\mathcal{X}_{t}-L%
        			}{\sigma _{i}\left( \mathcal{X}_{i}\right) \sqrt{\Delta }}\right)
        			dt\right\vert \widetilde{\mathcal{F}}_{t_{i}}\right] .
        		\end{eqnarray}%
        		The above estimates \eqref{eq:ineq} and \eqref{eq:B} imply that for ${\ C= 2\sqrt{c^{-1}+1}\left\vert
        			\left\vert \frac{1}{a}\right\vert \right\vert _{\infty } }$ 
        		\begin{eqnarray*}
        			\widetilde{\Xi}_{jk} &=&\Delta ^{\frac{1}{2}}\sum_{i=j} ^{{j+k }}\widetilde{\mathbb{E%
        			}}_{0,x}\left[ \left. F\left( \frac{\mathcal{X}_{i}-L}{\sigma (\mathcal{X}%
        				_{i})\sqrt{\Delta }}\right) \right\vert \widetilde{\mathcal{F}}_{t_{j}}\right] \\
        			&\leq &C\int_{L}^{\infty }\frac{1}{\Delta ^{\frac{1}{2}}}\exp \left( -\frac{%
        				\left( y-L\right) ^{2}}{2c\left\vert \left\vert a\right\vert \right\vert
        				_{\infty }\Delta }\right) \widetilde{\mathbb{E}}_{0,x}\left[ \left. \bar{\Lambda}_{t_{j}\mapsto t_{{j+k+1 }}}^{y}\right\vert \widetilde{\mathcal{F}}%
        			_{t_{j}}\right] dy.
        		\end{eqnarray*}%
        		Observe that using \eqref{moduluso}, we have 
        		\begin{equation*}
        			\widetilde{\Xi}_{jk}        \leq C\widehat{c}\sqrt{\left( k+1\right) \Delta }\int_{L}^{\infty }\frac{1}{\Delta ^{%
        					\frac{1}{2}}}\exp \left( -\frac{\left( y-L\right) ^{2}}{2c\|a\|_\infty\Delta 
        			}\right) dy.  
        		\end{equation*}%
        		From here, taking the power $ q $, the maximum and the expectation on both sides (\ref{eq:condb}) follows.
        		
        		For $k=n-1$ we can deduce from (\ref{eq:condb}),
        		%               that 
        		%               \begin{equation*}
        		%               \Delta ^{\frac{1}{2}}\widetilde{\mathbb{E}}_{0,x}\left[ \left(
        		%               \sum_{i=0}^{n-1}F(X_{i}^{L,\sigma })\right) ^{q}\bar{M}^n_{n}\right] ^{\frac{1%
        		%                       }{q}}\leq C\sqrt{T}.
        		%               \end{equation*}%
        		that  (\ref{eq:essb}) is satisfied after observing that the last term in the sum in (%
        		\ref{eq:essb}), can be controlled too as
        		\begin{align*}
        			\widetilde{\mathbb{E}}_{0,x}\left[
        			\left( F(X_{t_{n}}^{L,\sigma })\right) ^{q}\bar{M}^n_{n}\right] \leq 1.
        		\end{align*}
        	\end{proof}
        	\begin{proof}[Proof of Lemma \protect\ref{cumulativecontrol}]
        		The estimates in (\ref{control2}) and (\ref{control3}) follow directly from
        		( \ref{ziid3}), so we only need to justify (\ref{control1}).
        		%               Note that for 
        		%%               From %
        		%%              \eqref{ziid} , we have that 
        		%%              \begin{equation}
        		%%              \widetilde{\mathbb{E}}_{i-1}\left[ {Z}_{i}\bar{m}_{i}\right] =2\sqrt{\Delta }%
        		%%              \left( \sqrt{\Delta }\sigma _{i-1}g_{i-1}(X_{i-1}^{L})-X_{i-1}^{L,\sigma }%
        		%%              \bar{\Phi}\left( X_{i-1}^{L,\sigma }\right) \right) \geq 0,  \label{eq:+}
        		%%              \end{equation}%
        		%%              since $g_{i+1}(z)=\frac{1}{\sigma _{i}\sqrt{2\pi \Delta }}\exp \left( -\frac{%
        		%%                      z^{2}}{2a_{i}\Delta }\right) $ and, for
        		%                $a=X_{i-1}^{L,\sigma }\geq 0$, we have using \eqref{def:gi} 
        		%               \begin{equation*}
        		%               a\bar{\Phi}\left( a\right) =\frac{a}{\sqrt{2\pi }}\int_{a}^{\infty }e^{-%
        		%                       \frac{x_{i}^{2}}{2}}dx_{i}\leq \frac{1}{\sqrt{2\pi }}\int_{a}^{\infty
        		%               }x_{i}e^{-\frac{x_{i}^{2}}{2}}dx_{i}=\frac{1}{\sqrt{2\pi }}e^{-\frac{a^{2}}{2%
        		%               }}=\sqrt{\Delta }\sigma _{i-1}g_{i+1}(a\sqrt{\Delta }\sigma _{i-1}).
        		%               \end{equation*}%
        		Since $X_{i-1}^{L}=\sqrt{\Delta }%
        		\sigma _{i-1}X_{i-1}^{L,\sigma }$ one obtains from \eqref{ziid} that $0\leq  \widetilde{\mathbb{E}}_{i-1}\left[ {Z}_{i}\bar{m}_{i}\right] \leq 2\sigma
        		_{i-1}\Delta g_{i}(X_{i-1}^{L})$. Therefore  (\ref{control1})
        		follows from (\ref{eq:condb}) with $q=1$ and $F\left( x\right) =e^{-\frac{%
        				x^{2}}{2}}$. 
        	\end{proof}
        	
        	\begin{proof}[Proof of Lemma \protect\ref{lemmaBm}] 
        		We divide the estimation for $ \theta _{t}^{R} $ into two parts using \eqref{ziid} as follows
        		\begin{align}
        			\theta _{t}^{R}=&\sum_{\left\{ i,t_{i}\leq t\right\} }\left((\Delta_iR^n)^2-\Delta\right)=\sum_{\left\{ i,t_{i}\leq t\right\} }\theta _{i}-4\sum_{\left\{ i,t_{i}\leq t\right\} }\Delta\left( X_{i-1}^{L,\sigma }\vartheta _{i-1}+ \vartheta _{i-1}^{2}\right).\label{eq:65}
        		\end{align}
        		Here, we have used the orthogonal sequence of random variables $ \theta_i $ defined as:
        		\begin{align*}
        			\theta_i:=&(\Delta_iR^n)^2-\widetilde{\mathbb{E}}_{i-1}\left[ (\Delta_iR^n)^2\bar{m}_{i}\right]\\
        			=&\left( {Z}_{i}-\widetilde{\mathbb{E}}_{i-1}\left[ {Z}%
        			_{i}\bar{m}_{i}\right] \right) ^{2}-\widetilde{\mathbb{E}}_{i-1}\left[ \left(
        			\left( {Z}_{i}-\widetilde{\mathbb{E}}_{i-1}\left[ {Z}_{i}\bar{m}_{i}\right]
        			\right) ^{2}\right) \bar{m}_{i}\right]  .
        		\end{align*}

        		Observe that, by orthogonality and the controls of the fourth moments in \eqref{ziid3}: 
        		\begin{equation}
        			\widetilde{\mathbb{E}}\left[ \left\vert \sum_{\left\{ i,t_{i}\leq t\right\}
        			}\theta _{i}\right\vert \bar{M}^n_{n}\right] ^{2}\leq \sum_{\left\{ i,t_{i}\leq
        				t\right\} }\widetilde{\mathbb{E}}\left[ \theta _{i}
        			^{2}\bar{m}_{i}\right] \leq c\Delta.  \label{bmc1}
        		\end{equation}%
        		%               Using (\ref{ziid}) and (\ref{ziid}), we deduce that 
        		%               \begin{align*}
        		%               \widetilde{\mathbb{E}}_{i-1}\left[ {Z}_{i}\bar{m}_{i}\right] =&2\sqrt{\Delta }%
        		%               \vartheta _{i-1} \\
        		%               \widetilde{\mathbb{E}}_{i-1}\left[ {Z}_{i}^{2}\bar{m}_{i}\right] =&\Delta
        		%               -4\Delta X_{i-1}^{L,\sigma }\vartheta _{i-1}.
        		%               \end{align*}
        		%               where 
        		%               \begin{equation*}
        		%               0\leq \vartheta _{i-1}\equiv \sigma _{i-1}\sqrt{\Delta }%
        		%               g_{i}(X_{i-1}^{L})-X_{i-1}^{L,\sigma }\bar{\Phi}\left( X_{i-1}^{L,\sigma
        		%               }\right) \leq \sigma _{i-1}\sqrt{\Delta }g_{i}(X_{i-1}^{L}).
        		%               \end{equation*}%
        		%               Next observe that 
        		%               \begin{align*}
        		%                       \widetilde{\mathbb{E}}_{i-1}\left[ \theta _{i}\bar{m}_{i}\right] =&\widetilde{%
        		%                               \mathbb{E}}_{i-1}\left[ {Z}_{i}^{2}\bar{m}_{i}\right] -\Delta -\left( \widetilde{%
        		%                               \mathbb{E}}_{i-1}\left[ {\ Z}_{i}\bar{m}_{i}\right] \right) ^{2} \\
        		%                       =&-4\Delta X_{i-1}^{L,\sigma }\vartheta _{i-1}-4\Delta \vartheta _{i-1}^{2}
        		%               \end{align*}
        		Furthermore for the second term in \eqref{eq:65}, using \eqref{ziid} and Lemma \ref{lem:essb}, we obtain that 
        		\begin{align}
        			\widetilde{\mathbb{E}}\left[ \left\vert \sum_{i=1}^{n}\Delta  \vartheta _{i-1}^{2} \right\vert \bar{M}^n_{n}%
        			\right] \leq &c\Delta \widetilde{\mathbb{E}}\left[ \left(
        			\sum_{i=0}^{n}e^{-\left( X_{i-1}^{L,\sigma }\right) ^{2}}\right) \bar{M}^n_{n}%
        			\right] \leq c\sqrt{\Delta }  \label{bmc2} \\
        			\widetilde{\mathbb{E}}\left[ \left\vert \sum_{i=1}^{n}\Delta  \left( X_{i-1}^{L,\sigma }\vartheta _{i-1}\right) 
        			\right\vert \bar{M}^n_{n}\right] \leq &c\Delta \sum_{i=1}^{n}\widetilde{\mathbb{E}%
        			}\left[ \left( \sum_{i=0}^{n}X_{i-1}^{L,\sigma }e^{-\frac{\left(
        					X_{i-1}^{L,\sigma }\right) ^{2}}{2}}\right) \bar{M}^n_{n}\right]  \notag \\
        			\leq &c\Delta \sum_{i=1}^{n}\widetilde{\mathbb{E}}\left[ e^{-\frac{\left(
        					X_{i-1}^{L,\sigma }\right) ^{2}}{4}}\bar{M}^n_{n}\right] \leq c\sqrt{\Delta }%
        			.\ \ \ \ \ \   \label{bmc3}
        		\end{align}
        		The inequalities (\ref{bmc1}), (\ref{bmc2}), (\ref{bmc3}) give us (\ref%
        		{Bmcontrol}).
        		Next, note that from \eqref{eq:gammainZ}, we have  
        		\begin{align*}
        			\theta _{t}^{\Gamma }= &\sum_{\left\{ i>0,t_{i}\leq t\right\} }(\Gamma
        			_{t_{i}}^{n}-\Gamma _{t_{i-1}}^{n})^{2}=\sum_{\left\{ i>0,t_{i}\leq
        				t\right\} }\gamma _{i}^{2}.
        		\end{align*}
        		
        		Therefore, similar to (\ref{bmc2}) and (\ref{bmc3}) one deduces that 
        		\begin{equation*}
        			\sup_{t\in \left[ 0,T\right] }\widetilde{\mathbb{E}}\left[ \theta _{t}^{\Gamma }%
        			\bar{M}^n_{n}\right] =\widetilde{\mathbb{E}}\left[ \left( \sum_{i=1}^{n}\gamma
        			_{i}^{2}\right) \bar{M}^n_{n}\right] \leq c\sqrt{\Delta }.
        		\end{equation*}%
        		Hence (\ref{Rmcontrol}) holds true.
        	\end{proof}

        	\begin{remark}\label{yar} 
        		
        		Besides (\ref{Bmcontrol}), the process $\theta ^{R}$
        		also satisfies the slightly stronger control 
        		\begin{equation}
        			\tilde{\mathbb{E}}\left[ \left( \sup_{t\in \left[ 0,T\right] }\theta
        			_{t}^{R}\right) ^{2}\bar{M}_{n}^{n}\right] \leq c\left( T\right) \Delta .
        			\label{Bmcontrol'}
        		\end{equation}%
        		To show this one uses the martingale property of the process $t\rightarrow
        		\sum_{\left\{ i,t_{i}\leq t\right\} }\theta _{i}$, and the Doob's maximal
        		inequality to deduce that 
        		\begin{equation*}
        			\tilde{\mathbb{E}}\left[ \sup_{t\in \left[ 0,T\right] }\left\vert
        			\sum_{\left\{ i,t_{i}\leq t\right\} }\theta _{i}\right\vert ^{2}\bar{M}%
        			_{n}^{n}\right] \leq c\left( p,T\right) \tilde{\mathbb{E}}\left[ \left\vert
        			\sum_{\left\{ i,t_{i}\in \left[ 0,T\right] \right\} }\theta
        			_{i}^{2}\right\vert \bar{M}_{n}^{n}\right] \leq c\left( T\right) \Delta.
        		\end{equation*}%
        		Moreover, using \eqref{eq:essb}, we have 
        		\begin{eqnarray}
        			\tilde{\mathbb{E}}\left[ \left\vert \sum_{i=1}^{n}\Delta 
        			\vartheta _{i-1}^{2} \right\vert ^{2}\bar{M}%
        			_{n}^{n}\right]  &\leq &c\Delta^2 \tilde{\mathbb{E}}\left[ \left(
        			\sum_{i=0}^{n}e^{-\left( X_{i-1}^{L,\sigma }\right) ^{2}}\right) ^{2}\bar{M}%
        			_{n}^{n}\right] \leq c{\Delta },  \label{bmc2'} \\
        			\tilde{\mathbb{E}}\left[ \left\vert \sum_{i=1}^{n}\Delta 
        			\left( X_{i-1}^{L,\sigma }\vartheta _{i-1}\right) 
        			\right\vert ^2\bar{M}_{n}^{n}\right]  &\leq &c\Delta^2 \tilde{%
        				\mathbb{E}}\left[ \left( \sum_{i=0}^{n}X_{i-1}^{L,\sigma }e^{-\frac{\left(
        					X_{i-1}^{L,\sigma }\right) ^{2}}{2}}\right) ^{2}\bar{M}_{n}^{n}\right] \leq c%
        			{\Delta },\ \ \ \ \ \   \label{bmc3'}
        		\end{eqnarray}%
        		which gives (\ref{Bmcontrol'}%
        		). Similarly, one can show that %
        		\begin{equation}
        			\tilde{\mathbb{E}}\left[ \left( \sup_{t\in \left[ 0,T\right] }\theta
        			_{t}^{\Gamma }\right) ^{2}\bar{M}_{n}^{n}\right] \leq c\left( T\right)
        			\Delta .  \label{bmcontrol''}
        		\end{equation}

        	\end{remark}

        	Now we prove that we can neglect the terms related to $\widehat{h} $ (defined in \eqref{habar}) when
        	studying the limits of \eqref{fprime} which is used in the proof of Theorem %
        	\ref{th:ch}.
        	
        	\begin{lemma}
        		\label{additionalterm}We have that for any $\epsilon>0 $ and $x_0-\epsilon 
        		>L $  
        		\begin{equation*}
        			\lim_{\Delta \mapsto 0 }\sup_{|x-x_0|<\epsilon }\sum_{i=1}^{n}\widetilde{\mathbb{%
        					E} }\left[ f_{i}\left( X_{i}^{n,x}\right) {{\mathcal{K}}}_{i}E_{i-1}^{n} 
        			\widehat{ h}_{i}\bar{M}^n _{i} \right] =0. %\label{limadterm}
        		\end{equation*}
        	\end{lemma}
        	
        	\begin{proof}
        		Recall that the function $f_{i}$ satisfies $
        		f_{i}(L)=0 $. It follows from \eqref{habar} that  
        		\begin{align*}
        			\widetilde{\mathbb{E}}_{i-1}\left[ \left( f_{i}\left( X_{i}\right)
        			-f_{i}\left( L\right) \right)e^{{\ \kappa}_{i}} \widehat{h}_{i}\bar{M}^n_{i} %
        			\right] = \widetilde{\mathbb{E}}_{i-1}\left[\varpi(f_i,X_i) Y_i\bar{m}_i\right].
        		\end{align*}
        		
        		Here, $Y_i:=X_{i-1}^LX_i^Le^{{ \kappa}_{i}}\left( {\varpi }_{i-1}+\partial _{i-1}\left( \frac{
        			b_{i-1}}{a_{i-1}}\right) \right)  1_{(U_{i}\leq p_{i})}$ with $ \tilde{\mathbb{E}}_{i-1}[\varpi(f_i,X_i) Y_i\bar{m}_i]=O_{i-1}^E(\Delta^{1/2}) $. 
        		Therefore the result follows from Lemmas \ref{boundedderivative} and \ref%
        		{lem:35} (C).
        	\end{proof}
        	
        	{\noindent\textbf{Proof of Lemma \ref{boundedderivative}.}} Note that (\ref%
        	{recurrenceformula}) implies that %\textcolor{red}{ I think it is
        	%better to use one of the proposed decompositions in page 22-24 to do this
        	%analysis  } 
        	
        	\begin{equation}  \label{eq:71a}
        		\partial _{x}f_{i}\left( x\right) =\widetilde{\mathbb{E}}_{i,x}\left[ {f}%
        		^{\prime }\left( X_{n}^{n,x}\right) E_{i:n}^{n}{{\mathcal{K}}}%
        		^n_{i:n} \bar{M}^n_{i:n}\right] +\sum_{j=i+1}^{n}\widetilde{\mathbb{E}}_{i,x}\left[ f_{j}\left(
        		X_{j}^{n,x}\right) E_{i:j-1}^{n}h_{j}{{\mathcal{K}}}^n_{i:j}\bar{M}^n_{i:j}
        		\right] .
        	\end{equation}
        	
        	%       Furthermore, for the last term, we have 
        	%       \begin{align*}
        	%               \widetilde{\mathbb{E}}\left[ f_{j}\left( X_{j}^{n,x}\right) E_{i:j-1}^{n}h_{j}\bar{M}^n
        	%               _{i:j}{{\mathcal{K}}}_{i:j}|X_{i}=x\right] =\widetilde{\mathbb{E}}\left[
        	%               \varpi(f_j,X^{n,x}_j) E_{i:j-1}^{n}h_{j}(X^{n,x}_j-L)\bar{M}^n_{i:j}{{\mathcal{%
        	%                                       \ K}}}_{i:j}|X_{i}=x\right]
        	%       \end{align*}
        	
        	We proceed to prove the boundedness of each of the above terms separately.
        	First, there exists a constant $\mathsf{M}$ independent of ${i}$, $n$ and $%
        	x\geq L $ such that 
        	\begin{equation*}  
        		%\label{eq:72ab}
        		\sup_{x\geq L}\left\vert \widetilde{\mathbb{E}}_{i,x}\left[ {f}^{\prime }\left(
        		X_{n}^{n,x}\right) E_{i:n}^{n}{{\mathcal{K}}}^n_{i:n}\bar{M}^n_{i:n} \right]%
        		\right\vert \leq \| f^{\prime }\| _{\infty } \widetilde{\mathbb{E}}_{i,x}\left[|
        		E_{i:n}^{n}|{{\mathcal{K}}}^n _{i:n} \bar{M}^n_{i:n}\right] \leq \mathsf{M}\|
        		f^{\prime }\| _{\infty }e^{\mathsf{M}(T-t_i)}
        	\end{equation*}
        	as $\widetilde{\mathbb{E}}_{i,x}\left[| E_{i:n}^{n}|{{\mathcal{K}}}%
        	^n_{i:n}\bar{M}^n_{i:n} \right] $ is uniformly bounded following from Lemma \ref{moments} and Cauchy-Schwartz inequality.
        	In order to obtain the upper bound for the second term in \eqref{eq:71a}, we
        	use ``path decompositions'' for $\bar{h}_i
        	=1_{(U_i\leq p_i)}$. With this in mind, we define 
        	\begin{align*}
        		\bar{\mathsf{h}}_{j,i}:=& 1_{(U_j\leq p_j, U_k>p_k,k=j+1,...,i)}.
        	\end{align*}
        	As usual, we simplify the notation in the case that $i=n $ as $\bar{\mathsf{h%
        	}}_{j}=\bar{\mathsf{h}}_{j,n} $. The function $\bar{\mathsf{h}}_{j,i}$
        	corresponds to the event that the Euler approximation touches the boundary
        	in the interval $[t_{j-1}, t_j ]$ and does not touch the boundary again
        	before time $t_i$.
        	
        	Note that the following path decomposition properties are satisfied: 
        	\begin{align}
        		\bar{\mathsf{h}}_{i,k}\bar{\mathsf{h}}_{j,k}=&0,\quad i\neq j\leq k ,  \notag
        		\\
        		\sum_{j=1}^{i}\bar{\mathsf{h}}_{j,i}=& \left(1-1_{(U_j>p_j,j=1,...,i)}\right
        		).  \label{eq:tildeha}
        	\end{align}
        	{\ }
        	
        	Now we will use properties of $\bar{\mathsf{h}} $ in order to rewrite for $%
        	\mathcal{B}_{j-1}:=\frac{b_{j-1}}{a_{j-1}}+X_{j-1}^{L}\partial _{j-1}\left( \frac{%
        		b_{j-1}}{a_{j-1}}\right) $: 
        	\begin{align*}
        		&\sum_{j=i+1}^{n}\!\!\widetilde{\mathbb{E}}_{i,x}\!\!\left[ f_{j}\left(
        		X_{j}^{n,x}\right) E_{i:j-1}^{n}\mathcal{B}_{j-1}\bar h_{j}{{\mathcal{K}}}
        		_{i:j}\bar{M}^n_{i:j}\right]= \widetilde{\mathbb{E}}_{i,x}\left[ f(X^{n,x}_n){{%
        				\mathcal{K}}}_{i:n}\bar{M}^n_{i:n} \!\!\sum_{j=i+1}^{n}\!\!E_{i:j-1}^{n}\mathcal{B}_{j-1}\bar{\mathsf{h}}_{j}%
        		\right]\!\!.
        	\end{align*}
        	Next, using $\left|\sum_{j=i+1}^{n}E_{i:j-1}^{n}\mathcal{B}_{j-1}\bar{\mathsf{h}}%
        	_{j}\right|\leq \max_{j\geq i+1}| E_{i:j-1}^{n}\mathcal{B}_{j-1}|$ and the moment
        	properties in Lemma \ref{moments}, we obtain:
        	
        	\begin{align*}
        		\left|\sum_{j=i+1}^{n}\widetilde{\mathbb{E}}\left[ f_{j}\left(
        		X_{j}^{n,x}\right) E_{i:j-1}^{n}\mathcal{B}_{j-1}\bar h_{j}{{\mathcal{K}}}
        		_{i:j}\bar{M}^n_{i:j}|X_{i}=x\right] \right|\leq Ce^{C(T-t_i)}\|f\|_\infty.
        	\end{align*}
        	
        	$\hfill\Box$
        	
        	A similar result is valid for second derivatives and is treated in Section %
        	\ref{app:6.2a}, but requires a deeper analysis of the path
        	decomposition properties used above.

        	\subsection{Some estimates on the Euler reflected scheme}
        	
        	\label{app:res} In this section, we study general moment estimates for the
        	reflected Euler scheme in the case that there is no drift. That is, we
        	assume that $b=0 $.
        	
        	The reflected Euler scheme for $x\geq L$ is defined as the solution $(%
        	\mathcal{X}^{n,x},\Lambda^{n,x})$ of the reflected stochastic equation 
        	\begin{align}
        		\mathcal{X}_{t}^{n,x}=& x+\int_{0}^{t}\sigma \left( \mathcal{X}_{\eta \left(
        			s\right) }^{n,x}\right) dW_{s}+\Lambda ^{n,x}_{t}, \label{reflectedy} \\
        		\mathcal{X}_{t}^{n,x}& \geq L,  \notag \\
        		\Lambda ^{n,x}_{t}=& \int_{0}^{t}1_{( \mathcal{X}_{s}^{n,x}=L) }d|\Lambda ^{n,x}|_{s}. 
        		\notag
        	\end{align}%
        	Here, $\eta \left( s\right) :=t_{i-1}$ for $s\in \lbrack t_{i-1},t_{i})$.
        	Following from Remark \ref{rem:3}, we have the representation formula 
        	\begin{equation*}
        		\widetilde{\mathbb{E}}\left[ f(X_{T}^{n,x})\bar{M}_{i:n}^{n}\Big |\mathcal{F}_{{i%
        		}} \right] =\mathbb{E}\left[ f(\mathcal{X}_{T}^{n,x})|\mathcal{F}_{t_{i}}%
        		\right].
        	\end{equation*}%
        	Let $\bar{\Lambda}^{y}$ denote the local time at $y$ associated to the
        	process $\mathcal{X}^{n,x}$ at $y$. We give next some exponential moment
        	estimates for the triplet $(\mathcal{X}^{n,x}, \Lambda^{n,x}, \bar{\Lambda}^{y})$ which are used to prove uniform moment estimates in Theorems \ref{th:main} and \ref{th:36}.
        	:
        	
        	\begin{lemma}
        		\label{lem:37} There exists a positive constant $C_{q}$ independent of $n$
        		and $x,y,z\in \mathbb{R}$ such that for any $q>0$ 
        		\begin{equation*}
        			e^{-q|x-z|}\mathbb{E}\left[ \exp \left( q\max_{t\in \lbrack 0,T]}|\mathcal{X}
        			_{t}^{n,x}-z|\right) \right] +\mathbb{E}\left[ \exp \left( q\Lambda^{n,x}_T\right) %
        			\right] +\mathbb{E}\left[ \exp \left( q\bar{\Lambda} _{T}^{y}\right) %
        			\right] \leq 2e^{C_{q}T}.
        		\end{equation*}
        		Similarly, for any $0<s<t\leq T$, we have for $\Lambda _{s\mapsto
        			t}^{y}:=\Lambda _{t}^{y}-\Lambda _{s}^{y}$ and  $\widehat{c}:=16(\|\sigma\|_\infty\vee\|\sigma\|^2_\infty) $
        		\begin{align}
        			\mathbb{E}\left[ \left( \Lambda^{n,x} _{s\mapsto t}\right) ^{2}\right] &\leq
        			\widehat{c}\left( t-s\right) ,~~~\mathbb{E}\left[ \Lambda^{n,x}_{s\mapsto t}\right] \leq \widehat{c}
        			\sqrt{t-s},  \label{moduluso} \\
        			\mathbb{E}\left[ \left( \bar{\Lambda}_{s\mapsto t}^{y}\right) ^{2}\right]
        			&\leq \widehat{c}\left( t-s\right) ,~~~\mathbb{E}\left[ \bar{\Lambda}_{s\mapsto t}^y%
        			\right] \leq \widehat{c}\sqrt{t-s}.  \notag
        		\end{align}
        	\end{lemma}
        	
        	\begin{proof}
        		By using the Skorohod equation (see Lemma 6.14, \cite{KS} page 210) we
        		deduce that for $ \Lambda_t\equiv \Lambda^{n,x}_t $
        		\begin{align*}
        			\Lambda _{t} =&\max \left( 0,\max_{s\in \left[ 0,t\right] }\left( -\left(
        			x+\int_{0}^{s}\sigma \left( \mathcal{X}_{\eta \left( p\right) }^{n,x}\right)
        			dW_{p}-L\right) \right) \right) \\
        			\leq &\max_{s\in \left[ 0,t\right] }\left\vert \int_{0}^{s}\sigma \left( 
        			\mathcal{X}_{\eta \left( p\right) }^{n,x}\right) dW_{p}\right\vert .
        		\end{align*}
        		
        		Therefore the result follows once we have exponential estimates for the
        		above maximum. By Dambis-Dubins-Schwartz theorem, we have that $%
        		\int_{0}^{s}\sigma \left( \mathcal{X}_{\eta \left( p\right) }^{n,x}\right)
        		dW_{p}\,$has a representation of the form 
        		\begin{equation*}
        			\int_{0}^{s}\sigma \left( \mathcal{X}_{\eta \left( p\right) }^{n,x}\right)
        			dW_{p}=B_{\int_{0}^{s}a \left( \mathcal{X}_{\eta \left( p\right)
        				}^{n,x}\right) dp},
        		\end{equation*}
        		where $B$ is a standard one-dimensional Brownian motion. Therefore 
        		\begin{equation*}
        			\max_{t\in \left[ 0,T\right] }\left\vert \int_{0}^{t}\sigma \left( \mathcal{%
        				X }_{\eta \left( p\right) }^{n,x}\right) dW_{p}\right\vert =\max_{t\in \left[
        				0,T\right] }\left\vert B_{\int_{0}^{t}a \left( \mathcal{X}_{\eta \left(
        					p\right) }^{n,x}\right) dp}\right\vert \leq \max_{t\in \left[ 0,\| a
        				\|_\infty T\right] }\left\vert B_{t}\right\vert.
        		\end{equation*}
        		It follows that for any $q>0$ and for $x\geq L$ 
        		\begin{align*}
        			\mathbb{E}\left[ \exp \left( q\max_{t\in \lbrack 0,T]}|\mathcal{X}
        			_{t}^{n,x}-z|\right) \right] \leq &e^{q|x-z|}\mathbb{E}\left[ \exp \left(
        			2q\max_{t\in \left[ 0,\| a \|_\infty T\right] }\left\vert B_{t}\right\vert
        			\right) \right] \\
        			=&e^{q|x-z|}\frac{2}{\sqrt{2\pi \| a\| _\infty T}}\int_{0}^{\infty }e^{2qw-\frac{w^{2}}{2 \|a\|_\infty T}}dw
        			\\
        			=& 2\bar{\Phi}\left(-2q\|\sigma\|_\infty\sqrt{T}\right)e^{q|x-z|+2q^{2}\|a\|_\infty T},
        		\end{align*}
        		which gives the claim. The exponential moment estimate for $\Lambda _{T}$
        		follows a similar argument.
        		
        		For the modulus of continuity result (\ref{moduluso}), we use the previous argument which gives 
        		\begin{align*}
        			\Lambda _{s\mapsto t}  
        			\leq &\max_{r\in \left[ s,t\right] }\left\vert \int_{s}^{r}\sigma \left( 
        			\mathcal{X}_{\eta \left( p\right) }^{n,x}\right) dW_{p}\right\vert \leq
        			\max_{r\in \left[ s,t\right] }\left\vert \widetilde{B}_{\int_{s}^{r}a
        				\left( \mathcal{X}_{\eta \left( p\right) }^{n,x}\right) dp}\right\vert
        			=\max_{r\in \left[ 0,\|a\|_\infty\left( t-s\right) \right] }\left\vert 
        			\widetilde{B}_{r}\right\vert
        		\end{align*}
        		
        		where $\widetilde{B}$ is, again, a standard one-dimensional Brownian motion.
        		
        		Now, we would like to extend all the above statements to $\bar{\Lambda}
        		^{y}.$ In order to do this, we use It\^{o}-Tanaka formula (see Theorem 7.1
        		in \cite{KS}) to deduce that 
        		\begin{align*}
        			\bar{\Lambda}_{s\mapsto t}^{y} =&(\mathcal{X}_{t}-y)^{+}-(\mathcal{X}
        			_{s}-y)^{+}-\int_{s}^{t}1_{(\mathcal{X}_{u}>y)}d\mathcal{X}_{u} \\
        			\leq &\left\vert \mathcal{X}_{t}-\mathcal{X}_{s}\right\vert +\left\vert
        			\int_{s}^{t}1_{(\mathcal{X}_{u}>y)}\sigma \left( \mathcal{X}_{\eta \left(
        				s\right) }^{n,x}\right) dW_{s}\right\vert +\Lambda _{s\mapsto t} \\
        			\leq &\left\vert \int_{s}^{t}\sigma \left( \mathcal{X}_{\eta \left(
        				s\right) }^{n,x}\right) dW_{s}\right\vert +\left\vert \int_{s}^{t}1_{( 
        				\mathcal{X}_{u}>y)}\sigma \left( \mathcal{X}_{\eta \left( s\right)
        			}^{n,x}\right) dW_{s}\right\vert +2\Lambda _{s\mapsto t}.
        		\end{align*}
        		
        		Therefore the estimates follow along similar arguments as the ones for $%
        		\Lambda $. In particular, one has that $C_q=8q^2\|a\|_\infty. $
        	\end{proof}
        	
        	Using similar arguments as in the above proof, one also obtains:
        	
        	\begin{corollary}
        		A. Let $F:\mathbb{R}\mapsto (0,\infty )$ be a positive valued
        		function with Gaussian decay at infinity. In other words, there exists $c>0$
        		such that $ F(x)\leq e^{-\frac{\left\vert x\right\vert ^{2}}{2c}} $, then
        		\begin{align*}
        			\sup_n\mathbb{E}\left[\exp\left(p\Delta^{-1/2}\sum_{i=1}^n\int_L^\infty
        			F\left(\frac{y-L}{\sqrt{\Delta}}\right) \bar{\Lambda}^y_{t_{i-1}\mapsto t_i}dy\right)\right]<\infty.
        		\end{align*}
        		B.  Let $(Y,B) $ be solution of the reflected equation (\ref{eq:rep}). Then, for
        		any $q>0 $, we have that 
        		\begin{align*}
        			\mathbb{E}[e^{q\max_{s\in [0,T]}|Y_s|}]+ \mathbb{E}[e^{qB_T}]<\infty.
        		\end{align*}
        	\end{corollary}
        	\begin{proof}
        		We just mention that in order to prove A, one uses that 
        		\begin{align*}
        			\sum_{i=1}^n    \bar{\Lambda}^y_{t_{i-1}\mapsto t_i}=\bar{\Lambda}^y_{0\mapsto T}\leq 
        			\left\vert \int_{0}^{T}\sigma \left( \mathcal{X}_{\eta \left(
        				s\right) }^{n,x}\right) dW_{s}\right\vert +\left\vert \int_{0}^{T}1_{( 
        				\mathcal{X}_{s}>y)}\sigma \left( \mathcal{X}_{\eta \left( s\right)
        			}^{n,x}\right) dW_{s}\right\vert +2\Lambda _{0\mapsto T}.
        		\end{align*}
        		For the second term above one has for $ C=\|\sigma\|_\infty $
        		\begin{align*}
        			&       \mathbb{E}\left[\exp\left(p\Delta^{-1/2}\int_L^\infty e^{-\frac{(y-L)^2}{2c\Delta}}
        			\int_{0}^{T}1_{( 
        				\mathcal{X}_{u}>y)}\sigma \left( \mathcal{X}_{\eta \left( s\right)
        			}^{n,x}\right) dW_{s}
        			dy\right)\right]\\
        			&\leq \mathbb{E}\left[\exp\left(p\sqrt{2\pi c}
        			\max_{s\leq CT}|W_{s}|
        			\right)\right].
        		\end{align*}
        	\end{proof}
        	
        	\subsection{Some uniform moment estimates}
        	
        	\label{sec:gron} 
        	In this section, we provide a series of uniform moment
        	estimates that will be used to bound first and second derivatives.
        	
        	We will use an induction argument in order to bound the second derivatives. With this in mind, we need the following time uniform moment estimates. In the proofs many of the constants that we used are not optimal as we have preferred constants that are easier to deduce.
        	
        	We start with a general lemma. Here, we recall that for a discrete time
        	process $M $, one denotes its quadratic variation by $[M]_i=\sum_{j=1}^i(%
        	\Delta_jM)^2 $ with $\Delta_jM=M_j-M_{j-1} $. We also define  $ J_i:=\Delta^{-1/2}\sum_{j=1}^i\int_L^\infty e^{-\frac{(y-L)^2}{2c\Delta}}\bar{\Lambda}^y_{t_{j-1}\mapsto t_j}dy \geq 0$.
        	
        	\begin{lemma}
        		\label{lem:32} Given a sequence of $\mathcal{F}_i $-measurable random variables 
        		$\alpha_i \equiv \alpha_i^n$ with finite moments and let $\mathsf{E}
        		_i\equiv \mathsf{E} _i^n\in  \mathcal{F}_i$, $i=1,...,n$, $\mathsf{E} _0=1$
        		be a sequence of random  variables which satisfy the following linear
        		stochastic difference equation  :  
        		\begin{align*}
        			\mathsf{E}_{i}=&\mathsf{E}_{i-1} \left(1+\alpha_{i}\right).
        		\end{align*}
        		Assume that $\alpha_i\in L^p(\Omega) $ for some fixed $p\in 4\mathbb{N} $. For $q= p,p/2 $, we assume the following inequality is satisfied 
        		\begin{equation}
        			\sum_{j=1}^i\left(e^{-q{C}(\Delta+\Delta_jJ)}(1+\alpha_j)^q-1\right)    
        			\leq M_i(q)+\Upsilon_i(q).
        			\label{eq:79}
        		\end{equation}
        		Here, $M(q) $ is a $ \left(\widetilde{\mathbb{P}}^n ,\left(\mathcal{F}_{i}\right)_{i=1}^n\right)$-martingale and $\Upsilon_i(q)$ is the  $\mathcal{F}_{i}$-measurable random variable.
        		%                 defined by
        		%                \begin{align*}
        		%                \Upsilon_i(q):=\sum_{j=1}^i \widetilde{\mathbb{E}}_{j-1}[e^{-q{C}
        		%                        \Delta}(1+\alpha_j)^q-1].
        		%                \end{align*}
        		Moreover, let $\widetilde C_q\geq 0$ be a constant independent of $n $ such that for any $ C\geq \widetilde{C}_q $ there exists a constant $ \bar{C}_q\geq 1 $ and  $\delta_0(q)>0 $ such that for $\Delta\leq 
        		\delta_0(q) $, $i=1,...,n, $ 
        		\begin{align}  \label{eq:cond}
        			\widetilde{\mathbb{E}} _{i-1}[ | \Delta_i\Upsilon(q)|]+ \widetilde{\mathbb{E}} _{i-1}[\Delta_i[M(q)]]\leq \bar{C}_q\Delta.
        		\end{align}
        		Then there exists positive constants $\mathsf{C}_p\geq 1 $ and $ \delta_1(p)>0$ such that 
        		\begin{align}  \label{eq:l35}
        			\sup_{n\geq T \delta_1(p)^{-1}}\widetilde{\mathbb{E}}\left[\max_i|e^{-{\mathsf{C}}_p(t_i+J_i)}\mathsf{E}_i|^p %
        			\right]^{1/p}\leq \mathsf{C}_p.
        		\end{align}
        		
        	\end{lemma}
        	\begin{proof}
        		The proof is similar to the one for bounding moments of
        		approximations of linear stochastic differential equations by means of Gronwall's
        		inequality. First, note that for $p $ even, and any $ C\geq \widetilde{C}_p $, $|e^{-{C}t_i} \mathsf{E}_i|^p 
        		$ can be rewritten using subsequent differences as  
        		\begin{align}  \label{eq:69}
        			|e^{-{C}(t_i+J_i)}\mathsf{E}_i|^p\leq 1+ \sum_{j=1}^i|e^{-{C}(t_{j-1}+J_{j-1})} 
        			\mathsf{E}_{j-1}|^p \left(\Delta_jM(p)+\Delta_j\Upsilon(p)\right) .
        		\end{align}
        		
        		From here, taking expectations and using the hypotheses (in particular %
        		\eqref{eq:cond} for $q=p $), one obtains after application of the discrete
        		version of Gronwall's inequality and $x+1\leq {e^x}$ for $x\geq 0 $
        		that $\max_i\widetilde{\mathbb{E}}[ |e^{-{C}(t_i+J_i)}\mathsf{E}_i|^p] \leq e^{\bar{C}_pT}$.
        		
        		In order to obtain the estimate in \eqref{eq:l35}, we first use the 
        		representation \eqref{eq:69} for the term within parentheses in $ |e^{-2\widehat{C}(t_i+J_i)}\mathsf{E}_i|^p=((e^{-\widehat{C}(t_i+J_i)}\mathsf{E}_i)^{p/2})^2 $. Therefore, we will use the power index $p/2 $, instead of $p $
        		and any constant $ \widehat{C} >\widetilde{C}_p\vee \widetilde{C}_{p/2}$. This
        		gives a martingale term which  can be bounded using Doob's martingale
        		inequality, the hypothesis \eqref{eq:cond} for $q=p/2 $ and the previous 
        		estimate as follows:  
        		\begin{align*}
        			&\widetilde{\mathbb{E}}\left[\max_i \left(\sum_{j=1}^i|e^{-\widehat{C}(t_{j-1}+J_{j-1})} 
        			\mathsf{E}_{j-1}|^{p/2} \Delta_jM(p/2)\right)^2\right]\\\leq& 4 \widetilde{\mathbb{%
        					E}}\left[\sum_{j=1}^n|e^{-\widehat{C}(t_{j-1}+J_{j-1})}\mathsf{E} _{j-1}|^{p}
        			\Delta_j[M(p/2)]\right] \\
        			\leq& 4\bar{C}_{p/2}\sum_{j=1}^ne^{-p(\widehat{C}-\bar{C}_p)t_{j-1}}\Delta\leq \frac{4\Delta\bar{C}_{p/2}}{1-e^{-p\Delta(\widehat{C}-\bar{C}_p)}}.
        		\end{align*}
        		In a similar fashion, one deals with the term of bounded variation. In fact,
        		\begin{align*}
        			&	\widetilde{\mathbb{E}}\left[\max_i \left(\sum_{j=1}^i|e^{-\widehat{C}(t_{j-1}+J_{j-1})} 
        			\mathsf{E}_{j-1}|^{p/2} \Delta_j\Upsilon(p/2)\right)^2\right]\\\leq& \widetilde{ 
        				\mathbb{E}}\left[|\Upsilon_n(p/2)|\sum_{j=1}^n|e^{-\widehat{C}(t_{j-1}+J_{j-1})}\mathsf{E%
        			} _{j-1}|^{p} |\Delta_j\Upsilon(p/2)|\right] \\
        			\leq& \frac{\bar{C}^2_{p/2}T\Delta}{1-e^{{-p\Delta(\widehat{C}-\bar{C}_p)}}}.
        		\end{align*}
        		Using the inequality $ \frac{x}{1-e^{-x}}\leq 2 $ for $ x\in (0,1) $ and putting these estimates together with $\Delta\leq \delta_1(p):=(p(\widehat{C}-\bar{C}_p))^{-1}$, we have
        		\begin{align*}
        			\widetilde{\mathbb{E}}\left[\max_i|e^{-\widehat{C}(t_i+J_i)}\mathsf{E}_i|^p %
        			\right]\leq 3\left(1+\frac{8\bar{C}_{p/2}+2\bar{C}^2_{p/2}T}{p(\widehat{C}-\bar{C}_p)}\right).
        		\end{align*}
        		Therefore the result follows by taking  $ \widehat{C}=\left(\bar{C}_p+p^{-1}\left({8\bar{C}_{p/2}}+2\bar{C}^2_{p/2}T\right)\right)\vee 6$. 
        	\end{proof}
        	
        	A first application of the above result is the following:
        	
        	\begin{proof}[Proof of Lemma \protect\ref{moments}]\
        		We prove  (\ref{Lnbounds}) for $i=0$ only. For $\mathsf{E}=\mathcal{K} $, we apply  Lemma \ref{lem:32} with $%
        		1+\alpha_i=e^{\kappa_i}$. 
        		The argument is obtained by using a Taylor expansion of order two of 
        		$  e^{-q{C}(\Delta+\Delta_iJ)}(1+\alpha_i)^q$. Without giving detailed calculations, one uses 
        		\begin{align*}
        			\Delta_iM(q)=&q\frac{b_{i-1}}{a_{i-1}}\left(\Delta_iX-\mathbb{E}_{i-1}[\Delta_iX\bar{m}_i]\right), \\
        			\Delta_i\Upsilon(q)=&\left(-C(\Delta+\Delta_iJ)+\kappa_i\right)^2\int_0^1e^{-qz{C}(\Delta+\Delta_iJ)}(1+\alpha_i)^{qz}(1-z)dz.
        		\end{align*}
        		%                        \begin{align*}
        		%                \zeta(q,\Delta):=&e^{\frac{q(q-1)b_{i-1}^2\Delta}{2a_{i-1}}}\left(\bar{\Phi}\left(-X_{i-1}^{L,\sigma}-\frac{qb_{i-1}}{\sigma_{i-1}}\sqrt{\Delta}\right)+e^{-2\frac{q(X_{i-1}-L)}{a_{i-1}}}\bar{\Phi}\left(X_{i-1}^{L,\sigma}-\frac{qb_{i-1}}{\sigma_{i-1}}\sqrt{\Delta}\right)\right).
        		%                \end{align*}
        		In order to check \eqref{eq:79} and \eqref{eq:cond}, we use \eqref{eq:A} and \eqref{eq:B}. With these arguments one sets $ \widetilde{C}_q>qC\left\|\frac{b}{\sigma}\right\|_\infty$ ($ C $ is defined right after \eqref{eq:B}), $ \delta_0(q)=1 $ so that $\widetilde{%
        			\mathbb{E}} _{i-1}[| \Delta_i\Upsilon(q)|]\leq \bar{C}_q\Delta$. The other hypothesis $\widetilde{%
        			\mathbb{E}} _{i-1}[\Delta_i[M(q)]]\leq \bar{C}_q\Delta$ follows easily.
        		This gives  the result in this case. 
        		%       , as the same proof applies for $%
        		%       \bar{M}_{i:n}^{n}$ (we start the process $X^{n}$ at time $\frac{it}{n}~$from 
        		%       $x$ and apply the same proof). First observe that due to Doob's maximal inequality is enough to bound
        		%       \begin{align*}
        		%       \widetilde{ \mathbb{E}}\left[ \left( e^{-C_kT}{{\mathcal{K}}}_{n}^{n}\right) ^{k}\right] =& 1+\sum_{i=1}^n\widetilde{    \mathbb{E}}\left[\left( e^{-C_kt_{i-1}}{{\mathcal{K}}}_{i-1}^{n}\right) ^{k}\left(e^{-C_k\Delta+k\kappa_i}-1\right)\right]\\
        		%               =&1+\sum_{i=1}^n\widetilde{ \mathbb{E}}\left[\left( {{\mathcal{K}}}_{i-1}^{n}\right) ^{k}\left(e^{-C_k\Delta+\frac{k(k-1)}{2}%
        		%                       \left(\frac{ b_{i-1}}{\sigma _{i-1}}\right )^2\Delta }-1\right)\right]\notag\\
        		%               \leq&1.\notag
        		%       \end{align*}
        		%Here, we have let $ C_k:= \frac{k(k-1)}{2}\|\frac{b}{\sigma}\|^2_\infty\exp\left(\frac{k(k-1)}{2}\|\frac{b}{\sigma}\|^2_\infty\right)$. 
        		
        		Next, we consider the proof of \eqref{Lnbounds} in the case $\mathsf{E}=E^n $, which are done
        		using the same type  of arguments as for $\mathcal{K}^n $. In fact, we have
        		in this case  
        		\begin{align*}
        			\Delta_iM(q)=&e_i^q-\widetilde{\mathbb{E}}
        			_{i-1}[e_i^q\bar{m}_i] ,\\
        			\Delta_i\Upsilon(q)=&-q{C}(\Delta+\Delta_i J)\!\!\int_0^1\!\!e^{-qz{C}(\Delta+\Delta_i J)}(1-z)dz(e_i^q-1)+\sum_{j=2}^qC^q_j\mathbb{E}_{i-1}[(e_i-1)^j\bar{m}_i]. 
        		\end{align*}
        		Here, we need to use Lemma \ref{lem:333} and Theorem \ref{WandF} to obtain
        		the required conditions in \eqref{eq:cond}.  In fact, one proves that $%
        		\widetilde{\mathbb{E}}_{i-1}[e_i^q\bar{m}_i]-1=  \widetilde{\alpha}_{i-1}\Delta $
        		where $\widetilde{\alpha}_{i-1}\in\mathcal{F}_{i-1} $ is a  bounded r.v. in
        		order to obtain that $\Delta_i\Upsilon(q) $ satisfies the required 
        		condition by choosing $\bar{C}_q $ large enough and $ \delta_0(q)=1 $. A similar argument is used for $  \mathbb{E}_{i-1}[(e_i-1)^j\bar{m}_i]$. The conditions for $%
        		\Delta_iM(p) $ are verified in a similar fashion using Lemma \ref{lem:333}
        		and Theorem \ref{WandF}.  
        		
        		Finding an upper bound
        		moments of $\bar{M}^n $ is easier if one notes that  
        		\begin{equation*}
        			\widetilde{\mathbb{E}}\left[\left( \bar{M}_{n}^{n}\right) ^{p}\right] =\widetilde{ 
        				\mathbb{E}}\left[\prod_{i=1}^{n}1_{\left( {X}_{i}>L\right) }\left( 1+\left(
        			2^{p}-1\right) 1_{\left( U_{i}\leq p_{i}\right) }\right)\right ] \leq  {\left(
        				2^{p}-1\right)^n \widetilde{\mathbb{E}}\left[\bar{M}^{n}_n \right ]\leq 2^{pn}.}
        		\end{equation*}
        		
        		The bounds for the case $\mathsf{E}=X $ are
        		straightforward as moments of the Euler scheme can be bounded easily.  This
        		finishes the proof.
        	\end{proof}

        	Now, we give some lemmas that will be used together with the above moment
        	estimates in order to bound second
        	derivatives (see e.g. the proof of Proposition \ref{prop:30}, Lemmas %
        	\ref{boundedderivative} and \ref{additionalterm}). In order to have uniform bounds we will also need the following
        	preparatory lemma. 
        	
        	%\begin{lemma}\label{lem:26}Let any $ K\in\mathbb{N}$ and $ C>0 $ be fixed and assume $ \mathsf{M}>2(C+K) $ then for $ \Delta\leq K^{-1}+C^{-1} $
        	%       \begin{align*}
        	%               \Delta\sum_{i=1}^ne^{-\mathsf{M}t_{i-1}+Ct_i}\leq K^{-1}
        	%       \end{align*}
        	%\end{lemma}
        	%\begin{proof}
        	%\end{proof}
        	%        \textcolor{red}{ DAN: READ THIS LEMMA AGAIN ...still needs to be worked out.. }
        	\begin{lemma}
        		\label{lem:35} Assume that $\mathsf{E} _i\in \mathcal{F}_i$, $i=1,...,n$, $ 
        		\mathsf{E} _0=1$ is as stated in Lemmas \ref{moments} or \ref{lem:32}. That 
        		is, we let $\mathsf{C}\geq 1 $ be the constant so that $\widetilde{\mathbb{E}}\left[
        		(\max_ie^{-\mathsf{C}t_i}|\mathsf{E} _i|)^4\right]^{1/4}\leq \mathsf{C} $. Furthermore, suppose 
        		that the sequence of r.v.'s $g_i\in\mathcal{F}_i $ satisfies $|g_i|\leq  
        		\mathsf{M}e^{\mathsf{M}(T-t_i)} $ a.s. for some fixed constant $\mathsf{M}%
        		>17\mathsf{C}$. Then for $\Delta\leq    ( \mathsf{M}-\mathsf{C})^{-1}\ln(\frac{\mathsf{M}-\mathsf{C}}{16\mathsf{C}}) $  
        		\begin{align*}
        			(A)\quad \left|\Delta \sum_{i=1}^n\widetilde{\mathbb{E}}\left[g_i\mathsf{E}
        			_{i-1}\bar{M}^n_i\right] \right |\leq \frac{\mathsf{M}}8e^{\mathsf{M}T}.
        		\end{align*}
        		
        		Assume now that the above constant $\mathsf{M}>0 $, also satisfies that $%
        		\mathsf{M}>\bar C_2$ for a universal constant $\bar{C}_2>\mathsf{C} $
        		depending only on the bounds for the coefficients. Then for $\Delta\leq
        		\frac 5{\mathsf{M}-\mathsf{C}} $ and $\xi_{i}=X_{i-1}^L,X_{i}^L, Z_i $, we have  
        		\begin{align*}
        			(B)\quad \left| \sum_{i=1}^n\widetilde{\mathbb{E}}\left[g_i\mathsf{E} _{i-1}
        			\xi_{i}1_{(U_i\leq p_i)}\bar{M}^n_i\right]\right|\leq \frac {\mathsf{M}} 8e^{%
        				\mathsf{M}T}.
        		\end{align*}
        		
        		Next, assume $Y_i\in\mathcal{F}_{i} $ satisfies for $p\geq 0 $, $\widetilde{ 
        			\mathbb{E}}_{i-1}[Y_i] =O^E_{i-1}(\Delta^{p})$ and $\widetilde{\mathbb{E}}%
        		_{i-1}[|Z_iY_i|] =O^E_{i-1}(\Delta^{p}) $. Then for $f\in C^1_b([L,\infty),%
        		\mathbb{R} )$ 
        		\begin{align*}
        			(C)\quad \sup_{x\geq L}\left|\widetilde{\mathbb{E}}_{0,x}\left[\sum_{i=1}^nf(X_i)%
        			\mathsf{E} _{i-1}Y_i\bar{M}_i^n\right]\right|\leq C(\|f\|_\infty+\|f^{\prime
        			}\|_\infty)\Delta^{p}.
        		\end{align*}
        		If $Y_i\geq 0 $ then the above result is satisfied with the weaker
        		hypotheses: $\widetilde{\mathbb{E}}_{i-1}[Y_i] =O^E_{i-1}(\Delta^p)$ and $f\in
        		C_b $.
        	\end{lemma}
        	
        	\begin{proof}
        		First, we apply the hypotheses to obtain that  
        		\begin{align*}
        			\Delta\sum_{i=1}^n\widetilde{\mathbb{E}}\left[|g_i||\mathsf{E} _{i-1}|\bar{M}^n_i%
        			\right]\leq &\Delta \mathsf{M}e^{\mathsf{M}T}\sum_{i=1}^ne^{-\mathsf{M}%
        				t_i+\mathsf{C}t_{i-1}} \widetilde{\mathbb{ E}}\left[e^{-\mathsf{C}t_{i-1}}|\mathsf{E} _{i-1}|\bar{%
        				M}^n_i\right] \\
        			\leq &2C\mathsf{M}e^{\mathsf{M}T}\Delta\sum_{i=1}^ne^{-\mathsf{M}t_i+\mathsf{C}t_{i-1}}
        			\leq \frac {\mathsf{M}}8e^{\mathsf{M}T}.
        		\end{align*}
        		In the last inequality, we have used the explicit formula for geometric sums
        		and the inequality $\frac{1-e^{-(\mathsf{M}-\mathsf{C})x}}{x} \geq 16\mathsf{C}$ for all $x\in (0, 
        		( \mathsf{M}-\mathsf{C})^{-1}\ln(\frac{\mathsf{M}-\mathsf{C}}{16\mathsf{C}}))$.
        		
        		The proof of (B) additionally uses the proof in Lemma \ref{lem:essb}. We 
        		do the proof in the case that $\xi_i=X_{i-1}^L $. The other cases follow 
        		similarly, taking into account Lemma \ref{lem:333}. The proof in this case starts by taking the 
        		conditional expectation of $1_{(U_i\leq p_i)} $, so that we have for $F(x) =x\bar{\Phi}(x)\leq \sqrt{\frac 2\pi}e^{-\frac{x^2}2}$:  
        		\begin{align*}
        			&\mathcal{A}:=\sum_{i=1}^n\widetilde{\mathbb{E}}\left[|g_i||\mathsf{E}
        			_{i-1}|X^{L}_{i-1}1_{(U_i\leq p_i)} \bar{M}^n_i\right]\\\leq &\sqrt{ \Delta}
        			\|\sigma\|_{\infty}\mathsf{M}e^{\mathsf{M}T}{\widetilde{\mathbb{E}}}\left[%
        			\sum_{i=1}^ne^{-Mt_i}| \widehat{\mathsf{E}} _{i-1}|F\left(\frac{\mathcal{X}
        				_{i-1}-L}{\sigma(\mathcal{X}_{i-1})\sqrt{\Delta}}\right )\right].
        		\end{align*}
        		Here, $\widehat{\mathsf{E}} $ denotes the process corresponding to the process ${\ \mathsf{E}} $
        		where one has replaced all occurrences of $X $  with $\mathcal{X} $. 
        		Note that by hypothesis
        		and Cauchy-Schwartz inequality, we have $\widetilde{\mathbb{E}}\left[ (\max_ie^{-\mathsf{C}t_i}|%
        		\widehat{\mathsf{E} }_i|)^2\right]^{1/2}\leq \mathsf{C}.$
        		
        		Recalling the value of the constants in the proof of Lemma \ref{lem:essb}, we have
        		for $\bar{C}_1:=4\sqrt{(1+c)}\|a\|_\infty\|a^{-1}\|_\infty $:  
        		\begin{align*}
        			\mathcal{A}\leq&\bar{C}_1\mathsf{M}e^{\mathsf{M}%
        				(T-\Delta)}\int_{L}^{\infty}
        			g_{c\|a\|_\infty\Delta}(y-L) \widetilde{\mathbb{E}}\left[\max_je^{-\mathsf{C}t_{j-1}}|\widehat{\mathsf{E}} _{j-1}| \sum_{i=1}^ne^{-(\mathsf{M}-\mathsf{C})t_{i-1}}\bar
        			\Lambda^y_{t_{i-1}\to t_i} \right]dy.
        			%                \\
        			%               \leq&\frac {\mathsf{M}}{4} e^{\mathsf{M}T} \int_{L}^{\infty}
        			%               g_{c\|a\|_\infty\Delta}(y-L) {\mathbb{E}}\left[\sum_{i=1}^ne^{-\mathsf{M}t_i}
        			%               {\ \mathbb{E}}_{i-1}\left[e^{\bar C_1\bar{\Lambda}^y_{t_{i-1}\to t_i}}-1%
        			%               \right] |\widehat{\mathsf{E}} _{i-1}|\right]dy.
        		\end{align*}
        		Before applying the same estimate procedure as in the proof of (A), we perform a summation by parts formula in :
        		\begin{align}
        			\sum_{i=1}^ne^{-(\mathsf{M}-\mathsf{C})t_{i-1}}\bar
        			\Lambda^y_{t_{i-1}\to t_i}=e^{-(\mathsf{M}-\mathsf{C})T}
        			\bar{\Lambda}^y_{0\to T}-
        			\sum_{i=1}^ne^{-(\mathsf{M}-\mathsf{C})t_{i-1}}\bar
        			\Lambda^y_{0\to t_i}(e^{-(\mathsf{M}-\mathsf{C})\Delta}-1).
        			\label{eq:6ab}
        		\end{align}
        		Now, we obtain upper bounds for the expectation using Cauchy-Schwartz inequality, the hypotheses and \eqref{moduluso}. For the square of the first term in \eqref{eq:6ab}, one obtains 
        		\begin{align}
        			&\widetilde{\mathbb{E}}\left[(e^{-(\mathsf{M}-\mathsf{C})T}
        			\bar{\Lambda}^y_{0\to T})^2\right]\leq e^{-2(\mathsf{M}-\mathsf{C})T}\widehat{c}T\leq
        			\frac{\widehat{c}}{2(\mathsf{M}-\mathsf{C})}. 
        			\label{eq:78a}
        		\end{align}
        		For the square of the second term in \eqref{eq:6ab}, one has two types of terms: diagonal and non-diagonal terms for the square of the sum. In each, one uses the inequality $ 1-e^{-x}\leq x $, $ x\geq 0 $, approximations with integrals and the Gamma function to bound those integrals. This gives for the diagonal terms:
        		\begin{align}
        			\widetilde{ \mathbb{E}}\left[\sum_{i=1}^ne^{-2(\mathsf{M}-\mathsf{C})t_{i-1}}
        			(\bar{\Lambda}^y_{0\to t_i})^2(e^{-(\mathsf{M}-\mathsf{C})\Delta}-1)^2\right]\leq& \widehat{c}(\mathsf{M}-\mathsf{C})^2\Delta e^{4(\mathsf{M}-\mathsf{C})\Delta}\int_\Delta^{T+\Delta}\!\!\!\!\!\!\!\!\!e^{-2(\mathsf{M}-\mathsf{C})s}sds\notag\\
        			\leq& \frac{\widehat{c}}4\Delta e^{4(\mathsf{M}-\mathsf{C})\Delta}.\label{eq:79a}
        		\end{align}
        		A similar calculation for the non-diagonal terms gives:
        		\begin{align}
        			&\widetilde{\mathbb{E}}\left[\sum_{\substack{i,j=1\\i<j}}^ne^{-2(\mathsf{M}-\mathsf{C})(t_{i-1}+t_{j-1})}
        			\bar{\Lambda}^y_{0\to t_i}\bar{\Lambda}^y_{0\to t_j}(e^{-(\mathsf{M}-\mathsf{C})\Delta}-1)^2\right]\notag\\&\leq
        			\widehat{c}(\mathsf{M}-\mathsf{C})^2 e^{4(\mathsf{M}-\mathsf{C})\Delta}\int_\Delta^{T+\Delta}ds\int_s^{T+\Delta}du
        			e^{-(\mathsf{M}-\mathsf{C})(s+u)}\sqrt{s}(\sqrt{s}+\widehat{c}\sqrt{u-s+\Delta})\notag\\
        			&\leq \widehat{c}\frac{e^{5(\mathsf{M}-\mathsf{C})\Delta}}{\mathsf{M}-\mathsf{C}}\left(\frac 14+\widehat{c}\frac{\pi}{8\sqrt{2}}\right).\label{eq:80a}
        		\end{align}

        		Therefore, the result follows using the estimates in Lemma \ref{moments} and
        		arguments similar to the previous case (A). Here, we have used $ \bar{C}_2:=\mathsf{C}+38\bar{C}_1c\widehat{c}(1+\widehat{c}) $ and $ \Delta\leq (5(\mathsf{M}-\mathsf{C}))^{-1} $. This constant is not optimal but it serves to obtain that  \eqref{eq:79a} and \eqref{eq:78a}$ + $\eqref{eq:80a} multiplied by $ 2\bar{C}_1\mathsf{M}e^{\mathsf{M}T} \mathsf{C}$ is smaller than $ \frac{\mathsf{M}}{16}e^{\mathsf{M}T}  $.
        		
        		For the proof of (C), one uses the decomposition $%
        		f(X_i)=(f(X_i)-f(X_{i-1}))+f(X_{i-1}) $.  Both terms are treated similarly.
        		
        		The rest of the proof is a simple application of the tower property of 
        		conditional expectations, Cauchy-Schwartz inequality, the moment hypothesis for $ \mathsf{E} $ and the Definition \ref{def:2u} to  
        		\begin{align*}
        			\left|\widetilde{\mathbb{E}}_{0,x}\left[\sum_{i=1}^nf(X_{i-1})\mathsf{E}%
        			_{i-1}Y_i \bar{M}_i^n\right]\right|\leq\|f\|_\infty \left|\widetilde{\mathbb{E}}%
        			_{0,x} \left[\max_i|\mathsf{E}_{i}|\sum_{i=1}^n|\widetilde{\mathbb{E}}%
        			_{i-1}[Y_i]|\bar{M }_n^n\right]\right|.
        		\end{align*}
        		In the case of $f(X_i)-f(X_{i-1}) $, one has to apply the mean value theorem
        		which gives that $|f(X_i)-f(X_{i-1}) |\leq \|f^{\prime }\|_\infty 
        		\sigma_{i-1}|Z_i|$. The rest of the proof follows similarly.
        		
        		In the case that $Y_i\geq 0 $ the proof is simpler and it does not require 
        		the use of the mean value theorem.
        	\end{proof}
        	
        	\begin{remark}
        		At first it may look odd to the reader that the obtained estimates in (A) and
        		(B) are smaller than the original bound for $|g_0| $. This is due to the
        		fact that the estimate for $ g_i $ is decreasing in $i $ faster than the estimate for $%
        		\mathsf{E} $ is increasing in $i $.
        		
        		In the application of (A) and (B) in the final estimates for second derivatives, the number $8$ corresponds to the number of terms of the same type where the above results will be applied. (A) will be applied to drift terms and (B) for all other terms. 
        		For more details about this see the proof of Proposition \ref{prop:30}. Note that in (C) we use $ f \in C^1_b$ while  in (B) we do not assume this property for $ g_i $.
        	\end{remark}
        	
        	%\begin{lemma}
        	%       Consider the products for $ c>0 $,
        	%       \begin{align*}
        	%               \mathsf{L}_i:=&\prod_{j=1}^i(1+c\Delta g_{j-1}(X_{j-1}^{L,\sigma}))\\
        	%               \bar{\mathsf{L}}_i:=&\prod_{j=1}^i\left (1+c\sqrt{\Delta} \int_{L}^{\infty }\exp \left( -\frac{\left( y-L\right) ^{2}}{2\|a\|_\infty \Delta }\right) \Lambda_{t_{j-1}\mapsto
        	%                       t_{j}}^{y}dy
        	%               \right )
        	%       \end{align*}
        	%Then $\mathbb{E}[\mathsf{L}_i]\leq\mathbb{E}[ \bar{\mathsf{L}}_i] $
        	%\end{lemma}
        	%\begin{proof}
        	%       We prove that $  \mathbb{E}[\mathsf{L}_i- \bar{\mathsf{L}}_i]\leq 0 $. 
        	%       In fact, to simplify the notation let 
        	%\end{proof}
        	%\subsection{A second Gronwall lemma}
        	%Maybe not needed or even impossible plan...
        	
        	\subsection{Uniform boundedness of the second derivative}

        	\label{app:6.2a}  
        	{In this section, we prove that the second derivative $\partial _{x}^2\mathbb{E}\left[ f\left( X_{T\wedge \tau
        			_{n}}^{n,x}\right)  \right] $ is uniformly bounded. This implies the uniform continuity of the
        		first derivative $\partial _{x}\mathbb{E}\left[ f\left( X_{T\wedge \tau
        			_{n}}^{n,x}\right)  \right] $ which, in turn  justifies the uniform convergence
        		property required in the proof of Theorem \ref{th:main}.
        		Many parts of the arguments below are similar to the study of the first derivatives,  except that expressions and equations involved are longer. Nevertheless, there are new components which do not appear in the study of the first derivative. In the next few paragraphs, we explain briefly the main line of the arguments and the differences between this and the analysis of the first derivative. }
        	
        	Rather than a direct proof of the boundedness of the second derivative (recall the proof of Lemma \ref{boundedderivative}) we will prove that the second derivative is bounded by induction on the index $ t_i =Ti/n$ backwardly from $ i=n $ to $ i=1 $.
        	The main reason for this change of argument is because in the proof 
        	of Lemma \ref{boundedderivative}
        	one uses strongly that the function $ f_i $ appearing in the second term on the right of the equality \eqref{fprime} vanishes at the boundary. In the case of the second derivative this property does not longer hold for the push forward of the first derivative.  
        	
        	Still, one can write the corresponding push-forward formula which appears in Lemma \ref{lem:33b} below. This formula includes the extra term $ \widehat{B}_k(e_k) $ which does not appear in the study of the first derivative. The iterated formula is obtained in Lemma \ref{lem:34} which now has three terms and a small remainder of order $\Delta^{1/2}$  (note that the corresponding push forward formula \eqref{fprime} had only two terms).
        	
        	The property that allows the analysis of the extra term is the fact that the product term $ \widehat{P}^n_{k-1}=\prod_{j=1}^{k-1} \widehat{A}^1_j(e_j)$ has as leading constant term the process $ M^n_{k-1} $ which indicates that essentially the process does not touch the boundary while $ \widehat{B}_k(e_k) $ has as main leading term $ -\frac{b_{k-1}}{a_{k-1}}1_{(U_k\leq p_k)} $ in \eqref{eq:tab}. Therefore this essentially means that the process does not touch the boundary until the time interval $ [t_{k-1},t_k] $ which is the crucial part of the argument in Lemma \ref{lem:36}. In order to be able to use this argument we make use of a ``path decomposition'' argument which appears at the beginning of the proof for Lemma \ref{lem:36}.

        	Before we start, we remark that an argument that will be repeatedly used is the following: Terms appearing in $%
        	\partial^2 _{x}\mathbb{E}_{i-1,x}\left[ f\left( X_{T\wedge \tau _{n}}^{n}\right)  %
        	\right] $ which converge to zero uniformly in $ (i,x,n) $ do not need to be considered
        	further. For the remaining terms we will prove that they are uniformly
        	bounded.\\[1mm]
        	
        	\noindent The objective of this section is to prove the following: 
        	
        	\begin{proposition}
        		\label{prop:30}  There exists a universal constant $\mathsf{M} $ which
        		depends only the constants of the problem. such that for all $i=0,...,n-1 $  
        		\begin{align}  \label{eq:87a}
        			\sup_{x\geq L}\left |\partial_{x}\mathbb{E}_{i,x}\left[ {f}^{\prime }\left( X_{n}^{n}\right)
        			E_{i:n}^{n}\bar{M}_{i:n}^{n}\right]\right| +\sup_{x\geq L}\left |\sum_{j=i}^{n}\partial_{x}\mathbb{E}
        			_{i,x}\left[ f_{j}\left( X_{j}^{n}\right) E_{i:j-1}^{n}{h}_{j}\bar{M}
        			_{i:n}^{n}\right]\right|\leq \mathsf{M}e^{\mathsf{M}(T-t_{i-1})}.
        		\end{align}
        		In particular, this implies that $ \|\partial_i^2f_i\|_\infty  \leq \mathsf{M}e^{\mathsf{M}(T-t_{i})}$.
        	\end{proposition}
        	
        	%The reason why we are considering a general starting point at $(t_i,X_i) $
        	%is due to the fact that we will be using an inductive argument to prove that
        	%a certain explicit bound is satisfied for the first and second derivative at
        	%each $(t_j,X_j) $ for $j>i $ and then prove that the same bound is satisfied
        	%at $(t_i,X_i) $. This will be possible through Lemma \ref{lem:35}.
        	
        	The proof will be carried out by induction. Therefore, we assume from now on that there exists a universal constant $%
        	\mathsf{M} $ such that  
        	\begin{align*}
        		\sup_{x\geq L}\left |\partial_{j} \mathbb{E}_{j,x}\left[ {f}^{\prime
        		}\left( X_{n}^{n,x}\right) E_{j:n}^{n}\bar{M}_{j:n}^{n}\right]\right|\leq &
        		\frac{\mathsf{M}}2e^{\mathsf{M}(T-t_{j})},\quad j>i,\\
        		\sup_{x\geq L}\left|\partial_{j}      \mathbb{E}
        		_{j,x}\left[ f_{k}\left( X_{k}^{n}\right) E_{j:k-1}^{n}{h}_{k}\bar{M}
        		_{j:n}^{n}\right]\right|\leq &
        		\frac{\mathsf{M}}2e^{\mathsf{M}(T-t_{j})}\Delta,\quad k>j>i.
        	\end{align*}
        	Our goal is to prove that both inequalities are satisfied for $ j=i $.
        	
        	Let us start with some lemmas about the iteration of the above derivatives. In
        	order to deal with both terms in \eqref{eq:87a}, we need to define their
        	corresponding conditional expectations. That is, define 
        	\begin{align*}
        		\widehat{g}_{k}(x):=&{\mathbb{E}}_{k,x}\left[f^{\prime }_n {E} _{k:n}^{n}\bar{ M 
        		}^n_{k:n} \right], \quad k\leq n,\\
        		\bar g_{k,j}(x):=&{\mathbb{E}}_{k,x}\left[f_j {E} _{k:j-1}^{n}{h}_j\bar{ M }^n
        		_{k:j} \right], \quad k<j\leq n.
        	\end{align*}
        	
        	\begin{lemma}
        		\label{lem:33b} Let $g_{k,j}\in\{\widehat{g}_{k},\bar{g}_{k,j}\}$. We have for $%
        		i+1\leq k<j $  
        		\begin{align*}
        			\partial_{k-1}{\mathbb{E}}_{k-1}\left[ g_{k,j}{e}_k\bar{m}_k \right]=&\widetilde{%
        				\mathbb{E}}_{k-1}\left[\left(\partial_k g_{k,j}\widehat{A} _k^1({e}_k)+ \varpi(%
        			{g}_{k,j},X_k)(X_k-L)\widehat{A}^0_k({e} _k)\right)\bar{m}_k\right] \\
        			&+ g_{k,j}(L) \widetilde{\mathbb{E}}_{k-1}[\widehat{B}(e_k)\bar{m}%
        			_k]+\|g_{k,j}\|_\infty {\ O_{k-1}^E(\Delta^{1/2}) }.
        		\end{align*}
        		Here, $\widehat{A}_k^\ell({e}_k)=\bar{A}_k^\ell({e}_k) \left(1+\frac{b_{k-1}}{%
        			\sigma_{k-1} }Z_k+\frac{b^2_{k-1}}{ 2a_{k-1} }(Z_k^2-\Delta)\right)$, $%
        		\ell=0,1 $ (recall the definitions in Lemma \ref{lem:22}). Furthermore, we
        		can write the following representations formulas:  
        		\begin{align}
        			\widehat{A}^1_k({e}_k)=&1_{(U_k>p_k)}+\mu^1_{k}+\varepsilon_k^1,  \notag \\
        			\widehat{A}^0_k({e}_k)=& -\frac{b_{k-1}}{a_{k-1}}1_{(U_k\leq
        				p_k)}+\mu^2_{k}+\varepsilon^2_k, \label{eq:tab} \\
        			\widehat{B}_k(e_k)=&-\frac{b_{k-1}}{a_{k-1}}1_{(U_k\leq p_k)}+\mu%
        			^3_k+2\Delta\bar{b}_{k-1}g_{k-1}(X_{k-1}^L) +\varepsilon^3_k. \notag
        		\end{align}

        		Here  $\mu
        		^\ell_{k},$ $\ell=1,2,3$ are $\mathcal{F}_{k}$-measurable random variable. In particular $\mu^\ell_{k} $, $\ell=1,2,3 $ are at most quadratic
        		polynomials without constant terms and coefficients  which depend on the
        		coefficients $ b $ and $ \sigma $ and their derivatives at $t_{k-1} $. These quadratic
        		polynomials have as variables $Z_k$, $(X_{k-1}-L)1_{(U_k\leq p_k)} $, $Z_k
        		1_{(U_k> p_k)}$, $\Delta 1_{(U_k> p_k)}$, $\Delta $ and they satisfy $\widetilde{%
        			\mathbb{E}}_{k-1}[\mu^\ell_k\bar{m}_k]=O^E_{k-1}(1) $ for $%
        		\ell=1,2,3$. Furthermore, $ \widetilde{%
        			\mathbb{E}}_{k-1}[\varepsilon^\ell_k\bar{m}_k]=O^E_{k-1}(%
        		\Delta^{1/2}) $ for $%
        		\ell=1,2,3$.
        		%  \begin{align*}
        		%               \widehat{A}_k^1({e}_k):=&
        		%               \bar{A}_k^1({e}_k) +\frac{b_{k-1}}{\sigma_{k-1} }Z_k1_{(U_k>p_k)}+\frac{b^2_{k-1}}{%
        		%                       2a_{k-1} }(Z_k^2-\Delta)1_{(U_k>p_k)}
        		%       \end{align*}
        		%\begin{align*}
        		%               \widehat{B}_{k}(e_k):=&-\frac{b_{k-1}}{\sigma_{k-1}}1_{(U_k\leq p_k)}+\mu_{k-1}
        		%\end{align*}

        	\end{lemma}
        	
        	\begin{proof}
        		The proof is an application of Lemma \ref{lem:22} together with Girsanov's
        		theorem. To the Girsanov change of measure, one applies Taylor expansion of
        		order 2 with integral residue
        		\begin{align*}
        			e^{\kappa_k}=1+\kappa_k+\frac{\kappa_k^2}{2}+\frac{\kappa_k^3}{2!}%
        			\int_0^1(1-u)^2e^{u\kappa_k}du.
        		\end{align*}
        		
        		Therefore the various $O_{i-1}^E(\Delta^{1/2})$ terms in the main result are
        		obtained using  Lemma \ref{lem:333} (see also Remark \ref{rem:11}) in all
        		the terms which multiply the above integral residue and which do not appear
        		in the definition of $\widehat{A}_k^\ell({e}_k) $.
        		
        		The formula for $\widehat{B}_k(e_k) $ is obtained
        		using the above expansion and applying the derivative of the conditional
        		expectation as indicated in the statement of Lemma \ref{lem:22}. For this
        		computation, one uses again the same argument used in the proof of Lemma \ref
        		{lem:22} by using Theorem \ref{WandF}, the decomposition $Z_k=\frac{X_k-L-(X_{k-1}-L)}{\sigma_{k-1}}$ and product rules for derivatives.

        		In
        		fact, note that using Lemma \ref{lem:22} one has for $\lambda_k=1,
        		1_{(U_k>p_k)} $ 
        		\begin{align*}
        			\partial _{k-1}\widetilde{\mathbb{E}}_{{k-1}}\left[ (X_{k}-L)\lambda _{k}\bar{m}%
        			_{k} \right] =&\widetilde{\mathbb{E}}_{k-1}\left[ \bar{A}_{k}^{1}({\lambda _{k}%
        			})\bar{m}_{k}\right], \\
        			\widetilde{\mathbb{E}}_{k-1}\left[ \bar{A}_{k}^{1}(1_{(U_k>p_k)})\bar{m}_{k}%
        			\right]=& \widetilde{\mathbb{E}}_{k-1}\left[1_{(U_k\leq p_k)}\bar{m}_{k}\right],
        			\\
        			\bar{A}_{k}^{1}(1)=&1_{(U_k>p_k)} +\sigma^{\prime }_{k-1
        			}\left(Z_k+1_{(U_k\leq p_k)}\frac{X_{k-1}-L}{\sigma_{k-1}}\right).
        		\end{align*}
        		
        		We remark that the term $2\Delta\bar{b}_{k-1}g_{k-1}(X_{k-1}^L) $ which
        		appears in $\widehat{B}_k(e_k) $ is due to the derivative of $\Delta\bar{b}%
        		_{k-1} \mathbb{E}_{k-1}\left[1_{(U_k>p_k)}\bar{m}_k\right ] $. The proof of
        		the qualitative statements in \eqref{eq:tab} are just a matter of detailed
        		algebra and  determining the order of each term using Definition \ref{def:2u}
        		and Remark  \ref{rem:11}. 
        		All the terms of order $O_{k-1}^E(\Delta^{1/2})
        		$ except for the third term in $ \widehat{B}_k(e_k) $ have been left as residual terms. 
        	\end{proof}
        	\begin{remark}
        		
        		The formulas in Lemma \ref{lem:33b} are also satisfied when one replaces  $%
        		(g_{k,j},e_k) $ by $(f_j, {h}_j)$. That is, $\widehat{A}_j^\ell({h}_j)=\bar{A}_j^\ell({h}_j) \left(1+\frac{b_{j-1}}{\sigma_{j-1} }%
        		Z_j+\frac{b^2_{j-1}}{ 2a_{j-1} }(Z_j^2-\Delta)\right)$, $\ell=0,1 $ with the
        		following representation formulas where $\mu ^\ell_{k} $, $ \varepsilon^\ell_k $ $%
        		\ell=4,5 $ satisfy the same properties as other $\mu ^\ell_{k} $, $%
        		\ell=1,2,3 $ in Lemma \ref{lem:33b}    
        		\begin{align*}
        			\widehat{A}^1_j({h}_j)=&-\frac ba(L)1_{(U_j\leq p_j)}+\mu^4_{j}+\varepsilon^4_j,  \notag \\
        			\widehat{A}^0_j({h}_j)=&2 \left(\partial_{j-1}\frac{b_{j-1}}{a_{j-1}}-\frac{%
        				b^2_{j-1}}{a^2_{j-1}}\right)1_{(U_j\leq
        				p_j)}+\mu^5_{j}+\varepsilon^5_j. 
        			%\label{eq:tab1}
        		\end{align*}
        		
        		In this case, note that $f_j(L)=0 $ and therefore there is no need to
        		compute $\partial_{j-1}\widetilde{\mathbb{E}}_{j-1}[h_j\bar{m}_j]=$ $\widetilde{%
        			\mathbb{E}}_{j-1}[\widehat{B}_j(h_j)\bar{m}_j] $. 
        		%Here, $\widetilde \lambda^4_{k}\in \mathcal{F}_{k}  $ satisfies the same above conditions as  $ \mu%
        		%^\ell_{k} $, $ \ell=1,2,3 $ except that instead of quadratic is a linear polynomial. Furthermore 
        		%\begin{align*}
        		%\lambda^0_{k-1}= \left(b^2\frac{\sigma'-\sigma}{\sigma^5}+\left(\frac{b}{a}\right)'\right)(X_{k-1}).
        		%\end{align*}
        	\end{remark}
        	We start with the iteration of push-forward derivatives formulas in Lemma \ref{lem:33b}.
        	
        	\begin{lemma}
        		\label{lem:34}
        		Define $\widehat{P}^n_{i:k}:=\prod_{j=i+1}^k\widehat{A}^1_j({e} _j) $. Then we
        		have the following expressions for the second derivatives:  
        		\begin{align}
        			\partial_{i}\mathbb{E}_{i}\left[ {f}^{\prime }\left( X_{n}^{n,x}\right)
        			E_{i:n}^{n}\bar{M}_{i:n}^{n}\right]=&\widetilde{\mathbb{E}}_{i,x}\left[
        			f^{\prime \prime }(X_n^{n,x}) \widehat{P} _{i:n}^{n}\bar{M}^n_{i:n}\right]\notag
        			\\+&\sum_{k=i}^{n}\widetilde{\mathbb{E}} _{i,x}\left[ \varpi(\widehat{g}_{k},X_k) 
        			\widehat{P}_{i:k-1}^{n}(X_k-L)\widehat{A}^0_k( {e}_k) \bar{M}^n_{i:k}\right]\notag\\ +&\sum_{k=i}^{n}\widehat{g}_{k}(L) \widetilde{ \mathbb{E}}_{i,x}\left[ \widehat{P}
        			_{i:k-1}^{n}\widehat{B}_k(e_k) \bar{M}^n_{i:k}\right]+O(\Delta^{1/2}).
        			\label{eq:85}
        		\end{align}
        		Similarly,  
        		\begin{align}
        			&\sum_{j=i}^{n}\partial_{x}{\mathbb{E}}_{i,x}\left[ f_j {E}_{i:j-1}^{n}{h}_j 
        			\bar{M}^n_{i:j}\right]=\sum_{j=i}^{n} \widetilde{\mathbb{E}}_{i,x}\left[
        			\partial_j f_j \widehat{P} _{i:j-1}^{n}\widehat{A}^1_j({h}_j) \bar{M}^n_{i:j}%
        			\right]  \label{eq:72aa} \\
        			&+\sum_{j=i}^{n} \widetilde{\mathbb{E}}_{i,x}\left[ (f_j -f_j(L))\widehat{P}
        			_{i:j-1}^{n}\widehat{A}^0_j({h}_j) \bar{M}^n_{i:j}\right]\\
        			& +\sum_{j=i}^{n}%
        			\sum_{k=i+1}^{j-1} \widetilde{\mathbb{E}} _{i,x}\left[ \varpi(\bar{g}_{k,j},X_k) 
        			\widehat{P}_{i:k-1}^{n}(X_k-L)\widehat{A}^0_k( {e}_k) \bar{M}^n_{i:k}\right] 
        			\notag \\
        			&+\sum_{j=i}^{n}\sum_{k=i+1}^{j-1}\bar{g}_{k,j}(L) \widetilde{ \mathbb{E}}_{i,x}%
        			\left[ \widehat{P} _{i:k-1}^{n}\widehat{B}_k(e_k) \bar{M}^n_{i:k}\right]%
        			+O(\Delta^{1/2}).  \notag
        		\end{align}
        	\end{lemma}

        	We remark that the statement related to the universal nature of $\mathsf{M} 
        	$ in Proposition \ref{prop:30} is unaffected by terms that converge to zero such as the ones appearing in the above formula because they converge to zero uniformly in $ (i,x,n) $.

        	We give now the moment estimate for $\widehat{P}^n $ which appears in the above
        	iterated formula.
        	
        	\begin{lemma}
        		\label{lem:41a} For any $p\in 2\mathbb{N} $, there exists a constant $\mathsf{C}_p>0 $ independent of $n$ such that  
        		\begin{align*}
        			\sup_{x\geq L}\widetilde{\mathbb{E}}\left[\max_i|e^{-{\mathsf{C}}_pt_i}\widehat{P}^n_i|^p\bar{
        				M}_n\right]^{1/p}\leq \mathsf{C}_p.
        		\end{align*}
        	\end{lemma}
        	
        	\begin{proof}
        		The proof is obtained through a small modification of Lemma \ref{lem:32} to the expectation operator given
        		by $\widetilde{\mathbb{E}}\left[\cdot\bar{                M}_n\right] $.
        		In fact, one starts by considering the subsequent
        		differences as in \eqref{eq:69}%
        		\begin{align*}
        			0\leq |e^{-{C}t_i}\widehat{P}^n_i|^p=&1+ \sum_{j=1}^i|e^{-{C}t_{j-1}} \widehat{%
        				P}^n_{j-1}|^p \left(e^{-p{C}\Delta}\widehat{A}^1_i({e}_i)^p-1\right) \\
        			\leq &1+\sum_{j=1}^i|e^{-{C}t_{j-1}} \widehat{P}^n_{j-1}|^p\left(e^{-p{C}\Delta}\left(\widehat{A}^1_i({e}_i)^p+1_{(U_i\leq p_i)}\right)-1\right).
        		\end{align*}
        		
        		Now we verify the  required conditions in Lemma \ref{lem:32} for the above upper
        		bound using the estimates in Lemma \ref{lem:333}.  We have that  
        		\begin{align*}
        			\Delta_i\Upsilon(p)=\widetilde{\mathbb{E}}_{i-1}\left[\left(e^{-p{C}\Delta}\left(\widehat{A}^1_i({e}_i)^p+1_{(U_i\leq p_i)}\right)-1\right)\bar{%
        				m}_i\right].
        		\end{align*}
        		Explicit formulas to rewrite the above increment are available in Lemma \ref%
        		{lem:33b}. When verifying the condition that $ |\Delta_i\Upsilon(p)|\leq 
        		\bar{C}_p\Delta$ is important to note that as $p\geq 2 $, all terms in the conditional expectation 
        		expansion of $\widehat{A}^1_i({e}_i)^p+1_{(U_i\leq p_i)}-1 $
        		will be smaller than a multiple of $ \Delta  $. Therefore one obtains the
        		required condition in \eqref{eq:cond} by considering a constant ${C}$ large enough. In fact, algebraic calculations lead to
        		\begin{align*}
        			\widehat{A}^1_i({e}_i)^p=&1_{(U_i> p_i)}+\mathsf{P}_{i}^p+\sum_{k=1}^{p-1}\binom{p}{k}1_{(U_i> p_i)}\mathsf{P}_i^{p-k}.
        		\end{align*} 
        		Here $  \mathsf{P}_i=\mathsf{P}_i(\Delta, Z_i1_{(U_i> p_i)},Z_i,(X_{i-1}-L)1_{(U_i\leq p_i)})$ is a polynomial of degree $ 3 $ with no constant term and random coefficients which are uniformly bounded in $ (i,x,n) $. Therefore one has 
        		
        		\begin{align*}
        			e^{-p{C}\Delta}\left(\widehat{A}^1_i({e}_i)^p+1_{(U_i\leq p_i)}\right)-1=e^{-p{C}\Delta}-1+
        			e^{-p{C}\Delta}\mathsf{P}_i^p+e^{-p{C}\Delta}\sum_{k=1}^{p-1}\binom{p}{k}1_{(U_i> p_i)}\mathsf{P}_i^{p-k}.
        		\end{align*}
        		Given the comments in Remark \ref{rem:11}, one obtains that
        		\begin{align*}
        			\left|\widetilde{\mathbb{E}}_{i-1}\left[\mathsf{P}_i^{p-k_1} 1_{(U_i> p_i)}\mathsf{P}_i^{p-k_2}  \bar{m}_i\right]\right |\leq C\Delta, \quad k_1,k_2\in\{1,...,p-1\}.
        		\end{align*} 
        		All other cases of products have similar bounds. A similar  but longer consideration also applies to the
        		condition for the quadratic  variation of $M(p) $. From here one obtains that \eqref{eq:cond} is satisfied.

        		The rest of the proof follows the one in Lemma \ref{lem:32}.
        	\end{proof}
        	
        	Now, we are ready to present the analysis of the derivative of each term in %
        	\eqref{eq:87a}. As both are treated in the same manner, we will give details of the analysis of the second
        	term leaving the first for the reader.
        	
        	\begin{proof}[Proof of Proposition \ref{prop:30}]We will prove by induction that
        		the second derivative is bounded locally on compacts by $\mathsf{M}e^{%
        			\mathsf{M}(T-t_i)} $ where $\mathsf{M}>0 $ is some universal constant to be fixed as in Lemma \ref{lem:35}. {That is, $ \mathsf{M} $ satisfies that $ \mathsf{M}\geq\bar C_2\vee (17C) $ for any $ C$ as considered in Lemmas \ref{moments} and \ref{lem:41a}. The constant $ 8$ corresponds to the number of terms that we have used to divide our analysis in parts which appear in Lemma \ref{lem:34} including all remainder terms considered as one extra term. }
        		
        		Our objective is to prove that the terms in \eqref{eq:85}$ + $ \eqref{eq:72aa} are
        		bounded by  $ {\mathsf{M}}e^{\mathsf{M}(T-t_{i-1})} $.
        		For this, we will often use Lemma \ref{lem:35}.
        		
        		Note that using Lemmas \ref{moments}, \ref{boundedderivative}, the
        		definition of $f_j $ and the argument in the proof of Lemma \ref{lem:35} (C)
        		by decomposing $f_j=(f_j-f_j(L))$, one obtains the bound: 
        		\begin{align}  \label{eq:gba}
        			\|\bar{g}_{k,j}\|_\infty\leq C\|f_j^{\prime
        			}\|_\infty e^{C(t_{j-1}-t_k)}\Delta\leq \mathsf{C}e^{\mathsf{C}%
        				(t_j-t_k-\Delta)}\Delta.
        		\end{align}
        		This is used without further mention in the estimates below.
        		
        		%               For the first term in \eqref{eq:72aa}, one uses Lemma \ref{boundedderivative}
        		%               and Lemma \ref{lem:35}(C), noting that $\bar{A}^1_j(\bar{h}_j)=O^E_{j-1}(1) $.
        		For the first and second term of \eqref{eq:85} and the second and third term in \eqref{eq:72aa}, one uses Lemmas \ref{boundedderivative}, \ref{moments}, \ref{lem:41a}, the inductive hypothesis and Lemma \ref{lem:35} (A), (B) and (C).

        		For the third term in \eqref{eq:85}  as well as the first and fourth terms in \eqref{eq:72aa} we need an additional delicate argument
        		because $\widehat{B}_k(e_k) $ contains the term $-\frac{b_{k-1}}{\sigma_{k-1}}%
        		\bar{h}_k $ as given in \eqref{eq:tab}. The same problem appears in the above mentioned terms. This problem is treated in the next lemma which will finish the proof of this proposition. 
        	\end{proof}
        	\begin{lemma}
        		\label{lem:36}
        		Using the universal constant $ \mathsf{M} $ we have the following inequality uniformly in $ (i,j,x,n) $
        		\begin{align}
        			\sum_{k=i}^n\left|\widehat{g}_{k}(L) \widetilde{ \mathbb{E}}_{i,x}\left[ \widehat{P}
        			_{i:k-1}^{n}\widehat{B}_k(e_k) \bar{M}^n_{i:k}\right]\right |\leq &\frac{ \mathsf{M}}%
        			8e^{\mathsf{M}(T-t_{i})}, \notag\\
        			\left |\sum_{j=i}^{n} \widetilde{\mathbb{E}}_{i,x}\left[
        			\partial_j f_j \widehat{P} _{i:j-1}^{n}\widehat{A}^1_j({h}_j) \bar{M}^n_{i:j}%
        			\right]\right |\leq &\frac{ \mathsf{M}}%
        			8e^{\mathsf{M}(T-t_{i})}, \notag\\
        			\left|\sum_{k=i+1}^{j-1}\bar{g}_{k,j}(L) \widetilde{ \mathbb{E}}_{i,x}%
        			\left[ \widehat{P} _{i:k-1}^{n}\widehat{B}_k(e_k) \bar{M}^n_{i:k}\right]\right |\leq &
        			\frac{\mathsf{M}}8e^{\mathsf{M}(T-t_{i})}\Delta.\label{eq:88}
        		\end{align}
        	\end{lemma}
        	\begin{remark}
        		\label{rem:36}
        		The difficulty in the proof of the above result resides in the fact that the
        		estimate in Lemma \ref{boundedderivative} only gives $\|\widehat{g}%
        		_k\|_\infty\leq \mathsf{C}e^{\mathsf{C}(T-t_k)} $. In a similar way, \eqref{eq:gba} and the inductive hypothesis does not
        		give the above estimates directly. Instead we will have to first isolate
        		terms that are of small order within $\widehat{B}_k(e_k) $. That is, we will show that it is enough to consider the first term $ -\frac{b_{k-1}}{a_{k-1}}1_{(U_k\leq p_k)} $ in \eqref{eq:tab}. Furthermore, it is enough to consider $  1_{(U_k\leq p_k)} $ as the coefficient is uniformly bounded. At the same time the leading terms in the product $\widehat{P}^n_{i:k-1}$ always start with $1_{(U_k>p_k)} $ (see the formula for $ \widehat{A}^1_k(e_k) $ in \eqref{eq:tab}). From these products one understands that the the important terms in the resulting product are related to the fact that the Euler
        		scheme touches the boundary in the time interval $[t_{k-1},t_k] $. The probability
        		estimate for this behavior will lead to the final result.
        	\end{remark}
        	\begin{proof}
        		In the proof we will use the following random variables ($ \bar{\mathsf{h}}_{r,k} $ has already been introduced in \eqref{eq:tildeha}) for path decompositions: 
        		\begin{align*}
        			\mathsf{h}_{r,k} :=&1_{(U_r\leq p_r,U_{r+1}>p_{r+1},...,U_{k-1}>
        				p_{k-1},U_k\leq p_k)}, \\
        			%\bar{\mathsf{h}}_{r,k}:=&1_{(U_r\leq p_r,U_{r+1}>p_{r+1},...,U_k> p_k)} \\
        			\widehat{\mathsf{h}}_{r,k}:=&1_{(U_r> p_r,...,U_{k-1}>p_{k-1},U_k\leq p_k)},
        			\\
        			\widetilde{e}_i:=&e_i+\frac{b_{i-1}}{\sigma_{i-1} }Z_i+\frac{b^2_{i-1}}{
        				2a_{i-1} }(Z_i^2-\Delta).
        		\end{align*}
        		
        		As explained in Remark \ref{rem:36}, we will only consider the important terms for the required estimates. As the first two inequalities in the statement are proved in the same manner which is used in the third case which is slightly more elaborated, we will only prove \eqref{eq:88}. We start using the following formula whose proof is based on path decompositions and induction:
        		\begin{align}
        			\widehat{P}_{i:k-1}^{n}1_{(U_k\leq p_k)}=&\widehat{P}_{i:k-1}^{n}\mathsf{h}_{k-1,k}+\Delta_{k-1}%
        			\widehat{P}^n_{i:\cdot}\widehat{\mathsf{h}}_{k-1,k}+\widehat{P}_{i:k-2}^{n}\widehat{\mathsf{h}}_{k-1,k}\notag\\
        			=&\sum_{r=i+1}^{k-1}\left(\Delta_r%
        			\widehat{P}^n_{i:\cdot}\widehat{\mathsf{h}}_{r,k}+\widehat{P}^n_{i:r}{%
        				\mathsf{h}}_{r,k}\right)+\widehat{\mathsf{h}}_{i+1,k}.\label{eq:89}
        		\end{align}
        		Note that each of the terms in the sum can be rewritten using the explicit formula for $ \mu_r^1 $ in $ \widehat{A}^1_r(e_r) $ as 
        		\begin{align*}
        			\Delta_r\widehat{P}^n_{i:\cdot}\widehat{\mathsf{h}}_{r,k}=&\widehat{P}%
        			^n_{i:r-1}\left(\partial_{r-1}X_r{e}_r-1 \right)\widehat{\mathsf{h}}%
        			_{r,k} +\varepsilon^6_r,\\
        			\widehat{P}^n_{i:r}{\mathsf{h}}_{r,k}=&\widehat{P}^n_{i:r-1}\left(\left(%
        			\partial_{r-1}X_r\widetilde{e}_r-1\right)+{\bar{\Delta}}{e}_r\frac{%
        				\sigma^{\prime }_{r-1}}{\sigma_{r-1}} (X_{r-1}-L)-\left({\bar{\Delta}}{e}%
        			_r-1\right) \right){\mathsf{h}}_{r,k}+\varepsilon^7_r
        			.
        		\end{align*}
        		Here, $\tilde{\mathbb{E}}_{r-1}[ \varepsilon^\ell_r\bar{m}_r]= O^E_{r-1}(\sqrt{\Delta})$, $ \ell=6,7 $.
        		It is important to note that in the above two equalities all terms are already of order  $O^E _{r-1}(1) $. In fact, they may be reduced to one of the
        		the following types $ \widehat{\mathsf{h}}_{r,k}\alpha_r $ and $\mathsf{h}_{r,k}\beta_r $ where
        		\begin{align*}
        			\alpha_r\in& \left\{1,Z_r,Z^2_r,\Delta,Z_r\Delta\right\},\\\beta_r\in&\left\{Z_r,X_{r-1}-L,Z_r(X_{r-1}-L),(X_{r-1}-L)%
        			\Delta,(X_{r-1}-L)^2,\Delta\right\}.
        		\end{align*}
        		
        		Therefore, the summability in $ r $ is uniform and one needs to prove that the summability in $ k $ is uniform in $ (i,j,x,n) $ for all the above terms. The
        		analysis is similar in each case including the case for $ \widehat{\mathsf{h}}_{i+1,k} $ in \eqref{eq:89}.
        		
        		These terms are analyzed using a combination of order $O^E $-type results (including those in  Lemma \ref{lem:essb} and its proof).  To obtain the summability in $ k  $, one notes that the $\mathcal{F}%
        		_{r} $-conditional expectation $  \widehat{\mathsf{h}}_{r,k}$ and $
        		{\mathsf{h}}_{r,k}$ is the probability that the Euler scheme starting at $%
        		(t_{r},X_{r}) $ hits the boundary for the first time in the interval $%
        		[t_{k-1},t_k] $. The density of this hitting time is bounded by 
        		\begin{align}
        			{ \frac{1}{c(t-t_{r})}e^{- \frac{(X_{r}-L)^2}{%
        						2C(t-t_{r})}}.  }
        			\label{eq:90}
        		\end{align}
        		
        		This bound is obtained by straight convolution of uniform upper Gaussian
        		estimates for the Euler scheme (for the technique, one has to reproduce  line by line the proof of Theorem 2.1 in \cite{LM} using Theorem 3.2 in \cite{FKL}.) and the
        		conditional density of hitting the boundary at some time in the interval $%
        		[t_{k-1},t_k]$ conditioned to $\mathcal{F}_{k-1} $. Here $0<  c\leq C $, are constants that depend on the boundedness hypothesis on the coefficients $ \sigma $ and $ b $ as well as the uniform ellipticity. In order to carry out the estimations, from now on these constants may change value from one line to the next.
        		
        		%               This bound is obtained by using Lamperti transform and Girsanov transformation to finally use an estimate for the
        		%               conditional density of hitting the boundary of the Wiener process at some time in the interval $%
        		%               [t_{k-1},t_k]$ conditioned to starting at $ \mu(X_r) $ where $\mu $ is given in \eqref{eq:lamp}. The upper and lower bounds for $ \sigma$ imply upper and lower bounds for the derivative of $ \mu $ which gives the  inequality.
        		
        		We will provide the main  details without adding cumbersome general lemmas as much as
        		possible and at the same time concentrate in the most elaborate case of possible $ \alpha_r $ which is
        		the case of $Z_r \widehat{\mathsf{h}}_{r,k}$.
        		
        		Using the bound in \eqref{eq:90}, we have 
        		\begin{align}
        			\left|\widetilde{\mathbb{E}}_{r-1}\left[Z_r \widehat{\mathsf{h}}_{r,k}\right]%
        			\right|\leq \int_{t_{k-1}}^{t_k}dt \int_{L}^\infty dy \frac{1}{
        				c(t-t_{r})}e^{- \frac{(y-L)^2}{2C(t-t_{r})}}\frac{|y-X_{r-1}|}{%
        				\sqrt{\Delta}}e^{- \frac{(y-X_{r-1})^2}{2a_{r-1}\Delta}%
        			}.
        			\label{eq:91}
        		\end{align}
        		The above exponentials can be rewritten as the difference of two terms given respectively
        		by 
        		\begin{align*}
        			e^{- \frac{(y-L)^2}{2C(t-t_{r})}}e^{- \frac{(y-X_{r-1})^2}{2a_{r-1}\Delta}
        			}=&e^{-\frac{(y-\mu)^2}{2\widetilde{\sigma}^2}}e^{-%
        				\frac{(X_{r-1}-L)^2}{2S_{r}}}, \\
        			\mu:=&\frac{La_{r-1}\Delta+X_{r-1}C(t-t_r)}{C(t-t_r)+a_{r-1}\Delta} \in [L,X_{r-1}].
        			%               \\
        			%               \mu_2:=&\frac{La_{r-1}\Delta-(X_{r-1}-2L)C(t-t_r)}{C(t-t_r)+a_{r-1}\Delta}\leq L.
        		\end{align*}
        		
        		Here, we let $%
        		\widetilde{\sigma}^2:=\frac{C(t-t_r)a_{r-1}\Delta}{S_r} $, $ {S_r:= C(t-t_r)+a_{r-1}\Delta}$, in order to simplify long expressions. Now one computes the integrals in \eqref{eq:91} by dividing the region of integration into $ [L,X_{i-1}] $ and $ (X_{i-1},\infty) $. This gives as upper bound  
        		
        		\begin{align*}
        			\left|\widetilde{\mathbb{E}}_{r-1}\left[Z_r \widehat{\mathsf{h}}_{r,k}\right]%
        			\right|\leq \int_{t_{k-1}}^{t_k}dt\frac{\Delta}{c(t-t_r)^{1/2}S_r^{3/2}}
        			\left((t-t_r)e^{-%
        				\frac{(X_{r-1}-L)^2}{C(t-t_r)}}+S_r^{1/2}e^{-%
        				\frac{(X_{r-1}-L)^2}{CS_{r}}}\right).
        		\end{align*}
        		Here we have used the fact that $|x|e^{-\frac{x^2}2}\leq e^{-1/2} $. Both terms inside the integral are treated with the same argument although they lead to different integrals. We start explaining how to bound the second term.

        		Using the above bound one finds that the $\mathcal{F}
        		_{i-1}$-conditional expectation of the corresponding term in \eqref{eq:88} is
        		bounded by a constant multiple of 
        		\begin{align*}
        			&\sum_{k=i+1}^{j-1}|\bar{g}%
        			_{k,j}(L)|\sum_{r=i+1}^{k-1}\Delta \int_{t_{k-1}}^{t_k}dt\widetilde{%
        				\mathbb{E}}_{i}\left[ \frac{|\widehat{P}^n_{i:r-1}| e^{-\frac{%
        						(X_{r-1}-L)^2}{2S_r}}\bar{M}^n_{i:r-1}}{
        				(t-t_r)^{1/2}S_r} \right]\\
        			\leq&2e^{-1}\mathsf{C}e^{\mathsf{C}%
        				(t_j-t_i-\Delta)} \sum_{r=i+1}^{j-2}\Delta^2
        			\int_{t_{r}}^{t_{j-1}} dt \widetilde{\mathbb{E}}_{i}\left[\frac{e^{-\mathsf{C}%
        					(t_{r-1}-t_i)}|\widehat{P}%
        				^n_{i:r-1}| F(\frac{X_{r-1}-L}{\sqrt{S_r}})\bar{M}^n_{i:r-1}%
        			}{(t-t_r)^{1/2}S_r} \right]\\
        			\leq &2e^{-1}\mathsf{C}e^{\mathsf{C}%
        				(t_j-t_i-\Delta)}  \int_{0}^{j-i-2}{
        				\Delta^{3/2}}\\&\times \frac{\widetilde{\mathbb{E}}_{i}\left[\max_{r=i+1,...j-2}e^{-\mathsf{C}%
        					(t_{r-1}-t_i)}|\widehat{P}^n_{i:r-1}| \sum_{r=i+1}^{j-2}F(\frac{X_{r-1}-L}{\sqrt{C\Delta}\sqrt{u+1}})%
        				\bar{M}^n_{i:r-1}\right]}{u^{1/2}(Cu+c)}.
        		\end{align*}
        		Here $F $ is a function satisfying \eqref{eq:Fbu'} and we have also used  \eqref{eq:gba}, exchange of summations and the
        		change of variables $t-t_r=\Delta u $.
        		
        		The boundedness of the above sum follows from Lemmas \ref{lem:41a}, \ref{lem:essb} (using $%
        		\Delta(u+1) \leq T$ instead of $\Delta $ which we can assume smaller than $ 1 $ without loss of generality), direct integration
        		(i.e. $\int_{0}^{j-i-2}\frac{du}{u^{1/2}(Cu+c)^{3/2}}\leq  C$) .
        		
        		As stated previously all other cases follow a similar structure in the
        		arguments. Note that in the case of $ \mathsf{h}_{r,k}\beta_r $, the situation is slightly different but easier to deal with because the density in \eqref{eq:91} is \begin{align*}
        			\frac{1}{%
        				\sqrt{2\pi\Delta}\sigma_{r-1}}e^{- \frac{(y+X_{r-1}-2L)^2}{2a_{r-1}\Delta}%
        			}
        		\end{align*} which will give a value $ \mu= \frac{La_{r-1}\Delta-(X_{r-1}-2L)C(t-t_r)}{S_r}\leq L$ which is farther from $ X_{r-1} $ in comparison with the case treated above.
        	\end{proof}
        	
        	\subsection{The convergence of the first and last boundary hitting times}
        	
        	Recall that $X^{n}$ is the sequence of Euler approximation processes of $X$ which start 
        	from $x\in D$. Under $ {\mathbb{Q}}^n $, $ (X^n,B^n) $ has the same law at time partition points as 
        	$(\mathcal{X}^{n,x},\Lambda^{L})$  the solution of the reflected
        	Euler scheme for $x\geq L$ as defined in \eqref{reflectedy}  of Section \ref{app:res} (see Remark \ref{rem:3}).

        	Let $(Y,B)$ be the solution of the reflected equation \eqref{reflected} in
        	the domain $[L,\infty )$, i.e., 
        	\begin{equation*}
        		Y_{t}=x+\int_{0}^{t}\sigma (Y_{s})dW_{s}+B_{t}.
        	\end{equation*}
        	%       and  That is, $(%
        	%       \mathcal{X}^{n,x},\Lambda ^{L})$ is solution of the system of equations 
        	%       \begin{align*}
        	%       \mathcal{X}_{t}^{n,x}=& \ x+\int_{0}^{t}\sigma \left( \mathcal{X}_{\tau
        	%               \left( s\right) }^{n,x}\right) dW_{s}+\Lambda _{t}^{L}, \ \ \ \mathcal{X}%
        	%       _{t}^{n,x} \geq L ,  \notag \\
        	%       \Lambda _{t}^{L}=& \int_{0}^{t}1_{( \mathcal{X}_{s}^{n,x}=L) }d\Lambda
        	%       _{s}^{L}.  \notag
        	%       \end{align*}
        	From Theorem \ref{bigT}, $(\mathcal{X}^{n,x},\Lambda^{L})$ converges in
        	distribution $(Y,B)$ uniformly on the path space on any bounded interval $%
        	[0,T]$. By using the Skorohod representation
        	theorem, we can assume, without loss of generality, that  $(\mathcal{X}^{n},\Lambda
        	^{L}) $ converges almost surely to $(Y,B)$ uniformly on the path space on
        	any bounded interval $[0,T]$. In particular, we can look at  the Skorohod representation used in the proof of Theorem \ref{bigT}.
        	Consequently, we can assume that $(\mathcal{X}^{n},\Lambda^{L})$ converges almost surely to $(Y,B)$. The following definitions and subsequent proposition (and proof) are given in terms of the representation of the processes defined using the Skorohod representation introduced in Theorem \ref{bigT} for which the almost sure convergence holds.
        	
        	Define $\tau_t$, $\tau _{t,n}$ to be the first times in $ [0,t] $ that the processes $Y$,
        	respectively $\mathcal{X}^{n,x}$ hit the boundary $L$. That is, 
        	
        	\begin{eqnarray*}
        		\tau_t&:=&t\wedge\inf \left\{ s\geq 0,\ \ Y_{s}=L\right\},\ \ \ \ \tau
        		_{t,n}:=t\wedge\inf \left\{ s\geq 0,~\mathcal{X}^{n,x}_s=L\right\},\\
        		\bar{\tau}_t^n&:=& t\wedge\inf \{t_{i}\leq t;\text{ there exists }s\in[t_{i-1},t_i],\ \mathcal{X}^{n,x}_s=L \}.
        	\end{eqnarray*}

        	Similarly define $\rho_t$, $\rho_{t,n}$ to be the last times before $t$ the
        	processes $Y$, respectively $\mathcal{X}^{n,x}$ hit the boundary $L$. i.e., 
        	\begin{eqnarray*}
        		\rho_t&:=&\sup \left\{ s\le t,\ \ Y_{s}=L \right\}, 
        		%,\ \ \mathrm{if}\ \ \left\{s\le t,\ \ Y_{s}=L \right\}\ne \emptyset,\ %\ \mathrm{otherwise}\ \ \rho_t:=t,
        		\\
        		\rho_{t,n}&:=&\sup \left\{ s\le t,\ \ \mathcal{X}^{n,x}_{s}=L \right\},\\ 
        		%,\ \ 
        		%        \mathrm{if}\ \ 
        		%\left\{ s\le t,\ \ \mathcal{X}^{n,x}_s=L \right\}\ne
        		%       \emptyset,\ \ \mathrm{otherwise}\ \ \rho_{t,n}:=t,\\
        		\bar{\rho}_t^{n}&:=&\sup\{t_{i}\leq t;\text{ there exists }s\in[t_{i-1},t_i],\ \mathcal{X}^{n,x}_s=L \} .
        		%\\
        		%&&\ \ \ \ \ \ \ \ \ \ \ \ \ \ \ \ \ \ \ \ \ \ \ \ \ \ \ \ \ \ \ \ \ \ 
        		%       \mathrm{if}\ \ 
        		%\left\{ s\le t,\ \ \mathcal{X}^{n,x}_s=L \right\}\ne
        		%       \emptyset,\ \ \mathrm{otherwise}\ \ \bar{\rho}_t^{n}:=t
        	\end{eqnarray*} 
        	
        	Recall that $\bar{\tau}_T^n=\inf\{t_i;U_i<p_i \} $ and $\bar{\rho}_T^{n}=\sup\{t_i;U_i<p_i\} $. 
        	
        	%	\footnote{As usual, we denote by $[q]$ the integer part of $q\in \mathbb{R}$.}
        	
        	Define $u $ and $v^{n,x} $ as the martingale parts of the processes $Y $ and 
        	$\mathcal{X}^{n,x} $ respectively. Note that we have 
        	\begin{equation*}
        		\tau_t=t\wedge\inf \left\{ s\geq 0,\ \ u_{s}=L\right\}\ \ \ \ \tau
        		_{t,n}=t\wedge\inf \left\{ s\geq 0,~v^{n,x}_s=L\right\}
        	\end{equation*}
        	as the processes $B$ and $\Lambda^L$ are null before the processes hit the
        	boundary. As $u$ and $\upsilon^{n,x}$ are time-changed Brownian motions
        	(starting from $x$) observe that they ``leave'' the domain $[L,\infty)$ as soon
        	as they hit the boundary. That is, $ L $
        	is a regular boundary point.

        	We have the following:

        	\begin{proposition}
        		\label{rdthe} For any fixed $ t\in [0,T] $,  the random times $(\bar{\tau}_t^n,\bar{\rho}_t^{n})$ converge, almost surely, to  $(\tau_{t},\rho_t)$.
        	\end{proposition}
        	
        	%\begin{proof} 
        	%Observe that also the martingale parts of the processes $Y$, respectively $\mathcal{X}^{n,x}$ converge uniformly on the interval $[0,t]$ and, we can define $\tau$ respectively $\tau _{n}$ to be the first times when the martingale parts are hit the boundary $L$ (the processes $B$ and $\Lambda^L$ are null before the processes hit the boundary). Let us denote these processes by $u$ respectively $\upsilon^{n,x}$. As these processes are time-changed Brownian motions (starting from $x$) observe that they leave the domain $[l,\infty)$ as soon as they hit the boundary. Of course, this is not the true for the reflected processes. 
        	
        	\begin{proof}
        		First, we note that in the case that $ \rho_t=-\infty  $ or     $ \inf \left\{ s\geq 0,\ \ Y_{s}^{1}=L\right\}=\infty $ then the result follows trivially. 
        		
        		Next, observe that $|\bar{\tau}_t^n-\tau_{t,n}|\le 1/n$ and $|\bar{\rho}_t^{n}-\rho_{t,n}|\le 1/n$
        		so it is enough to prove that $(\tau_{t,n},\rho_{t,n})$ converge almost surely to  $(\tau_{t},\rho_t)$. Since $(\mathcal{X}^{n,x},\Lambda ^{L})$ converge almost surely to $(Y,B)$,
        		it follows that also the martingale parts $\upsilon^{n,x}$ (we use the same
        		notation as that of the original martingale parts) converge almost surely to 
        		$u$ uniformly on the path space on any bounded interval $[0,T]$. Just as the
        		original processes, $u$ and $\upsilon^{n,x}$ are time-changed Brownian
        		motions (starting from $x$) and therefore they leave the domain $[L,\infty)$
        		as soon as they hit the boundary. We will divide the proof in three cases.
        		
        		{\it Case 1:} On the set of paths  $A:=\{u_{s}\in D; s\in [0,t]\}$ we argue that by the uniform convergence of $%
        		v^{n,x}$ to $u$ we can deduce immediately that $\lim_{n\rightarrow \infty
        		}\tau _{t,n}=t=\tau_t$. 
        		
        		{\it Case 2:} Now we consider the complement set $ A^c $.
        		In this case $u_{\tau_t}=L$. 
        		
        		{\it Case 2a:}      Note that on the set $B:=\{u_s>L, s\in (0,t)\}\cap \{u_t=L\} $ we have $%
        		\tau_t=t $. Consider $L(\epsilon):=\min_{s\leq t-\epsilon}u_s>L $ and let $%
        		n_0\in\mathbb{N} $ so that $\max_{s\in [0,t]}|u_s-v^{n,x}_s|\leq \frac{%
        			L(\epsilon)-L}2$ for all $n\geq n_0 $. Then one has that $%
        		|\tau_t-\tau_{t,n}|\leq \epsilon $ for all $n\geq n_0 $ as $\tau_{t,n}\in
        		[t-\epsilon,t]. $ From here the convergence follows.
        		
        		{\it Case 2b:} If $ \tau_t=0 $ the convergence statement follows trivially. 
        		
        		{\it Case 2c: }Therefore, note that for arbitrary $m~$such~that~$\frac{1}{m}<\tau_t $%
        		, we have that, almost surely, 
        		\begin{equation*}
        			\lim_{n\rightarrow \infty }\min_{s\in \left[ 0,\tau_t -\frac{1}{m}\right]
        			}\upsilon^{n,x}_{s}=\min_{s\in \left[ 0,\tau_t -\frac{1}{m}\right] }u_{s} >L.
        		\end{equation*}%
        		Therefore, for $n$ sufficiently large it follows that $\min_{s\in \left[
        			0,\tau_t -\frac{1}{m}\right] }\upsilon^{n,x}_{s} >L$ a.s. and hence $%
        		\tau _{t,n} >\tau _{t} -\frac{1}{m}$ a.s., which in turn leads to
        		\begin{equation}
        			\tau_{t} \leq \liminf_{n\rightarrow \infty }\tau _{t,n} .  \label{inftau}
        		\end{equation}

        		Next, assume without loss of generality that $\tau_t<t$ and let $m$ be such that 
        		$\frac{1}{m}<t-\tau_t $. As $u $ is a continuous martingale with strictly
        		increasing quadratic variation, we have that it will take values smaller
        		than $L $ on the interval $\left[ \tau_{t} ,\tau_{t} +\frac{1}{m}\right] $
        		almost surely (for a complete proof one uses the Markov property of Brownian
        		motion and the fact that starting from $L $ the Brownian motion oscillates
        		above and below this value with probability one). Hence 
        		\begin{equation*}
        			\lim_{n\rightarrow \infty }\min_{s\in \left[ \tau_{t} ,\tau_{t} +\frac{1}{m}%
        				\right] }\upsilon^{n,x}_{s}=\min_{s\in \left[ \tau_{t} ,\tau_{t} +\frac{1}{m}
        				\right] }u_{s} <L,
        		\end{equation*}%
        		therefore, for $n$ sufficiently large, it follows that $\min_{s\in \left[
        			\tau_{t} ,\tau _{t}+\frac{1}{m}\right] }\upsilon^{n,x}_{s} <L$ a.s.
        		and hence $\tau _{t,n} <\tau_t +\frac{1}{m}$ a.s., which in turn implies
        		\begin{equation}
        			\lim \sup_{n\rightarrow \infty }\tau _{t,n} \leq \tau _{t}  \label{suptau}
        		\end{equation}%
        		almost surely. The claim follows from (\ref{inftau}) and (\ref{suptau}).
        		
        		Note that the above implies, in particular that $\tau_t $ is continuous on
        		the set of paths such that
        		\begin{align*}
        			E:=\{\exists s\in (0,t), u_s=L; \forall \epsilon>0, \exists s_0\in
        			(s,s+\epsilon), u_{s_0}<L\}.
        		\end{align*}
        		This explicit description of the continuity set for $\tau_t$ will be used in
        		the following step of the proof.
        		
        		We turn now to the convergence $\lim_{n\rightarrow \infty }\rho_{t,n}=
        		\rho_t $. We are prevented to use a direct
        		approach as $\rho _{t}$ is no longer a stopping time. So will revert time to
        		be able to apply the above argument for $ \tau_t $. The first step is to transform the
        		martingale term in \eqref{reflected} into a Brownian motion using a Lamperti
        		transformation. Then we will change the underlying measure so that the new
        		process becomes a reflected Brownian motion and then characterize $\rho _{t}$
        		in terms of the increments of the martingale component of $Y$. We give next
        		the details of the argument: Define first $\mu :[0,\infty )\rightarrow
        		\lbrack 0,\infty )$, 
        		\begin{equation}
        			\mu (x)=\int_{0}^{x}\frac{1}{\sigma \left( z\right) }dz.
        			\label{eq:lamp}
        		\end{equation}%
        		By applying It\^{o}'s formula to \eqref{reflected}, we deduce that for $ u\geq 0 $
        		\begin{align*}
        			%                        \mu \left( Y_{u}\right) =&\mu \left( x\right) +W_{u}+\int_{0}^{u}\frac{1}{%
        			%                        \sigma \left( Y_{s}\right) }dB_{s}-\int_{0}^{u}\frac{\sigma ^{\prime }\left(
        			%                        Y_{s}\right) }{2}ds,~~~u\geq 0. \\
        			\mu \left( Y_{u}\right)      =&\mu \left( x\right) +\widehat{W}_{u}+\widehat{B}_{u},\\
        			%                \mu \left( \mathcal{X}_{u}^{n,x}\right) =&\ \mu \left( x\right)
        			%                +\int_{0}^{u}\frac{\sigma \left( \mathcal{X}_{\tau \left( s\right)
        			%                        }^{n,x}\right) }{\sigma \left( \mathcal{X}_{s}^{n,x}\right) }%
        			%                dW_{s}+\int_{0}^{u}\frac{1}{\sigma \left( \mathcal{X}_{s}^{n,x}\right) }%
        			%                d\Lambda _{t}^{L}-\int_{0}^{u}\frac{\sigma ^{\prime }\left( \mathcal{X}%
        			%                        _{s}^{n,x}\right) }{2a \left( \mathcal{X}_{s}^{n,x}\right) }a
        			%                \left( \mathcal{X}_{\tau \left( s\right) }^{n,x}\right) ds \\
        			\mu \left( \mathcal{X}_{u}^{n,x}\right)       =&\mu \left( x\right) +\widehat{W}_{u}^{n}+\widehat{B}_{u}^{n}
        		\end{align*}
        		where for $ u\geq 0 $ and $ \eta(s):=\max\{t_i;t_i\leq s\} $, we have
        		\begin{align}
        			\label{eq:w1}   
        			\begin{split}                  \widehat{W}_{u} :=&W_{u}-\int_{0}^{u}\frac{\sigma ^{\prime }\left( Y_{s}\right) 
        				}{2}ds,\\
        				\widehat{B}_{u} :=&\int_{0}^{u}\frac{1}{\sigma \left( Y_{s}\right) }%
        				dB_{s},\\
        				\widehat{W}_{u}^{n} :=&\int_{0}^{u}\frac{\sigma \left( \mathcal{X}_{\eta \left(
        						s\right) }^{n,x}\right) }{\sigma \left( \mathcal{X}_{s}^{n,x}\right) }%
        				d{W}_{s}-\int_{0}^{u}\frac{\sigma ^{\prime }\left( \mathcal{X}%
        					_{s}^{n,x}\right) }{2a \left( \mathcal{X}_{s}^{n,x}\right) }a
        				\left( \mathcal{X}_{\eta \left( s\right) }^{n,x}\right) ds, \\
        				\widehat{B}_{u}^{n} :=&\int_{0}^{u}\frac{1}{\sigma \left( \mathcal{X}%
        					_{s}^{n,x}\right) }d\Lambda _{t}^{L}.
        			\end{split}
        		\end{align}
        		
        		Observe that, since $(\mathcal{X}^{n,x},\Lambda ^{L})$ converges almost
        		surely to $(Y,B)$ uniformly on the path space on any bounded interval $[0,T]$%
        		, also $\left( \widehat{W}^{n},\widehat{B}^{n}\right) $ converges to $\left( \widehat{W},%
        		\widehat{B}\right) $ uniformly on the (corresponding) path space on any bounded
        		interval $[0,T\dot{]}$.  We will use the notation $%
        		\widehat Y _s=\mu (Y_s)$ for $s\in [0,t]$. Observe that for $s\in \left[ \rho
        		_{t},t\right] $,%
        		\begin{equation*}
        			\mu (Y_t)-\mu (Y_s)=\widehat{W}_{t}-\widehat{W}_{s}
        		\end{equation*}%
        		as the local time term remains constant after $\rho _{t}$. So 
        		\begin{align*}
        			\rho _{t} :=&\sup \left\{ s\leq t,\ \mu (Y_s)=\mu \left( L\right)
        			\right\} \\
        			=&\sup \left\{ s\leq t,\  \mu (Y_t)-\left( \widehat{W}_{t}-\widehat{W}%
        			_{s}\right) =\mu \left( L\right) \right\} \\
        			=&t-t\wedge \inf \left\{ s\geq 0,\ \ \mu (Y_t)-\bar{W}_{s}=\mu \left(
        			L\right) \right\} ,
        		\end{align*}
        		where\ $\bar{W}$ is defined as 
        		\begin{equation*}
        			\bar{W}_{s}:=\widehat{W}_{t}-\widehat{W}_{t-s},~~~s\in \left[ 0,t\right].
        		\end{equation*}%
        		Similarly,
        		\begin{align*}
        			\rho _{t,n} =&\sup \left\{ s\leq t,\ \ \left( \widehat{W}_{t}^{n}-\widehat{W}%
        			_{s}^{n}\right) =\mu \left( \mathcal{X}_{t}^{n,x}\right) -\mu \left(
        			L\right) \right\} \\
        			=&t-t\wedge\inf \left\{ s\geq 0,\ \ \mu \left( \mathcal{X}_{t}^{n,x}\right)
        			-\bar{W}_{s}^{n}=\mu \left( L\right) \right\}
        		\end{align*}
        		and\ $\bar{W}^{n}$ is defined as 
        		\begin{equation*}
        			\bar{W}_{s}^{n}:=\widehat{W}_{t}^{n}-\widehat{W}_{t-s}^{n},~~~s\in \left[ 0,t\right]
        			.
        		\end{equation*}%
        		Since $\mu \left( \mathcal{X}_{t}^{n,x}\right) -\bar{W}_{s}^{n}$ converges
        		almost surely to $\mu (Y_{t}) -\bar{W}_{s}$ uniformly on the path space
        		corresponding to the interval $[0,t]$, the result will follow if we show
        		that $t-\rho _{t}$ is a stopping time of the type treated in the first part
        		of this proof under the appropriate reversed filtration. Define $u_s:= \mu (Y_t)-\bar{W}_{s}$ and $v^{n,x}_s:=\mu \left( \mathcal{X}_{t}^{n,x}\right) -%
        		\bar{W}_{s}^{n} $ and use $\mu(L) $ instead of $L $.
        		
        		Then, we proceed to the proof of convergence using a similar number of cases as for $ \tau_t $. That is, as before we have that on the set $\widetilde{A}:=\{u_{s}\in D;s\in [0,t]\}$ by the
        		uniform convergence of $v^{n,x}$ to $u$ that $\lim_{n\rightarrow \infty
        		}\rho _{t,n}=t=\rho_t$. So
        		we only need to prove the statement for the complement set of $\widetilde A$. 
        		Consider first the set $\widetilde{B}:=\{u_s\in D, s\in [0,t), u_t=L\} $. On
        		this set, $\rho_t=t $. For any $\epsilon>0$, consider $L(\epsilon):=\min_{s%
        			\leq t-\epsilon}u_s>L $ and let $n_0\in\mathbb{N} $ so that $\max_{s\in
        			[0,t]}|u_s-{v}^{n,x}_s|\leq \frac{L(\epsilon)-L}2$ for all $n\geq n_0 $.
        		Then one has that $|\rho_t-\rho_{t,n}|\leq \epsilon $ for all $n\geq n_0 $
        		as $\rho_{t,n}\in [t-\epsilon,t] $.
        		
        		Finally we prove the statement in the case when $\rho_t<t$ (for $ \tau_t $ they were divided into Cases 2b and 2c). As in the proof for $\tau_t$,
        		we have that $\rho_t $ is continuous on the set 
        		\begin{align*}
        			\widetilde{E}:=\{\exists s\in [0,t), u_s=\mu(L); \forall \epsilon>0, \exists
        			s_0\in (s,s+\epsilon), u_{s_0}<\mu(L)\}.
        		\end{align*}
        		
        		The end of the proof consists in proving that $\mathbb{P}(\widetilde{A}\cup%
        		\widetilde{B}\cup\widetilde{E})=1 $. In the previous case this property is
        		straightforward as we know that $u$ is a square integrable martingale with a
        		strictly increasing quadratic variation and, therefore, a time-changed
        		Brownian motion. As a result, it will have positive and negative oscillations
        		almost surely after it hits the boundary.
        		
        		Let $\mathbb{\widehat{P}}$ be a probability measure absolutely continuous w.r.t. 
        		$\mathbb{P}$ such that 
        		\begin{equation*}
        			\left. \frac{d\mathbb{\widehat{P}}}{d\mathbb{P}}\right\vert _{\mathcal{F}%
        				_{u}}=\exp \left( \int_{0}^{u}\frac{\sigma ^{\prime }\left( Y_{s}\right) }{%
        				2 }dW_{s}-\frac{1}{2}\int_{0}^{u}\left( \frac{%
        				\sigma ^{\prime }\left( Y_{s}\right) }{2}%
        			\right) ^{2}ds\right) .
        		\end{equation*}%
        		By Girsanov's theorem, we get that, under this measure, the process $\widehat{W}$ defined in \eqref{eq:w1} is a
        		Brownian motion. Moreover, under $\mathbb{\widehat{P}}$, the reversed process has
        		a decomposition of the form 
        		\begin{equation*}
        			\bar{W}_{s}=Z_{s}^{1}+Z_{s}^{2},~~~~s\in \left[ 0,t\right]
        		\end{equation*}%
        		where $Z^{1}$ is a Brownian motion, adapted with respect to the the
        		filtration 
        		\begin{equation*}
        			\mathcal{H}_{s}=\sigma \left( Y_{t}\cup \left\{ \widehat{W}_{t}-\widehat{W}%
        			_{t-u}\right\} ,u\in \left[ 0,s\right] \right)
        		\end{equation*}%
        		and $Z^{2}$ is a finite variation process (for more details, see Section VI.4 in \cite{protter}). By using yet another Girsanov
        		transformation, one shows that there exist a probability measure $\bar{\mathbb{P}}$
        		absolutely continuous with respect to $\mathbb{\widehat{P}}$ under which $\bar{W}
        		$ becomes a Brownian motion. Therefore by using the fact that Brownian
        		motion will have to take values smaller than $\mu(L) $ after it hits the
        		boundary the first time, one obtains that $\mathbb{P}(\widetilde{A}^c\cap\widetilde{B%
        		}^c\cap\widetilde{E}^c)=0$ from which the result follows.
        	\end{proof}

        \end{document}